\documentclass[12pt]{amsart}

\usepackage{amsfonts, amstext, amsmath, amsthm, amscd, amssymb, enumitem, url}
\usepackage{tikz-cd}
\usetikzlibrary{shapes.geometric, arrows}
\usepackage{xcolor}
\usepackage{svg}
\usepackage{pdfpages}
\usepackage{pdflscape}
\usepackage{spverbatim}

\usepackage{microtype}
\usepackage[parfill]{parskip}

\usepackage[colorlinks,citecolor=cyan,linkcolor=magenta]{hyperref}
\usepackage[nameinlink]{cleveref}

\usepackage{array}
\newcolumntype{C}[1]{>{\centering\let\newline\\\arraybackslash\hspace{0pt}}m{#1}}
\usepackage{multicol}
\usepackage{multirow}
\usepackage{makecell}
\usepackage{booktabs}

\newtheorem{thm}{Theorem}[section]
\newtheorem{cor}[thm]{Corollary}
\newtheorem{lemma}[thm]{Lemma}
\newtheorem{prop}[thm]{Proposition}
\newtheorem{claim}[thm]{Claim}
\newtheorem{conj}[thm]{Conjecture}
\newtheorem{introconstr}[thm]{Construction}

\newtheorem*{thm:structuralstability}{\Cref{thm:structuralstability}}
\newtheorem*{thm:horsuralmostequiv}{\Cref{thm:horsuralmostequiv}}

\theoremstyle{definition}
\newtheorem{defn}[thm]{Definition}
\newtheorem{rmk}[thm]{Remark}
\newtheorem{constr}[thm]{Construction}
\newtheorem{eg}[thm]{Example}
\newtheorem{noneg}[thm]{Non-example}
\newtheorem{quest}[thm]{Question}
\newtheorem{prob}[thm]{Problem}

\numberwithin{equation}{section}

\usepackage{todonotes}

\renewcommand{\epsilon}{\varepsilon}

\usepackage{stmaryrd}
\newcommand{\cut}{\!\bbslash\!}

\newcommand{\Diffeo}{\mathrm{Diffeo}}

\newcommand{\sing}{\mathrm{sing}}
\newcommand{\slope}{\mathrm{slope}}
\newcommand{\width}{\mathrm{width}}

\begin{document}

\title[Horizontal Goodman surgery and almost equivalence]{Horizontal Goodman surgery and almost equivalence of pseudo-Anosov flows}

\author{Chi Cheuk Tsang}
\address{Département de mathématiques \\
Université du Québec à Montréal \\
201 President Kennedy Avenue \\
Montréal, QC, Canada H2X 3Y7}
\email{tsang.chi\_cheuk@uqam.ca}

\maketitle

\begin{abstract}
We provide an exposition of a `horizontal' generalization of Goodman's surgery operation on (pseudo-)Anosov flows. 
This operation is performed by cutting along a specific kind of annulus that is transverse to the flow and regluing with a Dehn twist of the appropriate sign.
We then show that performing horizontal Goodman surgery on a transitive pseudo-Anosov flow yields an almost equivalent flow, i.e. the original flow and the surgered flow are orbit equivalent after drilling out a finite collection of closed orbits.
We obtain some almost equivalence results by applying this theorem on examples of the surgery operation.
Along the way, we also show a structural stability result for pseudo-Anosov flows.
\end{abstract}


\section{Introduction} \label{sec:intro}

An \textbf{Anosov flow} on a closed oriented $3$-manifold is a smooth flow that contracts along a \textbf{stable} line bundle and expands along an \textbf{unstable} line bundle. These flows were introduced by Anosov in the 1960s. Over the years, it has been shown that they satisfy nice dynamic properties and have connections with various aspects of low-dimensional topology. These days, Anosov flows continues to be an area of active research both in the context of dynamical systems and topology.

In the early days of the subject, not much was known about the prevalence of Anosov flows. Back then, it was a possibility that every Anosov flow is, up to finite cover, a suspension flow of an Anosov map on the torus or the geodesic flow of a negatively curved surface.

This situation changed with the appearance of \cite{Goo83}. In that paper, Goodman showed that given a closed orbit $\gamma$ of an Anosov flow on a closed oriented $3$-manifold $M$, one can construct Anosov flows on $3$-manifolds obtained by performing Dehn surgery on $M$ along $\gamma$. Using this operation, Goodman showed in particular that there exist Anosov flows on many hyperbolic $3$-manifolds.

Together with techniques introduced in later work such as \cite{Fri83}, \cite{BF15}, \cite{BBY17}, we now know of many examples of Anosov flows, as well as many examples of $3$-manifolds that admit Anosov flows.

In this paper, we wish to revisit Goodman's ideas. We quickly summarize Goodman's construction: Given a closed orbit $\gamma$ of an Anosov flow $\phi^t$, one can find an annulus $A$ transverse to the flow and running parallel to $\gamma$. One can then construct a new flow by cutting along $A$ and regluing with a Dehn twist. Here the induced stable and unstable foliations on $A$ play a key role --- the form of these allows one to choose a regluing map that compounds the dynamics of the original flow, so that the reglued flow is also Anosov.

It is an observation of many experts that, on the other hand, the fact that $A$ runs parallel to a closed orbit plays no role in the construction at all. In other words, as long as one can find an annulus $A$ transverse to the flow and interacting with the stable and unstable foliations in the same way as the annuli in Goodman's setup, then one can cut along this $A$ and reglue it to get an Anosov flow as well.

Over time, this became more or less a folklore construction. The first goal of this paper is to fill this gap in the literature by providing a careful exposition.

Fix an Anosov flow on a closed oriented $3$-manifold $M$. We say that a curve $c$ in $M$ is \textbf{positive/negative} if the slope of its tangent field $Tc$ is everywhere positive/negative when projected onto the plane in $TM$ spanned by the stable and unstable line bundles, respectively. Here we consider the stable/unstable line bundle to be the vertical/horizontal direction, respectively, and use the orientation on $M$ to make sense of the sign of the slope.

We say that a positive curve $c$ has a \textbf{steady} tangent field if for every $x, y \in c$ and $t>0$ such that $y=\phi^t(x)$, the slope of $Tc|_y$ is greater than that of $d\phi^t(Tc|_x)$. In words, this means that whenever $c$ passes over itself, then the overpass is of a greater slope than the underpass. See \Cref{fig:steadycurve} left. The definition of a negative curve having steady tangent field is symmetric, see \Cref{fig:steadycurve} right.

We say that a curve $c$ is a \textbf{positive/negative horizontal surgery curve} if it is positive/negative and it has a steady tangent field, respectively.

\begin{introconstr} \label{constr:introhorsur}
Let $\phi^t$ be a (pseudo-)Anosov flow on a closed oriented $3$-manifold $M$. Let $c$ be a positive/negative horizontal surgery curve.

Then for every positive/negative integer $n$, respectively, there exists an annulus $A$ transverse to the flow and containing $c$, and a Dehn twist map $\sigma: A \to A$ with coefficient $n$, such that cutting $M$ along $A$ and regluing with $\sigma$ gives a (pseudo-)Anosov flow.

Moreover, the orbit equivalence class of the reglued (pseudo-)Anosov flow only depends on $c$ and $n$. As such, we denote the reglued flow by $\phi^t_{\frac{1}{n}}(c)$.
\end{introconstr}

The notation $\phi^t_{\frac{1}{n}}(c)$ is inspired from the fact that this flow is defined on the $3$-manifold obtained by performing Dehn surgery on $M$ along $c$ with coefficient $\frac{1}{n}$ (with respect to the natural choice of basis), for which it is common to denote by $M_{\frac{1}{n}}(c)$.

We refer to the operation of obtaining $\phi^t_{\frac{1}{n}}(c)$ from $\phi^t$ as performing \textbf{horizontal Goodman surgery on $\phi^t$ along $c$ with coefficient $\frac{1}{n}$}. We insert the word `horizontal' here for two reasons:
\begin{enumerate}
    \item to emphasize that the annuli we cut along are transverse to the flow (thought of as going vertically upwards) and may or may not run parallel to any closed orbit, and
    \item in anticipation for connections between this operation and the horizontal surgery operation on veering triangulations (defined in \cite{Tsa22a}).
\end{enumerate}
We elaborate on the second point in \Cref{subsec:vt}.

The reader will notice that we have stated \Cref{constr:introhorsur} for \emph{pseudo-}Anosov flows as well. These are essentially Anosov flows with finitely many singular orbits. We remark more on this in \Cref{subsec:intropa}. For now, we simply note that we do not allow horizontal surgery curves to intersect the singular orbits.

The second goal of this paper is an original result that states that applying horizontal Goodman surgery on a transitive (pseudo-)Anosov flow yields an almost equivalent flow. Here we recall the definition of almost equivalence before stating the theorem more formally.

\begin{defn} \label{defn:almostequiv}
Let $\phi^t_i$ be a pseudo-Anosov flow on a closed oriented $3$-manifold $M_i$ for $i=1,2$. Suppose there exists a finite collection of closed orbits $\mathcal{C}_i$ of $\phi^t_i$ and an orbit equivalence between $\phi^t_1$ restricted to $M_1 \backslash \mathcal{C}_1$ and $\phi^t_2$ restricted to $M_2 \backslash \mathcal{C}_2$. Then we say that $\phi^t_1$ and $\phi^t_2$ are \textbf{almost equivalent}. 
\end{defn}

\begin{thm} \label{thm:horsuralmostequiv}
Let $\phi^t$ be a transitive pseudo-Anosov flow on a closed oriented $3$-manifold $M$. Let $c$ be a positive/negative horizontal surgery curve for $\phi^t$. Then for every positive/negative integer $n$, respectively, the flow $\phi^t_{\frac{1}{n}}(c)$ is almost equivalent to $\phi^t$.
\end{thm}

Our motivation for \Cref{thm:horsuralmostequiv} comes from the following conjecture of Fried and Ghys.

\begin{conj}[Fried, Ghys] \label{conj:Ghys}
Any two transitive Anosov flows with orientable stable and unstable line bundles are almost equivalent.
\end{conj}

In a planned sequel \cite{Tsa24}, we will use \Cref{thm:horsuralmostequiv} to obtain partial progress towards \Cref{conj:Ghys}, showing that certain families of transitive Anosov flows are almost equivalent to suspension Anosov flows.

We will also provide lots of examples in this paper. In \Cref{subsec:curveeg}, we discuss examples of horizontal surgery curves. In \Cref{subsec:horsureg}, we apply \Cref{constr:introhorsur} to a subset of these examples to construct various (pseudo-)Anosov flows. One can then apply \Cref{thm:horsuralmostequiv} to obtain almost equivalence results concerning these flows. We give a few sample corollaries below. We refer to \Cref{subsec:curveeg} for the definitions of `straight curve', `scalloped transverse torus', and `Reebless non-scalloped transverse torus'.

\begin{cor} \label{cor:intropAmap}
Let $S$ be a closed oriented surface and let $f:S \to S$ be an orientation preserving pseudo-Anosov map. Let $c \subset S$ be a straight curve with positive/negative slope. Then for every positive/negative integer $n$, respectively, $f \tau^n_c$ is a pseudo-Anosov map.

Moreover, the suspension flow of $f \tau^n_c$ is almost equivalent to the suspension flow of $f$.
\end{cor}

\begin{cor} \label{cor:introscalloptorus}
Let $\phi^t$ be a pseudo-Anosov flow on a closed oriented $3$-manifold $M$. Suppose $T$ is a scalloped transverse torus. Furthermore, suppose that each closed leaf $\ell^s$ of the induced stable foliation on $T$ intersects each closed leaf $\ell^u$ of the induced unstable foliation on $T$ exactly once. In other words, with respect to fixed bases $(\ell^s, m^s)$ and $(\ell^u, m^u)$, the identity map on $T$ is represented by some matrix 
$\begin{bmatrix}
a & b \\
-1 & d
\end{bmatrix}$.

Then for every matrix of the form
$\begin{bmatrix}
m & n \\
p & q
\end{bmatrix}$
where $p < 0$, one can construct a pseudo-Anosov flow by cutting along $T$ and regluing using a map with the given matrix representative. 
Moreover, if $\phi^t$ is transitive, then this flow is almost equivalent to $\phi^t$.
\end{cor}

\begin{cor} \label{cor:introinteriortorus}
Let $\phi^t$ be a pseudo-Anosov flow on a closed oriented $3$-manifold $M$. Suppose $T$ is a Reebless non-scalloped transverse torus. 
Then one can construct a pseudo-Anosov flow by cutting along $T$ and regluing it along some map, so that $T$ becomes a scalloped torus and so that each closed leaf $\ell^s$ of the induced stable foliation on $T$ intersects each closed leaf $\ell^u$ of the induced unstable foliation on $T$ exactly once. 
Moreover, if $\phi^t$ is transitive, then this flow is almost equivalent to $\phi^t$.
\end{cor}

\Cref{cor:intropAmap} generalizes \cite[Theorem B1]{DS19} to pseudo-Anosov maps, yet we lose the control on the number of closed orbits that one has to drill out.
\Cref{cor:introscalloptorus} generalizes \cite[Proposition 2.7]{Bar98}.

For the rest of this introduction, we will go into more details on choosing to work with pseudo-Anosov flows, the sketches of proofs of \Cref{constr:introhorsur} and \Cref{thm:horsuralmostequiv}, and we will address a gap in \cite{Bar98}.

\subsection{Pseudo-Anosov flows} \label{subsec:intropa}

Pseudo-Anosov flows were introduced by Thurston in the 1980s. These generalize Anosov flows by allowing finitely many singular orbits. For the small price of losing smoothness along these singular orbits, one gains much broader applicability, since pseudo-Anosov flows end up being vastly more prevalent among $3$-manifolds than Anosov flows. See, for example, \cite{Mos96}, \cite{Cal00}, and \cite{BF15}. As a result, it is natural to work with pseudo-Anosov flows more generally when coming from the perspective of applications in $3$-manifold topology.

The dynamical system literature, on the other hand, has traditionally focused on Anosov flows. For most purposes, this is not a big issue. Many results for Anosov flows and their proofs generalize immediately to pseudo-Anosov flows. This is the case, for example, for the work in \cite{Goo83} and \cite{Fri83}.

One result about Anosov flows that does not generalize directly to pseudo-Anosov flows, however, is structural stability. More precisely, the following result, proven in, for example, \cite{KM73}, becomes false if one replaces `Anosov flow' by `pseudo-Anosov flow'.

\begin{thm} \label{thm:anosovstructuralstability}
Let $\phi^t$ be an Anosov flow on a closed $3$-manifold $M$. There exists $\epsilon>0$ so that for every flow $\overline{\phi}^t$ on $M$, if the generating vector fields $\dot{\phi}$ and $\dot{\overline{\phi}}$ are $\epsilon$ close in the $C^1$-topology, then $\phi^t$ and $\overline{\phi}^t$ are orbit equivalent through a homeomorphism that is isotopic to identity. 
\end{thm}

See \cite[Section 3.5]{Mos96} for a notion of `dynamic blow up', which produces counter-examples in the pseudo-Anosov flow setting.

This is a problem for us because we need some sort of structural stability result in order to show the second statement in \Cref{constr:introhorsur}, see the discussion in the next subsection.
One approach for generalizing \Cref{thm:anosovstructuralstability} correctly to pseudo-Anosov flows might be to restrict to perturbations supported away from the singular orbits, yet whether this indeed works is unclear from the proofs of \Cref{thm:anosovstructuralstability} in the literature. Essentially, this is because those proofs construct the homeomorphism through analytic methods, which interact poorly with the existence of singular orbits.

What we end up doing in this paper is further restricting to perturbations that are supported away from specified neighborhoods of the singular orbits. With this stronger restriction, we are able to use a combination of results from axiom A flows and topological constructions to show the following theorem.

\begin{thm} \label{thm:structuralstability}
Let $\phi^t$ be a pseudo-Anosov flow on a closed $3$-manifold $M$. Let $\nu$ be a fixed neighborhood of $\sing(\phi^t)$. There exists $\epsilon>0$ so that for every pseudo-Anosov flow $\overline{\phi}^t$ on $M$, if the generating vector fields $\dot{\phi}$ and $\dot{\overline{\phi}}$ are $\epsilon$ close in the $C^1$-topology and if $\dot{\phi} = \dot{\overline{\phi}}$ in $\nu$, then $\phi^t$ and $\overline{\phi}^t$ are orbit equivalent through a homeomorphism that is isotopic to identity.
\end{thm}

We prove \Cref{thm:structuralstability} in \Cref{sec:structuralstability}. To our knowledge, \Cref{thm:structuralstability} is the first structural stability statement made for pseudo-Anosov flows in the literature, even though it is likely already known to some experts. We state it here in the hopes of providing a useful reference.

\subsection{Sketch of proof of \Cref{constr:introhorsur}} \label{subsec:introhorsurproof}

There are two parts to \Cref{constr:introhorsur}:
\begin{enumerate}[label=(\Roman*)]
    \item That one can locate a suitable annulus and a suitable regluing map to perform the surgery operation.
    \item That the orbit equivalence class of the surgered flow only depends on $c$ and $n$.
\end{enumerate}

For part (I), we define the notion of a surgery annulus and that of a surgery map in \Cref{subsec:horsurgluing}. A \textbf{surgery annulus} is an annulus transverse to the flow that carries a foliation $\mathcal{H}$ by closed curves and a foliation $\mathcal{K}$ by non-separating arcs, so that $\mathcal{H}$ and $\mathcal{K}$ have steady tangent fields, in the same sense as that for a positive/negative curve defined above. The foliations $\mathcal{H}$ and $\mathcal{K}$ allows us to identify $A \cong S^1 \times I$ in some way, and we define \textbf{surgery maps} to be those maps that shear along the $S^1$ direction in some specific way.

The fact that $\mathcal{H}$ and $\mathcal{K}$ have steady tangent fields plays a large role in ensuring that the surgery map compounds the dynamics of the (pseudo-)Anosov flow. More precisely, we will show that the surgered flow is Anosov away from the singular orbits via a cone field argument in \Cref{subsec:horsurpA}. The cones we use are defined in terms of $T\mathcal{H}$ and $T\mathcal{K}$, and the fact that these are steady implies that the cones are contained in one another in the correct way. Aside from this, the workings of the argument are very similar to that employed in \cite{Goo83}, even though we fill in more details. 

One then has to show that the surgered flow has the correct dynamics at the singular orbits. From the way we define pseudo-Anosov flows, it suffices to find an orbit equivalence between the original and surgered flow near each singular orbit. We do this with the help of a local section in \Cref{subsec:horsurpA}.

For part (II), we have to show that the orbit equivalence class of the surgered flow is independent on the choices of a surgery annulus and a surgery map. The main tool for this part of the proof is \Cref{thm:structuralstability}. We essentially show that the space of the various choices that one can make is path-connected, then apply \Cref{thm:structuralstability} along paths of pseudo-Anosov flows to obtain the desired orbit equivalence.

It is not hard to see that \Cref{thm:structuralstability} also implies that the orbit equivalence class of $\phi^t_{\frac{1}{n}}(c)$ is invariant under isotopies of $c$ through horizontal surgery curves. This is recorded as \Cref{cor:invcurveisotopy} in \Cref{subsec:horsurinvar}.

\subsection{A gap in \cite{Bar98}}

An impetus for us in writing down a careful exposition of \Cref{constr:introhorsur} is a gap that is present in \cite{Bar98}. The statement of \cite[Theorem 2.1]{Bar98} is essentially that the first part of \Cref{constr:introhorsur} holds for positive/negative curves in general, without any assumptions on the steadiness of the tangent field. However, this statement is false. In \Cref{noneg:braid} we demonstrate an example of a positive curve $c$ so that cutting along an annulus $A$ containing $c$ and regluing with a Dehn twist with coefficient $1$ returns a reducible $3$-manifold. Meanwhile, it is a classical fact that reducible $3$-manifolds cannot admit Anosov flows.

We should remark that the particular application of \cite[Theorem 2.1]{Bar98} in that paper turns out to be valid. In fact, \Cref{cor:introscalloptorus} is based on the same ideas as the material in \cite[Section 5]{Bar98}. Also, the results in \cite{Bar98} are subsumed by more recent work of Barbot and Fenley in \cite{BF13} and \cite{BF15}, which are developed using different techniques.

\subsection{Sketch of proof of \Cref{thm:horsuralmostequiv}} \label{subsec:introalmostequivproof}

A brief overview of the proof is as follows.
\begin{itemize}
    \item We first perform surgeries along closed orbits so that the horizontal surgery curve $c$ is homotopic to a closed orbit $\gamma$. It is a well-known fact that performing surgeries along closed orbits does not change the almost equivalence class of the flow.
    \item The homotopy between $c$ and $\gamma$ is in general not an isotopy --- generically there are finitely many times where the homotopy goes through a curve that self-intersects. Nevertheless, we arrange it so that the curves in the homotopy that do not self-intersect are horizontal surgery curves.
    \item It follows from \Cref{cor:invcurveisotopy} that, away from the self-intersection times, the orbit equivalence class of $\phi^t_{\frac{1}{n}}(c)$ does not change along the homotopy.
    \item At the self-intersection times, we perform a local computation (\Cref{fig:localmoves1} and \Cref{fig:localmoves2}) to show that the surgered flows immediately before and after are almost equivalent. 
    \item Combining the previous two points, we conclude that $\phi^t_{\frac{1}{n}}(c)$ is almost equivalent to the flow obtained by performing surgery on $\phi^t$ along $\gamma$. As mentioned in the first point, the latter flow is almost equivalent to $\phi^t$.
\end{itemize}

There are many technical details in practice for carrying out this proof. The main source of these technicalities comes from the fact that a smooth curve passing through an orbit $\gamma$ is no longer smooth after we perform surgery on $\gamma$. Indeed, in general it is not even continuous since it gets sheared at $\gamma$, but even if one reconnects the ends by isotoping along $\gamma$, the rejoined curve is only piecewise smooth. As a result, we have to generalize the machinery to these piecewise smooth curves. We do so in \Cref{sec:piecewisesmooth}.

Actually, we only carry out this task for a restricted class of piecewise smooth curves. This helps simplify the analysis in our proofs, but still, the arguments get rather involved. 
In any case, with this generalization done, the proof itself is contained in \Cref{sec:almostequiv}.

\subsection*{Outline of paper}

In \Cref{sec:prelim}, we recall some basic definitions on pseudo-Anosov flows. In \Cref{sec:horsurcurve}, we define horizontal surgery curves and discuss examples of these. In \Cref{sec:horsur}, we explain the horizontal Goodman surgery operation, proving \Cref{constr:introhorsur}.

In \Cref{sec:piecewisesmooth}, we develop a partial generalization of horizontal Goodman surgery to piecewise smooth curves. In \Cref{sec:almostequiv}, we prove \Cref{thm:horsuralmostequiv}.

In \Cref{sec:questions}, we discuss some future directions. Finally, in \Cref{sec:structuralstability}, we prove \Cref{thm:structuralstability}.

\subsection*{Acknowledgements}

We would like to thank Pierre Dehornoy for discussions during the \textit{Anosov dynamics} program at CIRM which lead to the results in this paper. We would like to thank the program organizers for enabling these discussions to take place. We would also like to thank Sergio Fenley, Jessica Purcell, and Mario Shannon for sharing their unpublished work with us. Finally, we would like to thank Thierry Barbot, Thomas Barthelm\'e, Anna Parlak, and Federico Salmoiraghi for helpful conversations. Part of this project was completed while the author is based at CIRGET. We would like to thank the center for its support.

\subsection*{Notational conventions}

Throughout this paper, 
\begin{itemize}
    \item $M$ will denote a closed oriented $3$-manifold.
    \item Let $c$ be a parametrized smooth embedded $1$-manifold (possibly with boundary) in $M$.
    \begin{itemize}
        \item $Tc$ will denote the tangent bundle of $c$. This is a sub-bundle of $TM$ defined over $c$.
        \item $\dot{c}$ will denote the derivative of $c$. This is a vector field defined over $c$.
    \end{itemize}
    \item Let $\phi^t$ be a smooth nonsingular flow on $M$.
    \begin{itemize}
        \item $T\phi$ will denote the tangent bundle of the flow lines of $\phi^t$. This is a sub-bundle of $TM$.
        \item $\dot{\phi}$ will denote the derivative of the flow lines of $\phi^t$. This is a vector field on $M$.
        \item $d\phi^t$ will denote the derivative of the diffeomorphism $\phi^t:M \to M$. This is a bundle map $TM \to TM$.
    \end{itemize}
\end{itemize}

\section{Preliminaries} \label{sec:prelim}

\subsection{Pseudo-Anosov flows} \label{subsec:pAflow}

We start by recalling the definition of an Anosov flow.

\begin{defn} \label{defn:Aflow}
An \textbf{Anosov flow} on a closed oriented $3$-manifold $M$ is a smooth flow $\phi^t$ for which there is a Riemannian metric $g$ and a continuous splitting of the tangent bundle into three $\phi^t$-invariant line bundles $TM=E^s \oplus T\phi \oplus E^u$ such that 
$$||d\phi^t(v)||_g < C \lambda^{-t} ||v||_g$$
for every $v \in E^s, t>0$, and 
$$||d\phi^t(v)||_g < C \lambda^t ||v||_g$$
for every $v \in E^u, t<0$, for some $C, \lambda>1$.
\end{defn}

A pseudo-Anosov flow is essentially an Anosov flow with finitely many singular orbits. The singular orbits are modelled after the following construction.

\begin{constr} \label{constr:phorbit}
Consider the map $\begin{pmatrix} \lambda & 0\\ 0 & \lambda^{-1} \end{pmatrix}: \mathbb{R}^2 \to \mathbb{R}^2$. This preserves the foliations of $\mathbb{R}^2$ by vertical and horizontal lines. Let $\phi_{n,0,\lambda}:\mathbb{R}^2 \to \mathbb{R}^2$ be the lift of this map over $z \mapsto z^{\frac{n}{2}}$ that preserves the lifts of the quadrants. (When $n$ is odd, one has to choose a branch of $z \mapsto z^{\frac{n}{2}}$ but it is easy to see that the result is independent of the choice.) Let $\phi_{n,k, \lambda}: \mathbb{R}^2 \to \mathbb{R}^2$ be the composition of $\phi_{n,0,\lambda}$ and rotating by $\frac{2\pi k}{n}$ anticlockwise. Also pull back the foliations of $\mathbb{R}^2$ by vertical and horizontal lines. The resulting two singular foliations are preserved by $\phi_{n,k,\lambda}$. We call them $\ell^s_n$ and $\ell^u_n$ respectively.

Let $\Phi_{n,k,\lambda}$ be the mapping torus of $\phi_{n,k, \lambda}$, $\Lambda^{s/u}$ be the suspension of $\ell^{s/u}_n$ respectively, and consider the suspension flow on $\Phi_{n,k, \lambda}$. We call the suspension of the origin the \textbf{pseudo-hyperbolic orbit} of $\Phi_{n,k,\lambda}$. 

We show an example where $n=3$ and $k=0$ in \Cref{fig:phorbit}. On the left we depict the dynamics of $\phi_{3,0, \lambda}$. On the right we illustrate a local picture of $\Phi_{3,0,\lambda}$ near the pseudo-hyperbolic orbit.

\begin{figure}
    \centering
\begingroup%
  \makeatletter%
  \providecommand\color[2][]{%
    \errmessage{(Inkscape) Color is used for the text in Inkscape, but the package 'color.sty' is not loaded}%
    \renewcommand\color[2][]{}%
  }%
  \providecommand\transparent[1]{%
    \errmessage{(Inkscape) Transparency is used (non-zero) for the text in Inkscape, but the package 'transparent.sty' is not loaded}%
    \renewcommand\transparent[1]{}%
  }%
  \providecommand\rotatebox[2]{#2}%
  \newcommand*\fsize{\dimexpr\f@size pt\relax}%
  \newcommand*\lineheight[1]{\fontsize{\fsize}{#1\fsize}\selectfont}%
  \ifx\svgwidth\undefined%
    \setlength{\unitlength}{201.95111613bp}%
    \ifx\svgscale\undefined%
      \relax%
    \else%
      \setlength{\unitlength}{\unitlength * \real{\svgscale}}%
    \fi%
  \else%
    \setlength{\unitlength}{\svgwidth}%
  \fi%
  \global\let\svgwidth\undefined%
  \global\let\svgscale\undefined%
  \makeatother%
  \begin{picture}(1,0.64757601)%
    \lineheight{1}%
    \setlength\tabcolsep{0pt}%
    \put(0,0){\includegraphics[width=\unitlength,page=1]{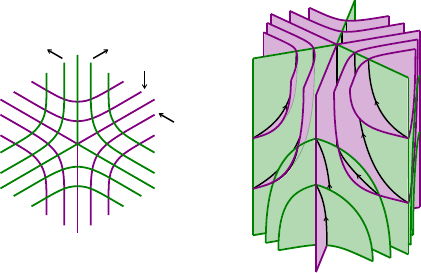}}%
    \put(0.16961672,0.54972479){\color[rgb]{0,0.50196078,0}\makebox(0,0)[lt]{\lineheight{1.25}\smash{\begin{tabular}[t]{l}$\ell^s_3$\end{tabular}}}}%
    \put(0.38601847,0.42264593){\color[rgb]{0.50196078,0,0.50196078}\makebox(0,0)[lt]{\lineheight{1.25}\smash{\begin{tabular}[t]{l}$\ell^u_3$\end{tabular}}}}%
  \end{picture}%
\endgroup%

    \caption{Left: The dynamics of $\phi_{3,0,\lambda}$. Right: A local picture of $\Phi_{3,0,\lambda}$}
    \label{fig:phorbit}
\end{figure}

\end{constr}

\begin{defn} \label{defn:pAflow}
A \textbf{pseudo-Anosov flow} on a closed oriented $3$-manifold $M$ is a continuous flow $\phi^t$ for which there is a path metric $d$ that is induced from a Riemannian metric $g$ away from a finite collection of closed orbits, which we call the \textbf{singular orbits}, such that:
\begin{itemize}
    \item The flow $\phi^t$ is smooth away from the union of singular orbits $\sing(\phi^t)$.
    \item Away from the singular orbits, there is a continuous splitting of the tangent bundle into three $\phi^t$-invariant line bundles $TM=E^s \oplus T\phi \oplus E^u$ such that 
    $$||d\phi^t(v)||_g < C \lambda^{-t} ||v||_g$$
    for every $v \in E^s, t>0$, and 
    $$||d\phi^t(v)||_g < C \lambda^t ||v||_g$$
    for every $v \in E^u, t<0$, for some $C, \lambda>1$.
    \item Each singular orbit $\gamma$ has a neighborhood $N$ and a map $f$ sending $N$ to a neighborhood of the pseudo-hyperbolic orbit in $\Phi_{n, k, \lambda}$, for some $n \geq 3$, such that
    \begin{itemize}
        \item $f$ is bi-Lipschitz on $N$ and smooth away from $\gamma$,
        \item $f$ preserves the orbits, and
        \item $f$ sends $E^s$ and $E^u$ to line bundles tangent to $\Lambda^s$ and $\Lambda^u$ respectively.
    \end{itemize}
\end{itemize}
\end{defn}

Notice that \Cref{defn:pAflow} is independent of the bi-Lipschitz class of the metric, i.e. if the conditions in \Cref{defn:pAflow} hold for one path metric $d$, then up to increasing the value of $C$, they hold for any other path metric $d'$ that is bi-Lipschitz to $d$ as well.

The following well-known lemma, however, states that there are certain metrics one can choose so that the contraction/expansion along the stable/unstable bundles are `instantaneous'. For completeness we provide a proof.

\begin{lemma} \label{lemma:instantmetricexist}
Suppose $\phi^t$ is a pseudo-Anosov flow. Then there exists a path metric $d$ that is induced from a Riemannian metric $g$ away from $\sing(\phi^t)$ such that 
$$||d\phi^t(v)||_g < \lambda^{-t} ||v||_g$$
for every $v \in E^s, t>0$, and 
$$||d\phi^t(v)||_g < \lambda^t ||v||_g$$
for every $v \in E^u, t<0$, for some $\lambda>1$.
\end{lemma}
\begin{proof}
Fix some path metric $d$ induced from a Riemannian metric $g$ away from $\sing(\phi^t)$, so that we have
$$||d\phi^t(v)||_g < C \lambda^{-t} ||v||_g$$
for every $v \in E^s, t>0$, and 
$$||d\phi^t(v)||_g < C \lambda^t ||v||_g$$
for every $v \in E^u, t<0$, for some $C, \lambda>1$.
The idea of the proof is to consider the path metric $\overline{d}(x,y)=\int_0^T d(\phi^t(x),\phi^t(y)) dt$, which is induced from the Riemannian metric $\overline{g}=\int_0^T (\phi^t)^*g dt$ away from $\sing(\phi^t)$, for large $T$.

Pick a large enough $T$ so that $C \lambda^{-T} < \frac{1}{2}$. For $v \in E^s$ and $t \in [0,T]$, we have
\begin{align*}
||v||_{\overline{g}} &= \int_0^t ||d\phi^t(v)||_g dt + \int_t^T ||d\phi^t(v)||_g dt
\end{align*}
\begin{align*}
||d\phi^t(v)||_{\overline{g}} &= \int_t^T ||d\phi^t(v)||_g dt + \int_T^{T+t} ||d\phi^t(v)||_g dt \\
&= \int_t^T ||d\phi^t(v)||_g dt + \int_0^{t} ||d\phi^{T}(d\phi^t(v))||_g dt \\
&< \int_t^T ||d\phi^t(v)||_g dt + \int_0^t C \lambda^{-T} ||d\phi^t(v)||_g dt \\
&< \int_t^T ||d\phi^t(v)||_g dt + \frac{1}{2} \int_0^t ||d\phi^t(v)||_g dt
\end{align*}

\begin{claim} \label{claim:instantmetric}
There exists $k > 0$ such that $\displaystyle \frac{\int_0^t ||d\phi^t(v)||_g dt}{\int_t^T ||d\phi^t(v)||_g dt} \geq kt$ for all $t \in [0,T]$.
\end{claim}

Assuming the claim for now, we have 
\begin{align*}
\log(\frac{||d\phi^t(v)||_{\overline{g}}}{||v||_{\overline{g}}}) \leq \log(\frac{\frac{1}{2} kt+1}{kt+1}) < -\frac{1}{2}kt
\end{align*}
for $t \in (0,T]$.
Hence 
$$||d\phi^t(v)||_{\overline{g}} < \overline{\lambda}^{-t} ||v||_{\overline{g}}$$
for every $v \in E^s$ and $t \in (0,T]$.

For larger values of $t$, we can write $t=nT+q$ where $n \in \mathbb{Z}_{\geq 0}$ and $q \in (0,T]$. Then we have
$$||d\phi^t(v)||_{\overline{g}} = ||\underbrace{d\phi^T( \cdots d\phi^T(}_{\text{$n$ times}} d\phi^q(v)) \cdots )||_{\overline{g}} < \underbrace{\overline{\lambda}^{-T} \cdots \overline{\lambda}^{-T}}_{\text{$n$ times}} \overline{\lambda}^{-q} ||v||_{\overline{g}} = \overline{\lambda}^{-t} ||v||_{\overline{g}}.$$

A symmetric argument shows that up to decreasing $\overline{\lambda}$, we can arrange for 
$$||d\phi^t(v)||_{\overline{g}} < \overline{\lambda}^t ||v||_{\overline{g}}$$
for every $v \in E^u, t<0$ as well.

It remains to show \Cref{claim:instantmetric}. We first argue that $B := \sup_{t \in [0,T]} ||d\phi^t|_{E^s}||_g < \infty$. To see this, we pick a neighborhood $N$ of $\sing(\phi^t)$ as in \Cref{defn:pAflow}. Let $\overline{N} \subset N$ be the smaller neighborhood consisting of points staying inside $N$ in time $[0,T]$ and write 
$$||d\phi^t|_{E^s}||_g = \max\{ ||d\phi^t|_{E^s|_{M \backslash \overline{N}}}||_g, ||d\phi^t|_{E^s|_{\overline{N} \backslash \sing(\phi^t)}}||_g\}.$$
$||d\phi^t|_{E^s|_{M \backslash \overline{N}}}||_g$ is uniformly bounded for $t \in [0,T]$ by compactness of $M \backslash \overline{N}$. $||d\phi^t|_{E^s|_{\overline{N} \backslash \sing(\phi^t)}}||_{g_0}$ is also uniformly bounded for $t \in [0,T]$ by uniform boundedness of $||d\phi^t|_{E^s}||$ for $\Phi_{n,k,\lambda}$ and bi-Lipschitzness of the maps $f$.

By a similar argument, we see that $b := \inf_{t \in [0,T]} ||d\phi^t|_{E^s}||_g > 0$. Hence 
$$\frac{\int_0^t ||d\phi^t(v)||_g dt}{\int_t^T ||d\phi^t(v)||_g dt} \geq \frac{bt||v||_g}{BT||v||_g} = \frac{b}{BT} t.$$
\end{proof}

We call a path metric $d$ as in \Cref{lemma:instantmetricexist} an \textbf{instantaneous metric}.
In the sequel, we will often fix an instantaneous metric to work with, and assume a value for $\lambda$ as in \Cref{lemma:instantmetricexist} without additional comment.

We remark that \Cref{defn:Aflow} and \Cref{defn:pAflow} are sometimes referred to as the definition of \emph{smooth} (pseudo-)Anosov flows. In comparison, one can find a definition of \emph{topological} pseudo-Anosov flows in, for example, \cite{Mos96}.
We will, however, choose to define orbit equivalence in the topological category in this paper.

\begin{defn} \label{defn:orbitequiv}
Let $\phi^t_i$ be a continuous flow on a $3$-manifold $N_i$ for $i=1,2$. Suppose there exists a homeomorphism $h:N_1 \to N_2$ that sends the orbits of $\phi^t_1$ to $\phi^t_2$ in an orientation preserving way (but not necessarily preserving their parametrizations), then we say that $\phi^t_1$ and $\phi^t_2$ are \textbf{orbit equivalent}, and that $h$ is an \textbf{orbit equivalence}. As a notational shorthand, we will write $\phi^t_1 \cong \phi^t_2$.
\end{defn}

Finally we recall the definition of the orbit space of a pseudo-Anosov flow. 

\begin{defn} \label{defn:orbitspace}
Let $\phi^t$ be a pseudo-Anosov flow on closed oriented $3$-manifold. It can be deduced from classical stable manifold theory (see, for example, \cite{HPS77}) that the plane field $E^s \oplus T\phi$ integrates uniquely to a 2-dimensional foliation. This foliation agrees with the image of the singular foliation $\Lambda^s$ on $\Phi_{n,k,\lambda}$ near $\sing(\phi^t)$, hence we can take the union of these to get a singular foliation on $M$, which we call the \textbf{stable foliation} and denote by $\Lambda^s$. Similarly, we have the \textbf{unstable foliation} $\Lambda^u$ which is tangent to $T\phi \oplus E^u$ away from the singular orbits. 

Now let $\widetilde{\phi}^t$ be the lifted flow to the universal cover $\widetilde{M}$. It is shown in \cite[Proposition 4.1]{FM01} that the space of orbits $\mathcal{O}$ of $\widetilde{\phi}^t$, with the quotient topology, is homeomorphic to $\mathbb{R}^2$. The lifted stable/unstable foliations $\widetilde{\Lambda^{s/u}}$ induce singular 1-dimensional foliations $\mathcal{O}^{s/u}$ on $\mathcal{O}$. We refer to the space $\mathcal{O}$ with the foliations $\mathcal{O}^{s/u}$ as the \textbf{orbit space} of $\phi^t$.
\end{defn}

\subsection{Positive/negative curves} \label{subsec:posnegcurve}

Let $\phi^t$ be a pseudo-Anosov flow on a closed oriented $3$-manifold $M$. Let $x$ be a point in $M \backslash \sing(\phi^t)$. Consider the $2$-dimensional space $E^s|_x \oplus E^u|_x$. In this paper, we will adopt the convention of considering $E^s|_x$ as the vertical direction and $E^u|_x$ as the horizontal direction. We will also adopt the convention of illustrating $E^s|_x$ in green and $E^u|_x$ in purple.

Fix an instantaneous metric $d$. Let $e^{s/u}$ be unit length vectors in $E^{s/u}|_x$, respectively, such that $(e^s, \dot{\phi}, e^u)$ determines a positive basis of $TM|_x$. Suppose $\ell$ is a line in $TM|_x$ spanned by a vector $ae^s+b\dot{\phi}+ce^u$. If $a \neq 0$ and $c \neq 0$, we define the \textbf{slope} of $\ell$ (with respect to $d$) to be $\frac{a}{c}$. Note that this is consistent with the directional convention specified above.
Note also that since the only indeterminancy for $e^s$ and $e^u$ is a reversal of both of their signs, the slope is well-defined.

We say that $\ell$ is \textbf{positive/negative} if the slope of $\ell$ is positive/negative, respectively. Notice that while the slope of $\ell$ depends on the choice of the instantaneous metric $d$, whether $\ell$ is positive/negative depends only on $\ell$.

Suppose that $L$ is a 1-dimensional sub-bundle of $TM$ defined over a compact set $K \subset M \backslash \sing(\phi^t)$. We say that $L$ is \textbf{positive/negative} if $L|_x$ is positive/negative at every $x \in K$, respectively.

\begin{defn} \label{defn:posnegcurve}
Let $c$ be a smooth embedded $1$-manifold (possibly with boundary) in $M \backslash \sing(\phi^t)$. We say that $c$ is \textbf{positive/negative} if $Tc$ is positive/negative. 
\end{defn}

In this paper, we will illustrate positive curves in red and negative curves in blue. See \Cref{fig:posnegcurve} for an example of a local picture of the orbit space, along with the images of a positive and a negative curve. In this paper, we consider flows to go vertically upwards and we illustrate the orbit space in the top-down view, i.e. our convention is that the orbits of $\widetilde{\phi^t}$ come out of the page.

\begin{figure}
    \centering
    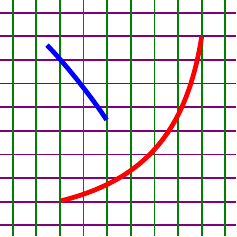
    \caption{The image of a positive curve (in red) and a negative curve (blue) in the orbit space.}
    \label{fig:posnegcurve}
\end{figure}

Our directional and coloring convention is borrowed from the veering triangulation literature. This is in anticipation of the potential connections between the material we present in this paper and the field of veering triangulations. We elaborate on this in \Cref{subsec:vt}.

\section{Horizontal surgery curves} \label{sec:horsurcurve}

In this section, we define and study horizontal surgery curves. These are the curves along which horizontal Goodman surgery can be performed.

\subsection{Definition and properties} \label{subsec:horsurcurvedefn}

Let $\phi^t$ be a pseudo-Anosov flow on a closed oriented $3$-manifold $M$. Fix an instantaneous metric on $M$.

\begin{defn} \label{defn:steady}
Suppose that $L$ is a positive 1-dimensional sub-bundle of $TM$ defined over a set $K \subset M \backslash \sing(\phi^t)$. We say that $L$ is \textbf{steady} if for every $x,y \in K, t>0$ such that $y=\phi^t(x)$, we have $\slope(L|_y) > \slope(d\phi^t(L|_x))$. 

Similarly, suppose that $L$ is a negative 1-dimensional sub-bundle of $TM$ defined over a set $K \subset M \backslash \sing(\phi^t)$. We say that $L$ is \textbf{steady} if for every $x,y \in K, t>0$ such that $y=\phi^t(x)$, we have $\slope(L|_y) < \slope(d\phi^t(L|_x))$. 
\end{defn}

Note that $L$ being steady is actually independent of the choice of the instantaneous metric. Indeed, a metric-free definition would be to say that $L|_y$ is closer to $E^s|_x$ than $d\phi^t(L|_x)$. However, we will choose to work with metrics in this paper for the sake of having quantitative control.

We explain our choice of terminology. Suppose $L$ is a positive 1-dimensional sub-bundle defined over $K$. Suppose $K$ is compact. Fix a point $x \in K$ and consider the function $f(t)= \slope(d\phi^{-t}(L|_{\phi^t(x)}))$ defined over $\{t \mid \phi^t(x) \in K\}$. Since $d\phi^{-t}$ expands along the stable direction and contracts along the unstable direction, $f(t)$ is increasing `in the long run'. 
More precisely, since $K$ is compact, $\slope(L|_{\phi^t(x)})$ is bounded between some $h>0$ and $H<\infty$, hence $f(t+T) > f(t)$ whenever $\lambda^{-2T} < \frac{h}{H}$ (and the values of $f$ are defined).

What the steadiness condition means is that $f(t)$ is an increasing function on the nose. That is, it increases steadily without wiggling around in the short run.

\begin{defn} \label{defn:horsurcurve}
Let $c$ be a positive/negative smooth embedded curve in $M \backslash \sing(\phi^t)$. We say that $c$ is a \textbf{positive/negative horizontal surgery curve}, respectively, if $Tc$ is steady.
\end{defn}

We introduce some terminology to help the reader visualize \Cref{defn:horsurcurve}. Let $c_1$ and $c_2$ be smooth embedded $1$-manifolds (possibly with boundary). We say that a triple $(x,y,t) \in c_1 \times c_2 \times (0,\infty)$ is a \textbf{time $t$ crossing of $c_2$ over $c_1$} if $y=\phi^t(x)$. If $c_1=c_2=c$, we abbreviate this to a \textbf{time $t$ crossing of $c$}. 

The motivation for this terminology comes from considering the situation in terms of the orbit space $\mathcal{O}$ --- (appropriate lifts of) $x$ and $y$ lie on the same orbit hence occupy the same point in $\mathcal{O}$. If, for example, $c$ is a horizontal surgery curve, then the fact that $\slope(Tc|_y) \neq \slope(d\phi^t(Tc|_x))$ implies that the segment of $c$ near $y$ crosses the segment of $c$ near $x$ transversely. Finally, the data that $y$ lies above $x$ gives this transverse intersection point the same data as a crossing in a knot diagram.

From this perspective, a positive/negative horizontal surgery curve is a positive/negative curve $c$ where every crossing is of the form in \Cref{fig:steadycurve} left/right, respectively. 

\begin{figure}
    \centering
    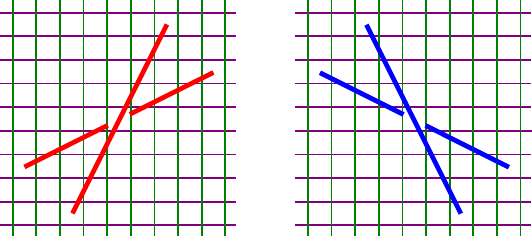
    \caption{Left/right: Local form of a positive/negative surgery curve, respectively, when projected to the orbit space.}
    \label{fig:steadycurve}
\end{figure}

We show the following two propositions concerning isotopies of horizontal surgery curves.

\begin{prop} \label{prop:horsurcurveflowisotopy}
Being a positive/negative horizontal surgery curve is preserved under isotopy along orbits of $\phi^t$. 

That is, if $c:S^1 \to M$ is a parametrized positive/negative horizontal surgery curve, and $c_s:S^1 \to M$ is a family of parametrized smooth embedded curves such that $c_0=c$ and $c_s(z)$ lies along an orbit of $\phi^t$ for every $z \in S^1$, then $c_s$ is a positive/negative horizontal surgery curve for every $s$, respectively.
\end{prop}
\begin{proof}
This follows immediately from the orbit space perspective mentioned above, since isotopy along orbits of $\phi^t$ does not change the image of the curve. Nevertheless, we will give a direct proof of the proposition. We do this in the case when $c$ is positive. The proof when $c$ is negative is similar.

We can write $c_s(z)=\phi^t(c(z))$, where $t=t(s,z)$ depends on $s$ and $z$. From this one can compute that $\dot{c_s} = d\phi^t(\dot{c}) + \frac{\partial t}{\partial z} \dot{\phi}$, which implies that $\slope(Tc_s) = \slope(d\phi^t(Tc))$.

Meanwhile, the crossings of $c_s$ are in one-to-one correspondence with the crossings of $c$. Precisely, the correspondence is given by 
$$(c_s(z_1), c_s(z_2), t+t(s,z_2)-t(s,z_1)) \longleftrightarrow (c(z_1),c(z_2),t).$$
From the computation in the previous paragraph, we can rephrase the steadiness condition for $c_s$ at the crossing $(c_s(z_1), c_s(z_2), t+t(s,z_2)-t(s,z_1))$ to be that $\slope(d\phi^{t(s,z_2)}(Tc|_{c(z_2)})) > \slope(d\phi^{t+t(s,z_2)}(Tc|_{c(z_1)}))$.
But this follows from the steadiness condition for $c$ at the crossing $(c(z_1),c(z_2),t)$, which states that $\slope(Tc|_{c(z_2)}) > \slope(d\phi^t(Tc|_{c(z_1)}))$.
\end{proof}

\begin{prop} \label{prop:horsurcurvestable}
Being a positive/negative horizontal surgery curve is a $C^1$-open condition. 

That is, if $c$ is a positive/negative horizontal surgery curve and $\overline{c}$ is a smooth embedded curve sufficiently close to $c$ in the $C^1$-topology, then $\overline{c}$ is also a positive/negative horizontal surgery curve, respectively.
\end{prop}
\begin{proof}
We prove this in the case when $c$ is positive. The proof when $c$ is negative is similar.

Recall that the \textbf{first return time} of a set $K$ is given by 
$$\inf \{t>0 \mid \text{there exists $x,y \in K$ such that $y=\phi^t(x)$}\}.$$
Since $c$ is compact and $Tc$ lies away from $T\phi$, the first return time of $c$ is positive. We fix some $T_0 > 0$ smaller than this value.

Fix an instantaneous metric. The idea of the proof is that for $\overline{c}$ close to $c$ in the $C^1$-topology,
\begin{enumerate}[label=(\roman*)]
    \item The maximum and minimum slopes of $T\overline{c}$ are bounded, so the steadiness condition for $\overline{c}$ holds for time $\geq T_1$ crossings for sufficiently large $T_1$.
    \item The first return time of $\overline{c}$ is $> T_0$.
    \item Each time $[T_0,T_1]$ crossing of $\overline{c}$ lies close to a crossing of $c$. For each such crossing, the steadiness condition is $C^1$-open, hence holds for $\overline{c}$.
\end{enumerate}

We elaborate more on the details. 
For (i), let $H$ and $h$ be the maximum and minimum slopes of $Tc$ respectively. Pick $T_1$ so that $\lambda^{-2T_1} < \frac{h}{H}$. For $\overline{c}$ close to $c$ in the $C^1$-topology, the maximum and minimum slopes of $T\overline{c}$ are close to $H$ and $h$ respectively, so the steadiness condition for $\overline{c}$ automatically holds for time $\geq T_1$ crossings.

For (ii), consider the annulus $\phi^{[0,T_0]}(c)$. This is embedded since the first return time of $c$ is $>T_0$. For $\overline{c}$ sufficiently close to $c$ in the $C^1$-topology, the corresponding annulus $\phi^{[0,T_0]}(\overline{c})$ is close to $\phi^{[0,T_0]}(c)$ in the $C^1$-topology, hence is embedded as well. This implies that the first return time of $\overline{c}$ is $> T_0$. In other words, every crossing of $\overline{c}$ is of time $>T_0$.

For (iii), we establish the following more general claim.

\begin{claim} \label{claim:curveperturbcrossings}
Let $M$ be a closed $3$-manifold and let $\phi^t$ be a flow on $M$. Let $c_1$ and $c_2$ be compact embedded 1-manifolds (possibly with boundary). Fix $T_0<T_1$. For every $\epsilon>0$, there exists $\delta>0$ such that for any compact embedded $1$-manifolds $\overline{c}_1$ and $\overline{c}_2$ that are $\delta$ close to $c_1$ and $c_2$ in the $C^0$-topology, respectively, every time $[T_0,T_1]$ crossing of $\overline{c}_2$ over $\overline{c}_1$ is $\epsilon$ close to a time $[T_0,T_1]$ crossing of $c_2$ over $c_1$.
\end{claim}
\begin{proof}
Consider the map 
\begin{align*}
F: c_1 \times c_2 \times [T_0,T_1] &\to M \times M \times [T_0,T_1] \\
(x,y,t) &\mapsto (\phi^t(x),y,t)
\end{align*}
The set $C$ can be written as $F^{-1}(\Delta \times [T_0,T_1])$ where $\Delta \subset M \times M$ is the diagonal. 

For a given $\epsilon$, we let $N_\epsilon(C)$ be the $\epsilon$-neighborhood of $C$. $F((c_1 \times c_2 \times [T_0,T_1]) \backslash N_\epsilon(C))$ is a compact set disjoint from $\Delta \times [T_0,T_1]$, hence is of distance $\eta>0$ away from it. 

Now for maps $\overline{F}$ that are $\eta$ close to $F$ in the $C^0$-topology, every point in $\overline{F}^{-1}(\Delta \times [T_0,T_1])$ lies in $N_\epsilon(C)$ hence is $\epsilon$ close to a point in $C$. In particular the claim holds for $\overline{c}_1$ and $\overline{c}_2$ that are sufficiently close to $c_1$ and $c_2$ in the $C^0$-topology, respectively.
\end{proof}

Returning to the proof of \Cref{prop:horsurcurvestable}, we take $c_1=c_2=c$ and $\overline{c}_1=\overline{c}_2=\overline{c}$ in \Cref{claim:curveperturbcrossings} to see that each time $[T_0,T_1]$ crossing $(\overline{x}, \overline{y}, \overline{t})$ of $\overline{c}$ can be made arbitrarily close to a crossing $(x,y,t)$ of $c$. 
In particular, $\slope(d\phi^{\overline{t}}(T\overline{c}|_{\overline{x}}))$ and $\slope(T\overline{c}|_{\overline{y}})$ can be made arbitrarily close to $\slope(d\phi^t(Tc|_x))$ and $\slope(Tc|_y)$ respectively, so that $\slope(Tc|_y)>\slope(d\phi^t(Tc|_x))$ implies $\slope(T\overline{c}|_{\overline{y}})>\slope(d\phi^{\overline{t}}(T\overline{c}|_{\overline{x}}))$.
\end{proof}

\subsection{Examples} \label{subsec:curveeg}

In this subsection, we present some examples of horizontal surgery curves. 

We will use the following condition for the first few families of examples.

\begin{lemma} \label{lemma:constantslope}
Suppose an instantaneous metric is chosen. Let $c$ be a positive/negative smooth embedded curve in $M \backslash \sing(\phi^t)$ such that the slope of $Tc|_x$ is some constant value $m$. Then $c$ is a positive/negative horizontal surgery curve, respectively.
\end{lemma}
\begin{proof}
We prove this for $m>0$. For every crossing $(x,y,t)$, the slope of $Tc|_y$ is $m$, while the slope of $d\phi^t(L|_x)$ is $< \lambda^{-2t}m < m$. The case when $m<0$ is similar.
\end{proof}

\begin{eg} \label{eg:pAmapstraightcurve}
Let $S$ be a closed oriented surface and let $f:S \to S$ be an orientation preserving pseudo-Anosov map. This means that there exists a pair of singular measured foliations $(\ell^s,\mu^s)$ and $(\ell^u,\mu^u)$ such that $f_* (\ell^s,\mu^s) = (\ell^s,\lambda^{-1} \mu^s)$ and $f_* (\ell^u,\mu^u) = (\ell^u,\lambda \mu^u)$ for some $\lambda>1$. We refer to \cite{FLP79} for more on pseudo-Anosov maps. 

The pair of measured foliations determines a Euclidean metric $ds^2 = d(\mu^s)^2 + d(\mu^u)^2$ in the complement of the set of their singularities $\mathfrak{s}$. We say that a curve $c$ on $S \backslash \mathfrak{s}$ is \textbf{straight} if it is locally a straight line with respect to this Euclidean metric.

Meanwhile, one can construct the mapping torus $T_f$ by taking $S \times [0,1]_t$ and identifying $(x,1)$ with $(f(x),0)$ for every $x \in S$. The vector field $\partial_t$ induces the pseudo-Anosov suspension flow $\phi^t_f$. 

\begin{prop} \label{prop:pAmapstraightcurve}
A straight curve $c$ on $S \times \{0\}$ with positive/negative slope is a positive/negative surgery curve for the suspension flow respectively.
\end{prop}
\begin{proof}
One can define an instantaneous metric for the suspension flow by $ds^2 = \lambda^{-2t} d(\mu^s)^2 + dt^2 + \lambda^{2t} d(\mu^u)^2$ away from the set of singular orbits. Under this instantaneous metric, a straight curve $c$ on $S \times \{0\}$ has constant slope $m$, hence is a positive/negative surgery curve if $m$ is positive/negative, respectively, by \Cref{lemma:constantslope}. 
\end{proof}

\end{eg}

\begin{eg} \label{eg:closedorbitcurve}
Let $\phi^t$ be a pseudo-Anosov flow on a closed oriented $3$-manifold and let $\gamma$ be a closed (potentially singular) orbit of $\phi^t$. Let $N$ be a tubular neighborhood of $\gamma$ and let $A$ be an annulus in $N$ with one boundary component $\overline{\gamma}$ lying along $\gamma$ and which is transverse to $\phi^t$ in its interior. We remark that in general $\overline{\gamma}$ is identified with $\gamma$ via a multi-fold covering map.

The stable and unstable foliations $\Lambda^s$ and $\Lambda^u$ induce foliations $A^s$ and $A^u$ on $A$, respectively, which intersect transversely. The leaves of $A^s$ are asymptotic to $\overline{\gamma}$ in the backward direction while the leaves of $A^u$ are asymptotic to $\overline{\gamma}$ in the forward direction. See \Cref{fig:closedorbitannulus}.

\begin{figure}
    \centering
    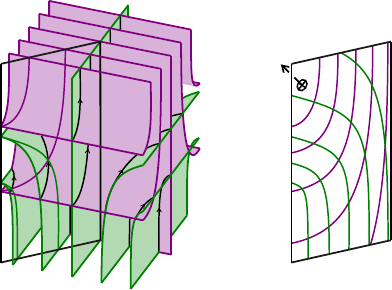
    \caption{An annulus $A$ in a positive quadrant of $N$ with one boundary component lying along $\gamma$. In the perspective of the figure, we are looking at $A$ from its back side.}
    \label{fig:closedorbitannulus}
\end{figure}

Notice that the local leaves of $\Lambda^s$ and $\Lambda^u$ containing $\gamma$ cut $N$ into a number of components, which we call \textbf{quadrants}, and $A$ must lie in one of these. Let us call a quadrant \textbf{positive/negative} if any curve within starting at $\gamma$ is positive/negative near $\gamma$, respectively.

\begin{prop} \label{prop:closedorbitcurve}
If $A$ lies in a positive/negative quadrant, then there is a negative/positive horizontal surgery curve lying on $A$ that is isotopic to $\overline{\gamma}$, respectively.
\end{prop}
\begin{proof}
We prove the proposition in the case when $A$ lies in a positive quadrant. The case when $A$ lies in a negative quadrant is similar. Fix an instantaneous metric. For each $m \in [-\infty,0]$, one can define a line field $L_m$ in the interior of $A$ by requiring that $L_m$ be of slope $m$ at each point. Here we take $L_0 = A^u$ and $L_{-\infty} = A^s$. 

Since $A^s$ and $A^u$ are almost parallel to $\overline{\gamma}$, up to shrinking $A$, we can pick a non-separating positive arc $l$ on $A$. For $m \in [-\infty,0]$, by following the trajectories of $L_m$ in the same direction as the flow on $\gamma$, we get a first return map $f_m$ on $l$. For $m=0$, $f_m$ is a contraction, i.e. $f_m(x)$ is closer to $l \cap \overline{\gamma}$ than $x$ for every $x$, while for $m=-\infty$, $f_m$ is an expansion, i.e. $f_m(x)$ is further away from $l \cap \overline{\gamma}$ than $x$ for every $x$ on which $f_m$ is defined. Moreover, $f_m$ is monotone in $m$ in the sense that for $m_1 < m_2$, $f_{m_1}(x)$ is further away from $l \cap \overline{\gamma}$ than $f_{m_2}(x)$ for every $x$. Hence by the intermediate value theorem, there exists $x \in l$ and $m \in (-\infty,0)$ so that $f_m(x)=x$. This gives a curve of constant slope $m<0$ on $A$. By \Cref{lemma:constantslope}, this is a negative horizontal surgery curve. See \Cref{fig:closedorbitconstantslopearg} for a schematic summary of this argument.
\end{proof}

\begin{figure}
    \centering
    \fontsize{8pt}{8pt}\selectfont
    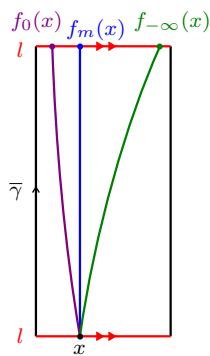
    \caption{We consider the first return maps $f_m$ on a non-separating arc $l$. By the intermediate value theorem, there exists $x \in l$ and $m \in (-\infty,0)$ so that $f_m(x)=x$.}
    \label{fig:closedorbitconstantslopearg}
\end{figure}

\end{eg}

\begin{eg} \label{eg:scalloptoruscurve}
Let $\phi^t$ be a pseudo-Anosov flow on a closed oriented $3$-manifold. Let $T$ be an embedded torus transverse to $\phi^t$. The stable and unstable foliations $\Lambda^s$ and $\Lambda^u$ induce foliations $T^s$ and $T^u$ on $T$, respectively, which intersect transversely. 

Suppose that $T^{s/u}$ are of the following form:
\begin{itemize}
    \item Each of $T^s$ and $T^u$ contains a number of closed leaves.
    \item Each closed leaf of $T^s$ and $T^u$ can be given an orientation, which we refer to as the \textbf{spiraling orientation}, so that
    \begin{itemize}
        \item each non-closed $T^{s/u}$-leaf spirals into a closed $T^{s/u}$-leaf in their spiraling orientations on each of its ends, and
        \item adjacent closed $T^{s/u}$-leaves have opposite spiraling orientations.
    \end{itemize}
    \item The closed $T^s$-leaves are not parallel to the closed $T^u$-leaves.
\end{itemize}
See \Cref{fig:scalloptorusfol} for an illustration of such a pair of foliations $(T^s, T^u)$.

\begin{figure}
    \centering
    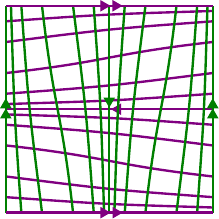
    \caption{The form of the foliations $(T^s, T^u)$ on a scalloped transverse torus.}
    \label{fig:scalloptorusfol}
\end{figure}

We remark that $T^{s/u}$ are of the above form if and only if $T$ projects down to a \textbf{scalloped region} in the orbit space. This follows from the results of \cite{BF13}. We refer to \cite[Section 2]{BFM23} for a nice exposition of the interaction between pseudo-Anosov flows and essential tori. Following the language of \cite{BFM23}, we will refer to $T$ as a \textbf{scalloped} transverse torus if our assumptions above on $T^{s/u}$ hold.

We orient $T$ so that the orbits of $\phi^t$ intersect $T$ positively. This induces an orientation on $H_1(T; \mathbb{R})$. As usual, we consider the homology class of the closed $T^s$-leaves to lie on the vertical axis while that of $T^u$ to lie on the horizontal axis. We say that a homology class lying in the first or third quadrant of $H_1(T; \mathbb{R})$ is \textbf{positive}, while a homology class lying in the second or fourth quadrant is \textbf{negative}. Note that under this terminology, the homology classes of the closed $T^{s/u}$-leaves are both positive and negative.

\begin{prop} \label{prop:scalloptoruscurve}
Let $T$ be a scalloped transverse torus. For every positive/negative primitive integer homology class $\alpha$, there is a positive/negative horizontal surgery curve on $T$, respectively, with homology class $\pm \alpha$.
\end{prop}
\begin{proof}
We prove this for the positive homology classes $\alpha$ that are not equal to the homology class of a closed $T^s$-leaf. The proof uses the same idea as \Cref{prop:closedorbitcurve}. Fix an instantaneous metric. For each $m \in (-\infty,\infty)$, we define a line field $L_m$ on $T$ by requiring that $L_m$ be of slope $m$ at each point. Here we take $L_0 = T^u$.

Consider the universal cover $\widetilde{T}$. Note that the pair of lifted foliations $(\widetilde{T^s}, \widetilde{T^u})$ is homeomorphic to the pair of foliations of $\mathbb{R}^2$ by vertical and horizontal lines respectively. Let $l$ be the lift of a closed $T^s$-leaf. For concreteness, we assume that $\alpha \cdot l$ lies to the right of $\alpha$. Following the trajectories of $L_m$ gives a map $f_m:l \to \alpha \cdot l$. (Here we need $\alpha$ not equal to the homology class of a closed $T^s$-leaf so that $l \neq \alpha \cdot l$.)

The dynamics of $\alpha^{-1} \circ f_0$ is of the following form: 
\begin{itemize}
    \item Let $x_n \in l$, $n \in \mathbb{Z}$, be the intersection points between $l$ and the lifts of the closed $T^u$-leaves, ordered by their positions on $l$. Then for some $k \geq 0$, $\alpha^{-1} \circ f_0$ sends $x_n$ to $x_{n-k}$ for every $n$.
    \item If $k=0$, then $x_n$ are the only fixed points of $\alpha^{-1} \circ f_0$. Furthermore, up to relabelling the indices, $x_{2n}$ are attracting fixed points while $x_{2n+1}$ are repelling fixed points.
\end{itemize}
In particular, there always exists $x$ so that $f_0(x)$ is below $\alpha \cdot x$.

Note that $f_m(x)$ is increasing in $m$. If we can show that $f_m(x)$ lies above $\alpha \cdot x$ for some $m \in (0,\infty)$, then we are done. This is because by the intermediate value theorem, there exists some $m$ so that $f_m(x) = \alpha \cdot x$. The trajectory of $L_m$ between $x$ and $\alpha \cdot x$ projects down to a curve on $T$ of constant slope $m$ and with homology class $\alpha$. By \Cref{lemma:constantslope}, this curve is a positive horizontal surgery curve. See \Cref{fig:scalloptorusconstantslopearg} for a schematic summary of this argument.

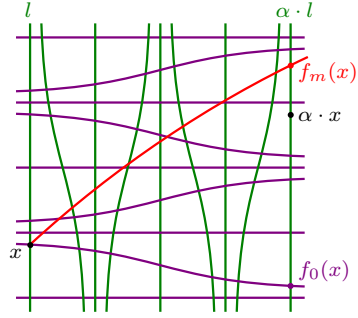
\begin{figure}
    \centering
    \fontsize{8pt}{8pt}\selectfont
\begingroup%
  \makeatletter%
  \providecommand\color[2][]{%
    \errmessage{(Inkscape) Color is used for the text in Inkscape, but the package 'color.sty' is not loaded}%
    \renewcommand\color[2][]{}%
  }%
  \providecommand\transparent[1]{%
    \errmessage{(Inkscape) Transparency is used (non-zero) for the text in Inkscape, but the package 'transparent.sty' is not loaded}%
    \renewcommand\transparent[1]{}%
  }%
  \providecommand\rotatebox[2]{#2}%
  \newcommand*\fsize{\dimexpr\f@size pt\relax}%
  \newcommand*\lineheight[1]{\fontsize{\fsize}{#1\fsize}\selectfont}%
  \ifx\svgwidth\undefined%
    \setlength{\unitlength}{141.61050391bp}%
    \ifx\svgscale\undefined%
      \relax%
    \else%
      \setlength{\unitlength}{\unitlength * \real{\svgscale}}%
    \fi%
  \else%
    \setlength{\unitlength}{\svgwidth}%
  \fi%
  \global\let\svgwidth\undefined%
  \global\let\svgscale\undefined%
  \makeatother%
  \begin{picture}(1,0.80762226)%
    \lineheight{1}%
    \setlength\tabcolsep{0pt}%
    \put(0,0){\includegraphics[width=\unitlength,page=1]{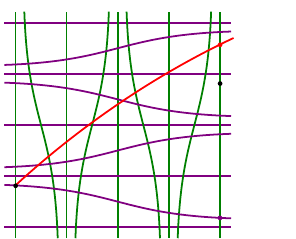}}%
    \put(0.70637237,0.78588572){\color[rgb]{0,0.50196078,0}\makebox(0,0)[lt]{\lineheight{1.25}\smash{\begin{tabular}[t]{l}$\alpha \cdot l$\end{tabular}}}}%
    \put(0.03876536,0.78588236){\color[rgb]{0,0.50196078,0}\makebox(0,0)[lt]{\lineheight{1.25}\smash{\begin{tabular}[t]{l}$l$\end{tabular}}}}%
    \put(-0.00213778,0.14367537){\color[rgb]{0,0,0}\makebox(0,0)[lt]{\lineheight{1.25}\smash{\begin{tabular}[t]{l}$x$\end{tabular}}}}%
    \put(0.76607012,0.09635413){\color[rgb]{0.50196078,0,0.50196078}\makebox(0,0)[lt]{\lineheight{1.25}\smash{\begin{tabular}[t]{l}$f_0(x)$\end{tabular}}}}%
    \put(0.76607012,0.50776177){\color[rgb]{0,0,0}\makebox(0,0)[lt]{\lineheight{1.25}\smash{\begin{tabular}[t]{l}$\alpha \cdot x$\end{tabular}}}}%
    \put(0.7660689,0.62148108){\color[rgb]{1,0,0}\makebox(0,0)[lt]{\lineheight{1.25}\smash{\begin{tabular}[t]{l}$f_m(x)$\end{tabular}}}}%
  \end{picture}%
\endgroup%

    \caption{We consider the first return maps $f_m:l \to \alpha \cdot l$. By the intermediate value theorem, there is a $m \in (0, \infty)$ so that $f_m(x)=\alpha \cdot x$.}
    \label{fig:scalloptorusconstantslopearg}
\end{figure}

To show the statement in the previous paragraph, take a path $c$ on $\widetilde{T}$ from $x$ to $\alpha \cdot x$ that is transverse to $\widetilde{T^s}$. Such a path exists since $\widetilde{T^s}$ is homeomorphic to the foliation of $\mathbb{R}^2$ by vertical lines. Let $m_0$ be the maximum slope of $Tc$. Then the trajectory of $L_{m_0}$ that starts at $x$ must lie above $c$, hence $f_{m_0}(x)$ lies above $\alpha \cdot x$.

To show the proposition when $\alpha$ equals to the homology class of a closed leaf of $T^u$, we can swap the roles of $T^s$ and $T^u$. Alternatively, one can use the argument in \Cref{prop:closedorbitcurve} near a closed leaf of $T^u$.
The proof for negative homology classes is similar.
\end{proof}

\end{eg}

\begin{eg} \label{eg:interiortoruscurve}
We present a class of examples that is similar to \Cref{eg:scalloptoruscurve}. Let $\phi^t$ be a pseudo-Anosov flow on a closed oriented $3$-manifold. Let $T$ be an embedded torus transverse to $\phi^t$. As in \Cref{eg:scalloptoruscurve}, we have induced foliations $T^s$ and $T^u$ on $T$.

In this example, we suppose that $T^{s/u}$ are of the following form:
\begin{itemize}
    \item Each of $T^s$ and $T^u$ contains a number of closed leaves.
    \item Each closed leaf of $T^s$ and $T^u$ can be given an orientation, which we refer to as the \textbf{spiraling orientation}, so that
    \begin{itemize}
        \item each non-closed $T^{s/u}$-leaf spirals into a closed $T^{s/u}$-leaf in their spiraling orientations on each of its ends, and
        \item adjacent closed $T^{s/u}$-leaves have opposite spiraling orientations.
    \end{itemize}
    \item The closed $T^s$-leaves \emph{are} parallel to the closed $T^u$-leaves.
\end{itemize}
See \Cref{fig:interiortorusfol} for an illustration of such a pair of foliations $(T^s, T^u)$.

\begin{figure}
    \centering
    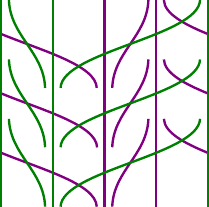
    \caption{The form of the foliations $(T^s, T^u)$ on a non-Reeb non-scalloped transverse torus.}
    \label{fig:interiortorusfol}
\end{figure}

In this case the closed leaves of $T^s$ and $T^u$ must be mutually disjoint, hence be arranged in a cyclic fashion around the torus. We observe that in-between each adjacent pair of closed $T^u$-leaves, there must be an even number of closed $T^s$-leaves, for otherwise those closed $T^u$-leaves would have the same spiraling orientations, and vice versa. 

Let $l^s$ be a closed $T^s$-leaf and $l^u$ be a closed $T^u$-leaf so that $l^s$ and $l^u$ cobound an annulus $A \subset T$ not containing any closed leaves in its interior. If $(T^s, T^u)$ inside $A$ is of the form in \Cref{fig:interiortorussign} left/right, we say that $T$ is a \textbf{positive/negative Reebless non-scalloped transverse torus}, respectively. Note that here we use the orientation on $T$ (induced by the flow) to make a distinction between \Cref{fig:interiortorussign} left/right. Note also that this definition is independent on which pair $(l^s, l^u)$ we choose.

\begin{figure}
    \centering
    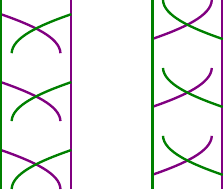
    \caption{Left/right: A portion of the foliation on a positive/negative Reebless non-scalloped transverse torus.}
    \label{fig:interiortorussign}
\end{figure}

In the terminology of \cite{BFM23}, a Reebless non-scalloped transverse torus $T$ is not a `good representative' of its isotopy class; the `good representative' is obtained by homotoping $T$ along flow lines so that it is a weakly embedded union of Birkhoff annuli. Similar to scalloped transverse torus, there is a way to characterize these tori in terms of the chain of lozenges preserved by the action of their $\pi_1$ on the orbit space. Since we will not utilize this perspective, we let the interested reader fill this out for themselves.

What we will prove is the following proposition.
\begin{prop} \label{prop:interiortoruscurve}
Let $T$ be a positive/negative Reebless non-scalloped transverse torus. Let $\alpha$ be a primitive homology class on $T$ that is not parallel to the closed leaves of $T^{s/u}$. Then there is a positive/negative horizontal surgery curve with homology class $\pm \alpha$.
\end{prop}
\begin{proof}
We prove this when $T$ is positive. The proof when $T$ is negative is similar. The argument we use is very similar to \Cref{prop:scalloptoruscurve}. Fix an instantaneous metric. For each $m \in (0,\infty)$, we define a line field $L_m$ on $T$ by requiring that $L_m$ be of slope $m$ at each point.

Consider the universal cover $\widetilde{T}$. Note that each of the lifted foliations $\widetilde{T^{s/u}}$ is homeomorphic to foliation of $\mathbb{R}^2$ by vertical lines. Let $l$ be the lift of a closed $T^s$-leaf. For concreteness, we assume that $\alpha \cdot l$ lies to the right of $\alpha$. Following the trajectories of $L_m$ gives a map $f_m:l \to \alpha \cdot l$. Fix a point $x \in l$. Note that $f_m(x)$ is increasing in $m$.

We claim that $f_m(x)$ lies above $\alpha \cdot x$ for some $m \in (0,\infty)$. The proof of this is exactly the same as in \Cref{prop:scalloptoruscurve}. Symmetrically, $f_m(x)$ lies below $\alpha \cdot x$ for some $m \in (0,\infty)$. Thus by the intermediate value theorem, there exists some $m$ so that $f_m(x) = \alpha \cdot x$. The trajectory of $L_m$ between $x$ and $\alpha \cdot x$ projects down to a curve on $T$ of constant slope $m$ and with homology class $\alpha$. By \Cref{lemma:constantslope}, this curve is a positive horizontal surgery curve.
\end{proof}

\end{eg}

\begin{eg} \label{eg:braidcurve}
Let $\phi^t$ be a pseudo-Anosov flow on a closed oriented $3$-manifold $M$. Suppose $c$ is a positive/negative horizontal surgery curve. In this example, we show how to produce more examples of positive/negative horizontal surgery curves by replacing $c$ by a negative/positive closed braid, respectively.

We recall the relevant terminology from braid theory. Consider the oriented $3$-manifold $N = S^1 \times I \times I$. A \textbf{closed braid} is a curve in $N$ whose image when projected to the first factor of $N$ is a monotone curve in $S^1$. See \Cref{fig:braids} for some examples. We consider two closed braids to be equivalent if they are isotopic through closed braids. A closed braid $\beta$ can be represented by the \textbf{braid diagram} obtained by projecting (generic positions of) $\beta$ to the first two factors $S^1 \times I$, where we also remember the over/under crossing data at each transverse double point, much like a knot diagram.

\begin{figure}
    \centering
    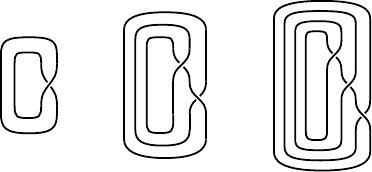
    \caption{Some examples of closed braids.}
    \label{fig:braids}
\end{figure}

A crossing in a braid diagram is said to be \textbf{positive/negative} if it is of the form in \Cref{fig:posnegcrossing} left/right, respectively. A braid is said to be \textbf{positive/negative} if it admits a braid diagram with only positive/negative crossings, respectively.

\begin{figure}
    \centering
\begingroup%
  \makeatletter%
  \providecommand\color[2][]{%
    \errmessage{(Inkscape) Color is used for the text in Inkscape, but the package 'color.sty' is not loaded}%
    \renewcommand\color[2][]{}%
  }%
  \providecommand\transparent[1]{%
    \errmessage{(Inkscape) Transparency is used (non-zero) for the text in Inkscape, but the package 'transparent.sty' is not loaded}%
    \renewcommand\transparent[1]{}%
  }%
  \providecommand\rotatebox[2]{#2}%
  \newcommand*\fsize{\dimexpr\f@size pt\relax}%
  \newcommand*\lineheight[1]{\fontsize{\fsize}{#1\fsize}\selectfont}%
  \ifx\svgwidth\undefined%
    \setlength{\unitlength}{55.2882151bp}%
    \ifx\svgscale\undefined%
      \relax%
    \else%
      \setlength{\unitlength}{\unitlength * \real{\svgscale}}%
    \fi%
  \else%
    \setlength{\unitlength}{\svgwidth}%
  \fi%
  \global\let\svgwidth\undefined%
  \global\let\svgscale\undefined%
  \makeatother%
  \begin{picture}(1,0.96643334)%
    \lineheight{1}%
    \setlength\tabcolsep{0pt}%
    \put(0,0){\includegraphics[width=\unitlength,page=1]{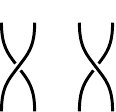}}%
    \put(0.08178085,0.85507963){\color[rgb]{0,0,0}\makebox(0,0)[lt]{\lineheight{1.25}\smash{\begin{tabular}[t]{l}$+$\end{tabular}}}}%
    \put(0.76034698,0.85508511){\color[rgb]{0,0,0}\makebox(0,0)[lt]{\lineheight{1.25}\smash{\begin{tabular}[t]{l}$-$\end{tabular}}}}%
  \end{picture}%
\endgroup%

    \caption{Left/right: A positive/negative crossing, respectively.}
    \label{fig:posnegcrossing}
\end{figure}

\begin{prop} \label{prop:braidcurve}
Let $\phi^t$ be a pseudo-Anosov flow on a closed oriented $3$-manifold $M$. Suppose $c$ is a positive/negative horizontal surgery curve. Fix a neighborhood $N \cong c \times I \times I$ of $c$ where the segments of the form $\{(x_0,y_0)\} \times I$ lie along orbits of $\phi^t$. 

Suppose $\beta$ is a negative/positive closed braid. Then there exists a positive/negative horizontal surgery curve $b$ that is isotopic to the curve obtained by inserting $\beta$ into $N$.
\end{prop}
\begin{proof}
We prove this in the case when $c$ is positive. The case when $c$ is negative is similar. We use the definition of horizontal surgery curves directly in this proof (as opposed to applying \Cref{lemma:constantslope}). 

Let $\pi:b \to c$ be the restriction to $b$ of projecting $N$ to its first factor. By placing $b$ close to $c$, we can assume that for every $x \in b$, the segment of $b$ near $x$ is close to the segment of $c$ near $\pi(x)$ in the $C^1$-topology. In particular, $b$ is a positive curve. To show that $Tb$ is steady, we analyze the crossings of $b$. Note that there are two types of crossings for $b$:
\begin{enumerate}
    \item Crossings $(x,y,t)$ where the orbit segment $\phi^{[0,t]}(x)$ from $x$ to $y$ stays inside $N$.
    \item Crossings $(x,y,t)$ where the orbit segment $\phi^{[0,t]}(x)$ from $x$ to $y$ does not stay inside $N$.
\end{enumerate}

The steadiness condition is satisfied for crossings of type (1) by definition of negative crossings. Compare \Cref{fig:posnegcrossing} with \Cref{fig:steadycurve}.

For the crossings of type (2), we use a similar argument as in \Cref{prop:horsurcurvestable}. Fix a instantaneous metric. 
Since each segment of $b$ is close to a segment of $c$ in the $C^1$-topology, the maximum and minimum slopes of $Tb$ are bounded, so the steadiness condition for $b$ holds for time $\geq T_1$ crossings, for some sufficiently large $T_1$.

Let $T_0$ be a positive number less than the first return time of $N$. We can assume that there are no time $\leq T_0$ crossing of $b$ of type (2).

By \Cref{claim:curveperturbcrossings}, each time $[T_0,T_1]$ crossing of $b$ lies close to a crossing of $c$. For each such crossing, the steadiness condition is $C^1$-open, hence holds for $b$.
\end{proof}

\end{eg}

\begin{eg} \label{eg:bicontactcurve}
In \cite[Proposition 2]{Mit95}, Mitsumatsu shows that if $\phi^t$ is an Anosov flow on a closed oriented $3$-manifold $M$, then there is a positive contact structure $\xi_+$ and a negative contact structure $\xi_-$ such that $T\phi = \xi_+ \cap \xi_-$. Notice that for fixed $x \in M$, the plane fields $d\phi^{-t}(\xi_{\pm}|_{\phi^t(x)})$ rotate monotonically towards $E^s|_x$ as $t$ increases. We refer to \cite[Section 2.2]{ET98} for a nice presentation of this theory of bi-contact structures.

The dynamics on $d\phi^{-t}(\xi_{\pm}|_{\phi^t(x)})$ implies the following proposition.

\begin{prop} \label{prop:bicontactcurve}
Let $c$ be a curve that is tangent to $\xi_{+/-}$. Then $c$ is a negative/positive horizontal surgery curve, respectively.
\end{prop}

\end{eg}

\section{Horizontal Goodman surgery} \label{sec:horsur}

In this section, we explain horizontal Goodman surgery.

\subsection{Surgery operation} \label{subsec:horsurgluing}

For the rest of this section, we fix a pseudo-Anosov flow $\phi^t$ on a closed oriented $3$-manifold $M$. We first define the types of annuli we cut the flow along.

\begin{defn} \label{defn:surann}
A \textbf{positive/negative surgery annulus} is an embedded oriented annulus $A \subset M \backslash \sing(\phi^t)$ that is positively transverse to the flow, along with:
\begin{itemize}
    \item a foliation $\mathcal{H}$ by positive/negative curves such that $T\mathcal{H}$ is steady, and
    \item a foliation $\mathcal{K}$ by negative/positive non-separating arcs such that $T\mathcal{K}$ is steady, respectively.
\end{itemize}
In particular, note that each leaf of $\mathcal{H}$ is a positive/negative horizontal surgery curve, respectively. See \Cref{fig:surann} for an illustration of a positive surgery annulus.

\begin{figure}
    \centering
    \fontsize{8pt}{8pt}\selectfont
\begingroup%
  \makeatletter%
  \providecommand\color[2][]{%
    \errmessage{(Inkscape) Color is used for the text in Inkscape, but the package 'color.sty' is not loaded}%
    \renewcommand\color[2][]{}%
  }%
  \providecommand\transparent[1]{%
    \errmessage{(Inkscape) Transparency is used (non-zero) for the text in Inkscape, but the package 'transparent.sty' is not loaded}%
    \renewcommand\transparent[1]{}%
  }%
  \providecommand\rotatebox[2]{#2}%
  \newcommand*\fsize{\dimexpr\f@size pt\relax}%
  \newcommand*\lineheight[1]{\fontsize{\fsize}{#1\fsize}\selectfont}%
  \ifx\svgwidth\undefined%
    \setlength{\unitlength}{97.73850569bp}%
    \ifx\svgscale\undefined%
      \relax%
    \else%
      \setlength{\unitlength}{\unitlength * \real{\svgscale}}%
    \fi%
  \else%
    \setlength{\unitlength}{\svgwidth}%
  \fi%
  \global\let\svgwidth\undefined%
  \global\let\svgscale\undefined%
  \makeatother%
  \begin{picture}(1,0.72535226)%
    \lineheight{1}%
    \setlength\tabcolsep{0pt}%
    \put(0,0){\includegraphics[width=\unitlength,page=1]{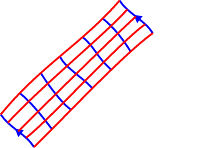}}%
    \put(0.37432506,0.60623593){\color[rgb]{1,0,0}\makebox(0,0)[lt]{\lineheight{1.25}\smash{\begin{tabular}[t]{l}$\mathcal{H}$\end{tabular}}}}%
    \put(0.68836274,0.65606734){\color[rgb]{0,0,1}\makebox(0,0)[lt]{\lineheight{1.25}\smash{\begin{tabular}[t]{l}$\mathcal{K}$\end{tabular}}}}%
  \end{picture}%
\endgroup%

    \caption{A positive surgery annulus.}
    \label{fig:surann}
\end{figure}

A \textbf{parametrized positive/negative surgery annulus} is a positive/negative surgery annulus $(A, \mathcal{H}, \mathcal{K})$ along with an orientation preserving diffeomorphism $\alpha: A \to S^1 \times I$ that sends the leaves of $\mathcal{H}$ to $S^1 \times \{k\}$ and the leaves of $\mathcal{K}$ to $\{h\} \times I$. 
\end{defn}

The following proposition states that surgery annuli can be built around any given horizontal surgery curve.

\begin{prop} \label{prop:horsurcurvetosurann}
Let $c$ be a positive/negative horizontal surgery curve. Then there exists a positive/negative surgery annulus $(A, \mathcal{H}, \mathcal{K})$ where $c$ is one of the leaves of $\mathcal{H}$ in the interior of $A$, respectively.
\end{prop}
\begin{proof}
We prove this when $c$ is positive. The case when $c$ is negative is similar. Take a thin annulus $A \subset M \backslash \sing(\phi^t)$ containing $c$ and transverse to the flow. Take a smooth foliation $\mathcal{H}$ on $A$ by closed curves containing $c$ as a leaf. Up to shrinking $A$, we can assume that $T\mathcal{H}$ is positive. We claim that up to further shrinking $A$, we can assume that $T\mathcal{H}$ is steady. 

The again follows from the same argument as in \Cref{prop:horsurcurvestable}. Fix an instantaneous metric. The maximum and minimum slopes of $T\mathcal{H}$ are bounded, so the steadiness condition holds for $(x,y,t)$ for $t \geq T_1$, for $T_1$ sufficiently large.

Let $T_0$ be the first return time of $A$. It remains to check the steadiness condition for $(x,y,t)$ for $t \in [T_0,T_1]$. By \Cref{claim:curveperturbcrossings}, each triple $(x,y,t) \in A \times A \times [T_0,T_1]$ for which $y=\phi^t(x)$ lies close to a crossing of $c$. For each such crossing, the steadiness condition is $C^1$-open hence holds for $T\mathcal{H}$.

To construct $\mathcal{K}$ on $A$, we can take trajectories of the line field defined by lines of slope, say, $-1$. The leaves of $\mathcal{K}$ will be transverse to those of $\mathcal{H}$ hence are non-separating arcs on $A$. The argument in \Cref{lemma:constantslope} shows that $T\mathcal{K}$ is steady.
\end{proof}

We then specify the types of maps we reglue the cut annuli along.

\begin{defn} \label{defn:surmap}
Let $(A, \mathcal{H}, \mathcal{K}, \alpha)$ be a parametrized positive/negative surgery annulus. A \textbf{surgery map} for $(A, \mathcal{H}, \mathcal{K}, \alpha)$ is a diffeomorphism $\sigma: A \to A$ such that $\sigma=\mathrm{id}$ near $\partial A$ and $\alpha \sigma \alpha^{-1}(h,k)=(h+\rho(k),k)$ where $\rho$ is of the form:
\begin{itemize}
    \item $\rho$ is non-increasing/non-decreasing, respectively,
    \item $\rho'(k) = R_0$ on a interval $J \subset I$, for some negative/positive number $R_0$, respectively, and
    \item $|\rho'(k)| \leq |R_0|$ on all of $I$.
\end{itemize}
Notice that $\alpha=\mathrm{id}$ near $\partial A$ implies that $\rho$ is integer valued near $\partial I$. We call $\rho(0)-\rho(1) \in \mathbb{Z}$ the \textbf{coefficient} of $\alpha$. Note that a surgery map for a positive/negative surgery annulus has positive/negative coefficient, respectively.

We say that $\rho$ is \textbf{$(\delta, R)$-thin} if the length of each component of $\rho(I \backslash J)$ is $<\delta$ and $|R_0| > R$. See \Cref{fig:horsurmapgraph}.
Intuitively, a surgery map twists around the annulus, and for $\delta$ small and $R$ large, a $(\delta, R)$-thin surgery map performs the twist within a thin sub-annulus.

\begin{figure}
    \centering
    \fontsize{10pt}{10pt}\selectfont
\begingroup%
  \makeatletter%
  \providecommand\color[2][]{%
    \errmessage{(Inkscape) Color is used for the text in Inkscape, but the package 'color.sty' is not loaded}%
    \renewcommand\color[2][]{}%
  }%
  \providecommand\transparent[1]{%
    \errmessage{(Inkscape) Transparency is used (non-zero) for the text in Inkscape, but the package 'transparent.sty' is not loaded}%
    \renewcommand\transparent[1]{}%
  }%
  \providecommand\rotatebox[2]{#2}%
  \newcommand*\fsize{\dimexpr\f@size pt\relax}%
  \newcommand*\lineheight[1]{\fontsize{\fsize}{#1\fsize}\selectfont}%
  \ifx\svgwidth\undefined%
    \setlength{\unitlength}{151.415672bp}%
    \ifx\svgscale\undefined%
      \relax%
    \else%
      \setlength{\unitlength}{\unitlength * \real{\svgscale}}%
    \fi%
  \else%
    \setlength{\unitlength}{\svgwidth}%
  \fi%
  \global\let\svgwidth\undefined%
  \global\let\svgscale\undefined%
  \makeatother%
  \begin{picture}(1,0.65889929)%
    \lineheight{1}%
    \setlength\tabcolsep{0pt}%
    \put(0,0){\includegraphics[width=\unitlength,page=1]{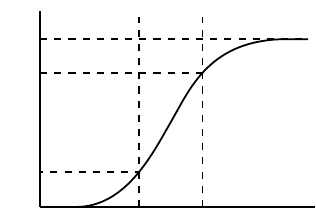}}%
    \put(-0.00199136,0.47285281){\color[rgb]{0,0,0}\makebox(0,0)[lt]{\lineheight{1.25}\smash{\begin{tabular}[t]{l}$< \delta$\end{tabular}}}}%
    \put(0.383417,0.63857036){\color[rgb]{0,0,0}\makebox(0,0)[lt]{\lineheight{1.25}\smash{\begin{tabular}[t]{l}$\rho'=R_0>R$\end{tabular}}}}%
    \put(-0.00199936,0.04887937){\color[rgb]{0,0,0}\makebox(0,0)[lt]{\lineheight{1.25}\smash{\begin{tabular}[t]{l}$< \delta$\end{tabular}}}}%
  \end{picture}%
\endgroup%

    \caption{A choice of $\rho$ for a $(\delta, R)$-thin surgery map on a negative surgery annulus.}
    \label{fig:horsurmapgraph}
\end{figure}
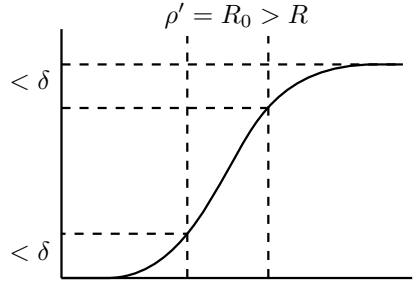

\end{defn}

We can now state the horizontal Goodman surgery operation on pseudo-Anosov flows.

\begin{constr} \label{constr:sur}
Let $\phi^t$ be a pseudo-Anosov flow on a closed oriented $3$-manifold $M$. Let $(A, \mathcal{H}, \mathcal{K}, \alpha)$ be a parametrized positive/negative surgery annulus, and let $\sigma:A \to A$ be a surgery map for $(A, \mathcal{H}, \mathcal{K}, \alpha)$.

Let $M \cut A$ be the space obtained by cutting $M$ along $A$. There are two copies of $A$ on $M \cut A$ --- one on the positive side of $A$ and the other on the negative side of $A$. We denote these by $A_+$ and $A_-$ respectively. Note that $\partial A_+ = \partial A_-$. 

We glue $A_-$ to $A_+$ via $\sigma: A_- \to A_+$. Let $M_\sigma(A, \mathcal{H}, \mathcal{K}, \alpha)$ be the resulting $3$-manifold, and let $\phi^t_\sigma(A, \mathcal{H}, \mathcal{K}, \alpha)$ be the resulting flow. 
\end{constr}

\subsection{Pseudo-Anosovity} \label{subsec:horsurpA}

Our goal in this subsection is to prove that for sufficiently thin surgery maps, the flow defined in \Cref{constr:sur} is pseudo-Anosov.

\begin{thm} \label{thm:horsurpA}
Let $\phi^t$ be a pseudo-Anosov flow on a closed oriented $3$-manifold $M$. Let $(A, \mathcal{H}, \mathcal{K}, \alpha)$ be a positive/negative surgery annulus. Then there exists $\delta$ and $R$ such that for every $(\delta, R)$-thin surgery map $\sigma$ for $(A, \mathcal{H}, \mathcal{K}, \alpha)$, the flow $\phi^t_\sigma(A, \mathcal{H}, \mathcal{K}, \alpha)$ is pseudo-Anosov.
\end{thm}
\begin{proof}
We prove the theorem in the case when $(A, \mathcal{H}, \mathcal{K})$ is positive. The case when $(A, \mathcal{H}, \mathcal{K})$ is negative is similar. For convenience, we write $\overline{M}$ for $M_\sigma(A, \mathcal{H}, \mathcal{K})$ and $\overline{\phi}^t$ for $\phi^t_\sigma(A, \mathcal{H}, \mathcal{K})$. 

There are two parts to this proof:
\begin{enumerate}
    \item We show that there is a continuous splitting of the tangent bundle into three $\overline{\phi}^t$-line bundles $T\overline{M} = \overline{E}^s \oplus T \overline{\phi} \oplus \overline{E}^s$ such that $\overline{E}^s$ and $\overline{E}^u$ have the appropriate contraction and expansion properties. This is done in two steps.
    \begin{enumerate}
        \item We construct cone fields $C^{cs}$ and $C^{cu}$ using $T\mathcal{H}$ and $T\mathcal{K}$ and show that they contract exponentially under $\overline{\phi}^t$ to conclude that they limit to `weak stable' and `weak unstable' plane fields $\overline{E}^{cs}$ and $\overline{E}^{cu}$.
        \item We construct cone fields $C^s$ and $C^u$ on $\overline{E}^{cs}$ and $\overline{E}^{cu}$, and show that they contract exponentially under $\overline{\phi}^t$ to conclude that they limit to the line fields $\overline{E}^{cs}$ and $\overline{E}^{cu}$, respectively.
    \end{enumerate}
    \item We then show that each singular orbit $\gamma$ has a neighborhood that is equivalent to a neighborhood of a pseudo-hyperbolic orbit. This is done by finding an appropriate conjugating map between the first return maps on local sections.
\end{enumerate}

The cut manifold $M \cut A$ will play a large role in part (1), so we take the time to set up some notation. We identify points on $M \cut A$ that lie away from $A_{\pm}$ with $M \backslash A$, and we denote points on $A_{\pm}$ by $x_{\pm}$, where $x$ is the corresponding point on $A$. We will also consider the restriction of $\phi^t$ to $M \cut A$. We abuse notation and also write this restricted flow as $\phi^t$, even though $\phi^t(x)$ is only defined from the point where the orbit begins from $A_+$ to the point where the orbit ends at $A_-$.

The flow $\overline{\phi}^t$ on $\overline{M}$ can be thought of as the `composition' of $\phi^t$ and $\sigma$, the latter coming up whenever an orbit ends at $A_-$ and has to start anew on $A_+$. In particular, we can analyze $d\overline{\phi}^t$ by taking the composition of $d\phi^t$ and $d\sigma$. We prove the following claim that gives us a starting point for analyzing the latter.

\begin{claim} \label{claim:thinsurmapsudyn}
Fix an instantaneous metric. Let $e^{s/u}$ be unit length vector fields spanning $E^{s/u}|_A$ respectively. The matrix representative for $d\sigma:TM|_{A_-} \to TM|_{A_+}$ in the bases $(\dot{\phi}, e^s, e^u)$ is of the form
$$\begin{bmatrix}
1 & S & U \\
0 & m & n \\
0 & p & q \\
\end{bmatrix}$$
where $S,U,m,n,p,q$ are smooth functions on $A$. 

Furthermore, for every $\epsilon>0$, there exists $\delta$ and $R$ so that whenever $\sigma$ is $(\delta, R)$-thin, we have $q > 1-\epsilon$ on $A$.
\end{claim}
\begin{proof}
The first statement is clear since $\sigma$ sends $\dot{\phi}$ to $\dot{\phi}$.

To prove the second statement, it suffices to consider $d\sigma$ as a map on $(TM/T\phi)|_A \cong TA$. 
We perform the following computations in $S^1 \times I$ via $\alpha$. 
Up to reversing the signs of $e^s$ or $e^u$, we can write
\begin{align*}
\partial_h &= a(h,k) e^s + b(h,k) e^u \\
\partial_k &= c(h,k) e^s - d(h,k) e^u
\end{align*}
for positive functions $a, b, c, d$.

Hence the matrix representation of $d\sigma$ under the basis $(e^s, e^u)$ is
\begin{scriptsize}
\begin{align*}
& \begin{bmatrix}
a(h+\rho(k),k) & c(h+\rho(k),k) \\
b(h+\rho(k),k) & -d(h+\rho(k),k) 
\end{bmatrix}
\begin{bmatrix}
1 & \rho'(k) \\
0 & 1
\end{bmatrix}
{\begin{bmatrix}
a(h,k) & c(h,k) \\
b(h,k) & -d(h,k) 
\end{bmatrix}}^{-1} \\
=&
\begin{bmatrix}
a(h+\rho(k),k) & a(h+\rho(k),k)\rho'(k)+c(h+\rho(k),k) \\
b(h+\rho(k),k) & b(h+\rho(k),k)\rho'(k)-d(h+\rho(k),k) 
\end{bmatrix}
\frac{1}{a(h,k)d(h,k)+b(h,k)c(h,k)} \begin{bmatrix}
d(h,k) & c(h,k) \\
b(h,k) & -a(h,k) 
\end{bmatrix} \\
=&
\frac{1}{a(h,k)d(h,k)+b(h,k)c(h,k)} \begin{bmatrix}
* & * \\
* & a(h,k)d(h+\rho(k),k)+b(h+\rho(k),k)c(h,k)-a(h,k)b(h+\rho(k),k)\rho'(k)
\end{bmatrix} \\
\end{align*}
\end{scriptsize}

Within $S^1 \times J$, where $\rho'(k) < -R$,
\begin{align*}
& \frac{a(h,k)d(h+\rho(k),k)+b(h+\rho(k),k)c(h,k)-a(h,k)b(h+\rho(k),k)\rho'(k)}{a(h,k)d(h,k)+b(h,k)c(h,k)} \\
>& R \frac{a(h,k)b(h+\rho(k),k)}{a(h,k)d(h,k)+b(h,k)c(h,k)} > 1
\end{align*}
for large $R$.

Within $S^1 \times (I \backslash J)$, where the size of $\rho(k)$ mod $\mathbb{Z}$ is $< \delta$, notice that 
\begin{align*}
& \frac{a(h,k)d(h+\rho(k),k)+b(h+\rho(k),k)c(h,k)}{a(h,k)d(h,k)+b(h,k)c(h,k)} \to 1
\end{align*}
as $\rho(k) \to 0$, hence 
\begin{align*}
& \frac{a(h,k)d(h+\rho(k),k)+b(h+\rho(k),k)c(h,k)-a(h,k)b(h+\rho(k),k)\rho'(k)}{a(h,k)d(h,k)+b(h,k)c(h,k)} \\
>& 1-\epsilon-\frac{a(h,k)b(h+\rho(k),k)\rho'(k)}{a(h,k)d(h,k)+b(h,k)c(h,k)} \geq 1-\epsilon \\
\end{align*}
for small $\delta$.
\end{proof}

We begin part (1) by defining a cone field $C^{cu}$ on $(M \cut A) \backslash \sing(\phi^t)$ as follows:
\begin{itemize}
    \item If $y = x_+ \in A_+$, then $C^{cu}|_y$ is the union of the two opposite (closed) quadrants bounded by $T\mathcal{H}|_x$ and $T\mathcal{K}|_x$ that meet $E^u$, direct sum with $T\phi|_x$. See \Cref{fig:surpAcone}.
    \item If $y = \phi^t(x_+)$ for $x_+ \in A_+$ and $t>0$, then $C^{cu}|_y = d\phi^t(C^{cu}_{x_+})$.
    \item If $y \notin \phi^{[0,\infty)}(A_+)$, then $C^{cu}|_y = E^u|_y \oplus T\phi|_y$.
\end{itemize}

\begin{figure}
    \centering
    \fontsize{10pt}{10pt}\selectfont
\begingroup%
  \makeatletter%
  \providecommand\color[2][]{%
    \errmessage{(Inkscape) Color is used for the text in Inkscape, but the package 'color.sty' is not loaded}%
    \renewcommand\color[2][]{}%
  }%
  \providecommand\transparent[1]{%
    \errmessage{(Inkscape) Transparency is used (non-zero) for the text in Inkscape, but the package 'transparent.sty' is not loaded}%
    \renewcommand\transparent[1]{}%
  }%
  \providecommand\rotatebox[2]{#2}%
  \newcommand*\fsize{\dimexpr\f@size pt\relax}%
  \newcommand*\lineheight[1]{\fontsize{\fsize}{#1\fsize}\selectfont}%
  \ifx\svgwidth\undefined%
    \setlength{\unitlength}{93.40295026bp}%
    \ifx\svgscale\undefined%
      \relax%
    \else%
      \setlength{\unitlength}{\unitlength * \real{\svgscale}}%
    \fi%
  \else%
    \setlength{\unitlength}{\svgwidth}%
  \fi%
  \global\let\svgwidth\undefined%
  \global\let\svgscale\undefined%
  \makeatother%
  \begin{picture}(1,0.91045712)%
    \lineheight{1}%
    \setlength\tabcolsep{0pt}%
    \put(0,0){\includegraphics[width=\unitlength,page=1]{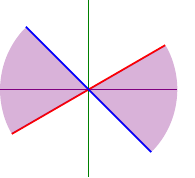}}%
    \put(0.64499879,0.69324246){\color[rgb]{1,0,0}\makebox(0,0)[lt]{\lineheight{1.25}\smash{\begin{tabular}[t]{l}$T\mathcal{H}$\end{tabular}}}}%
    \put(0.52301641,0.12008146){\color[rgb]{0,0,1}\makebox(0,0)[lt]{\lineheight{1.25}\smash{\begin{tabular}[t]{l}$T\mathcal{K}$\end{tabular}}}}%
  \end{picture}%
\endgroup%

    \caption{Defining the cone $C^{cu}_y$ if $y = x_+ \in A_+$.}
    \label{fig:surpAcone}
\end{figure}

By identifying $\overline{M}$ with the complement of the interior of $A_-$ in $M \cut A$, $C^{cu}$ descends to a discontinuous cone field $\overline{C}^{cu}$ on $\overline{M}$. We want to show that
\begin{enumerate}[label=(\roman*)]
    \item If $y=\overline{\phi}^t(x)$ where $t \geq 0$, then $\overline{C}^{cu}|_y \supset d\overline{\phi}^t(\overline{C}^{cu}|_x)$.
    \item For every $x$, $\bigcap_{t \in [0,\infty)} d\overline{\phi}^t(\overline{C}^{cu}|_{\overline{\phi}^{-t}(x)})$ is a plane $\overline{E}^{cu}|_x$ at $x$.
    \item $\overline{E}^{cu}$ is a continuous plane field.
\end{enumerate}

For (i), it suffices to show that
\begin{itemize}
    \item if $y=\phi^t(x)$ where $t \geq 0$ in $M \cut A$, then $C^{cu}|_y = d\phi^t(C^{cu}|_x)$, and
    \item $C^{cu}|_{x_+} \supset d\sigma(C^{cu}|_{x_-})$.
\end{itemize}
The first point follows from the definition of $C^{cu}$. For the second point, the steadiness of $T\mathcal{H}$ and $T\mathcal{K}$ implies that $C^{cu}|_{x_-} \cap TA_-$ lies within the union of quadrants bounded by $T\mathcal{H}|_x$ and $T\mathcal{K}|_x$ that contain $A^u$. These are then sent into themselves under $\sigma$, by definition of gluing maps. See \Cref{fig:surpAarg} for a pictorial summary.

\begin{figure}
    \centering
    \fontsize{10pt}{10pt}\selectfont
\begingroup%
  \makeatletter%
  \providecommand\color[2][]{%
    \errmessage{(Inkscape) Color is used for the text in Inkscape, but the package 'color.sty' is not loaded}%
    \renewcommand\color[2][]{}%
  }%
  \providecommand\transparent[1]{%
    \errmessage{(Inkscape) Transparency is used (non-zero) for the text in Inkscape, but the package 'transparent.sty' is not loaded}%
    \renewcommand\transparent[1]{}%
  }%
  \providecommand\rotatebox[2]{#2}%
  \newcommand*\fsize{\dimexpr\f@size pt\relax}%
  \newcommand*\lineheight[1]{\fontsize{\fsize}{#1\fsize}\selectfont}%
  \ifx\svgwidth\undefined%
    \setlength{\unitlength}{105.5377822bp}%
    \ifx\svgscale\undefined%
      \relax%
    \else%
      \setlength{\unitlength}{\unitlength * \real{\svgscale}}%
    \fi%
  \else%
    \setlength{\unitlength}{\svgwidth}%
  \fi%
  \global\let\svgwidth\undefined%
  \global\let\svgscale\undefined%
  \makeatother%
  \begin{picture}(1,1.2290764)%
    \lineheight{1}%
    \setlength\tabcolsep{0pt}%
    \put(0,0){\includegraphics[width=\unitlength,page=1]{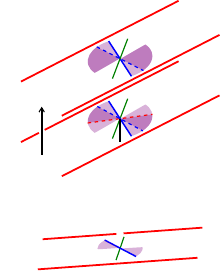}}%
    \put(0.21974641,0.57100777){\color[rgb]{0,0,0}\makebox(0,0)[lt]{\lineheight{1.25}\smash{\begin{tabular}[t]{l}$\sigma$\end{tabular}}}}%
    \put(0.57340539,0.30075651){\color[rgb]{0,0,0}\makebox(0,0)[lt]{\lineheight{1.25}\smash{\begin{tabular}[t]{l}$\phi^t$\end{tabular}}}}%
    \put(-0.0028685,0.06603881){\color[rgb]{1,0,0}\makebox(0,0)[lt]{\lineheight{1.25}\smash{\begin{tabular}[t]{l}$A_+$\end{tabular}}}}%
    \put(-0.0028685,0.69358502){\color[rgb]{1,0,0}\makebox(0,0)[lt]{\lineheight{1.25}\smash{\begin{tabular}[t]{l}$A_+$\end{tabular}}}}%
    \put(-0.0028685,0.43893953){\color[rgb]{1,0,0}\makebox(0,0)[lt]{\lineheight{1.25}\smash{\begin{tabular}[t]{l}$A_-$\end{tabular}}}}%
    \put(0,0){\includegraphics[width=\unitlength,page=2]{surpAarg.pdf}}%
  \end{picture}%
\endgroup%

    \caption{A pictorial summary of the proof that $\overline{C}^{cu}|_y \supset d\overline{\phi}^t(\overline{C}^{cu}|_x)$.}
    \label{fig:surpAarg}
\end{figure}

For (ii), we use slopes in order to make a quantitative argument. For $x \in M \cut A$ and a cone $C$ in $TM|_x$ containing $T\phi|_x$ and disjoint from $E^s|_x$, we let $\width_L(C)$ be the difference between $L$ times the maximum slope and $L$ times the minimum slope of lines contained in $C$. Here $L$ is some large number to be chosen later.

Note that 
\begin{itemize}
    \item If $y=\phi^t(x)$ where $t \geq 0$ in $M \cut A$, then $\width_L(d\phi^t(C^{cu}|_x)) \leq \lambda^{-2t} \width_L(C^{cu}|_x)$. 
    \item In the notation of \Cref{claim:thinsurmapsudyn}, 
    \begin{align*}
        \width_L(d\sigma(C^{cu}|_{x_-})) =& \left| \frac{mL\mathrm{maxslope}(C^{cu}|_{x_-})+n}{pL\mathrm{maxslope}(C^{cu}|_{x_-})+q} - \frac{mL\mathrm{minslope}(C^{cu}|_{x_-})+n}{pL\mathrm{minslope}(C^{cu}|_{x_-})+q} \right| \\
        =& \frac{(mq-np)L|\mathrm{maxslope}(C^{cu}|_{x_-})-\mathrm{minslope}(C^{cu}|_{x_-})|}{|pL\mathrm{maxslope}(C^{cu}|_{x_-})+q||pL\mathrm{minslope}(C^{cu}|_{x_-})+q|} \\
        =& \frac{(mq-np)\width_L(C^{cu}|_{x_-})}{|pL\mathrm{maxslope}(C^{cu}|_{x_-})+q||pL\mathrm{minslope}(C^{cu}|_{x_-})+q|} \\
        <& \width_L(C^{cu}|_{x_-})
    \end{align*}
    for large $L$. Here we used the fact that the slopes of $T\mathcal{H}$ and $T\mathcal{K}$ are uniformly bounded, which is in turn due to compactness of $A$.
\end{itemize}
 
This implies that $\width_L(d\overline{\phi}^t(\overline{C}^{cu}|_{\overline{\phi}^{-t}(x)})) \to 0$ as $t \to \infty$, giving (2).

By pulling back the plane field $\overline{E}^{cu}$ on $\overline{M}$ to $M \cut A$, we have a plane field whose restriction to $A_-$ maps to its restriction to $A_+$ under $\sigma$. Abusing notation, we denote this plane field as $\overline{E}^{cu}$ as well.

To show (3), suppose for the sake of contradiction that there exists $x_n \in \overline{M}$ converging to $x \in \overline{M}$ such that $\overline{E}^{cu}|_{x_n}$ are bounded away from $\overline{E}^{cu}|_x$. We lift this to $M \cut A$. By uniformity of the convergence in (2), for large enough $t$, $d\overline{\phi}^t(\overline{C}^{cu}|_{\overline{\phi}^{-t}(x_n)})$ are bounded away from $d\overline{\phi}^t(\overline{C}^{cu}|_{\overline{\phi}^{-t}(x)})$. This implies that $\overline{C}^{cu}|_{\overline{\phi}^{-t}(x_n)}$ are bounded away from $\overline{C}^{cu}|_{\overline{\phi}^{-t}(x)}$. But this is false since $E^u|_{\overline{\phi}^{-t}(x_n)} \subset \overline{C}^{cu}|_{\overline{\phi}^{-t}(x_n)}$ converges to $E^u|_{\overline{\phi}^{-t}(x)} \subset \overline{C}^{cu}|_{\overline{\phi}^{-t}(x)}$.

This establishes the existence of the continuous plane field $\overline{E}^{cu}$. 
Note that $\overline{E}^{cu} \subset C^{cu}$ in $M \cut A$. In particular the slopes of lines contained in $\overline{E}^{cu}$ are uniformly bounded. 

We will now find $\overline{E}^u \subset \overline{E}^{cu}$. The argument for this is similar to (2) above in that we use slopes to quantify the dynamics on $\overline{E}^{cu}$ (but there are some subtle differences which we point out in \Cref{rmk:horsurpAproof}). Let us first define a notion of slope for lines in $\overline{E}^u$: Let $x \in (M \cut A) \backslash \sing(\phi^t)$. Let $e^{s/u}$ be unit length vectors in $E^{s/u}|_x$ respectively. A nonzero vector in a line $\ell$ in $\overline{E}^{cu}|_x$ can be written as $a \dot{\phi} + b (c e^s + e^u)$. We define the \textbf{$cu$-slope} of $\ell$ to be $\frac{a}{b}$. 
For a cone $C$ in $\overline{E}^u|_x$, we let $\width(C)$ be the difference between the maximum $cu$-slope and the minimum $cu$-slope of lines contained in $C$. Note that while the $cu$-slope is only well-defined up to a sign, the width of a cone is well-defined.

If $y = \phi^t(x)$ where $t \geq 0$ in $M \cut A$, then since $d\phi^t(\dot{\phi}|_x) = \dot{\phi}|_y$ and $||d\phi^t(e^u)|| \leq \lambda^{-t}$, we see that $d\phi^t$ acts on the $cu$-slopes by multiplication by a factor $\leq \lambda^{-t}$.

To analyze the effect of $d\sigma$ on $cu$-slopes, we use \Cref{claim:thinsurmapsudyn}: The $cu$-slope of 
$$d\sigma(a \dot{\phi} + b (c e^s + e^u)) = (a+bcS+bU) \dot{\phi} + bq (c e^s + e^u)$$
is $\frac{1}{q} (\frac{a}{b} + (cS+U))$. For large enough $R$ and small enough $\delta$, $q$ is bounded from below by $1-\epsilon$ with $\epsilon$ arbitrarily close to $0$. Also, $S$ and $U$ are uniformly bounded by compactness of $A$, and $c$ is uniformly bounded by the boundedness of slopes of lines in $\overline{E}^{cu}$. This implies that $d\sigma:\overline{E}^{cu}|_{x_-} \to \overline{E}^{cu}|_{x_+}$ acts on the set of slopes by multiplication by a factor $< \frac{1}{1-\epsilon}$ then translation by $\leq m$.

Now let $T_0$ be the first return time of $A$.
Take $\epsilon$ small enough so that $\displaystyle \frac{\lambda^{-\frac{T_0}{2}}}{1-\epsilon} < 1$. Define a cone field $C^u$ in $\overline{E}^{cu}$ as follows:
\begin{itemize}
    \item If $y = \phi^t(x_+)$ for $x_+ \in A_+$ and $t \geq 0$, then $C^u|_y$ is the union of lines of $cu$-slope $[-\overline{m} \lambda^{-\frac{t}{2}}, \overline{m} \lambda^{-\frac{t}{2}}]$, where $\displaystyle \overline{m} = m \left( 1-\frac{\lambda^{-\frac{T_0}{2}}}{1-\epsilon} \right)^{-1}$.
    \item If $y \notin \phi^{[0,\infty)}(A_+)$, then $C^u|_y = E^u|_y$.
\end{itemize}

By again identifying $\overline{M}$ with the complement of the interior of $A_-$ in $M \cut A$, $C^u$ descends to a discontinuous cone field $\overline{C^u}$ on $\overline{M}$. 

Our analysis on the actions of $d\phi^t$ and $d\sigma$ on the $cu$-slope implies that 
\begin{itemize}
    \item if $y=\phi^t(x)$ where $t \geq 0$ in $M \cut A$, then $C^u|_y \supset d\phi^t(C^u|_x)$ and $\width(d\phi^t(C^u|_x)) \leq \lambda^{-t} \width(C^u|_x)$, and
    \item $C^u|_{x_+} \supset d\sigma(C^u|_{x_-})$ and $\width(d\sigma(C^u|_{x_-})) < \width(C^u|_{x_-})$
\end{itemize}

These imply
\begin{enumerate}[label=(\roman*)]
    \item If $y=\overline{\phi}^t(x)$ where $t \geq 0$, then $\overline{C^u}|_y \supset d\overline{\phi}^t(\overline{C^u}|_x)$.
    \item For every $x$, $\bigcap_{t \in [0,\infty)} d\overline{\phi}^t(\overline{C^u}|_{\overline{\phi}^{-t}(x)})$ is a line $\overline{E^u}|_x$ at $x$.
\end{enumerate}

A similar argument as for $\overline{E}^{cu}$ shows that 
\begin{enumerate}[label=(\roman*), resume]
    \item $\overline{E}^u$ is a continuous line field.
\end{enumerate}

This establishes the existence of the continuous line field $\overline{E}^u$. Similarly as $\overline{E^{cu}}$, we have that $\overline{E^u} \subset C^u$, hence the $cu$-slopes $\overline{E^u}$ are uniformly bounded. In particular the length of a vector in $\overline{E^u}$ differ from the length of its projection to $E^u$ along $E^s \oplus T\phi$ by a bounded factor. Since $d\phi^t$ on $M \cut A$ expands $E^u$ by a factor of at least $\lambda^t$, and by \Cref{claim:thinsurmapsudyn} the $e^u$ coordinate of $d\sigma(e^u)$ is $>1-\epsilon>\lambda^{-\frac{T_0}{2}}$, we conclude that $d\overline{\phi^t}$ expands $\overline{E}^u$ exponentially. 

Symmetrically, one finds $\overline{E}^s$. In particular $\overline{E}^s$ lies in the union of the two opposite (closed) quadrants bounded by $T\mathcal{H}|_x$ and $T\mathcal{K}|_x$ that meet $E^s$, direct sum with $T\phi|_x$. Hence $T\overline{M} = \overline{E}^s \oplus T \overline{\phi}^t \oplus \overline{E}^s$. This concludes part (1) of the proof.

For part (2), notice that $\dot{\phi} = \dot{\overline{\phi}}$ in a neighborhood of $\sing(\phi^t)$. In particular, the contraction/expansion properties of $\overline{E}^{s/u}$ force $\overline{E}^{s/u} = E^{s/u}$ on the local stable/unstable leaves of each singular orbit, respectively. (Alternatively, this also follows from our construction of $\overline{E}^{s/u}$.)

The difficulty however is that $\overline{E}^{s/u}$ and $E^{s/u}$ will generally differ away from these singular leaves. What we will do is to find a homeomorphism $H:\nu \to \overline{\nu}$ between neighborhoods of $\sing(\phi^t)$ such that
\begin{itemize}
    \item $H$ is bi-Lipschitz and smooth away from $\sing(\phi^t)$,
    \item $H$ preserves the orbits,
    \item $H$ sends leaves of the stable/unstable foliation $\Lambda^{s/u}$ of $\phi^t$ to leaves of the stable/unstable foliation $\overline{\Lambda}^{s/u}$ of $\overline{\phi}^t$ respectively, and
    \item $H$ is identity on the local leaves of the singular orbits.
\end{itemize}
Once we do so, the local equivalences between $\sing(\overline{\phi}^t) = \sing(\phi^t)$ and pseudo-hyperbolic orbits are obtained by composing those for $\phi^t$ with $H$.

Consider a local section $P$ to $\phi^t$ in a neighborhood of a singular orbit $\gamma$. We have induced singular foliations $P^{s/u} = P \cap \Lambda^{s/u}$ on $P$. Let $\mathfrak{p}^{s/u}$ be the singular leaf of $P^{s/u}$, respectively. Notice that $\phi^t$ induces a first return map $g:Q_1 \to Q_2$ where $Q_i$ are sub-surfaces of the local section containing $P \cap \gamma$. In fact, if $\gamma$ is $n$-pronged, we can pick $Q_i$ to be a $2n$-gon with sides on $\Lambda^s$ and $\Lambda^u$ alternatingly. See \Cref{fig:surpAsing} left. In this case we say that $Q_i$ is the \textbf{saturation} of $Q_i \cap \mathfrak{p}^s$ and $Q_i \cap \mathfrak{p}^u$, in the sense that every point in $Q_i$ is the intersection between the $P^u$-leaf through a point on $Q_i \cap \mathfrak{p}^s$ and the $P^s$-leaf through a point on $Q_i \cap \mathfrak{p}^u$.

Now we have another set of induced singular foliations $\overline{P}^{s/u} = P \cap \overline{\Lambda}^{s/u}$ on $P$. Note that the singular leaves of these equal $\mathfrak{p}^{s/u}$. Let $\overline{Q}_i$ be the saturation of $Q_i \cap \mathfrak{p}^s$ and $Q_i \cap \mathfrak{p}^u$, i.e. we take every point that is the intersection between the $\overline{P}^u$-leaf through a point on $Q_i \cap \mathfrak{p}^s$ and the $\overline{P}^u$-leaf through a point on $Q_i \cap \mathfrak{p}^u$. See \Cref{fig:surpAsing} right.

\begin{figure}
    \centering
\begingroup%
  \makeatletter%
  \providecommand\color[2][]{%
    \errmessage{(Inkscape) Color is used for the text in Inkscape, but the package 'color.sty' is not loaded}%
    \renewcommand\color[2][]{}%
  }%
  \providecommand\transparent[1]{%
    \errmessage{(Inkscape) Transparency is used (non-zero) for the text in Inkscape, but the package 'transparent.sty' is not loaded}%
    \renewcommand\transparent[1]{}%
  }%
  \providecommand\rotatebox[2]{#2}%
  \newcommand*\fsize{\dimexpr\f@size pt\relax}%
  \newcommand*\lineheight[1]{\fontsize{\fsize}{#1\fsize}\selectfont}%
  \ifx\svgwidth\undefined%
    \setlength{\unitlength}{191.99092835bp}%
    \ifx\svgscale\undefined%
      \relax%
    \else%
      \setlength{\unitlength}{\unitlength * \real{\svgscale}}%
    \fi%
  \else%
    \setlength{\unitlength}{\svgwidth}%
  \fi%
  \global\let\svgwidth\undefined%
  \global\let\svgscale\undefined%
  \makeatother%
  \begin{picture}(1,0.47129608)%
    \lineheight{1}%
    \setlength\tabcolsep{0pt}%
    \put(0,0){\includegraphics[width=\unitlength,page=1]{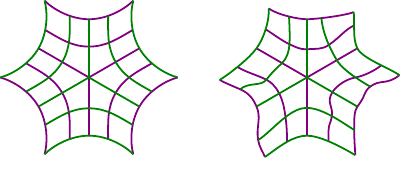}}%
    \put(0.19436453,0.00399929){\color[rgb]{0,0,0}\makebox(0,0)[lt]{\lineheight{1.25}\smash{\begin{tabular}[t]{l}$Q_i$\end{tabular}}}}%
    \put(0.72788345,0.00399799){\color[rgb]{0,0,0}\makebox(0,0)[lt]{\lineheight{1.25}\smash{\begin{tabular}[t]{l}$\overline{Q}_i$\end{tabular}}}}%
  \end{picture}%
\endgroup%

    \caption{The local sections $Q_i$ and $\overline{Q_i}$ that we use to establish the local model for $\sing(\overline{\phi}^t)$.}
    \label{fig:surpAsing}
\end{figure}

We can define homeomorphisms $h_i:Q_i \to \overline{Q}_i$ by starting with the identity on $Q_i \cap \mathfrak{p}^{s/u}$ and extending them by requiring that they send $P^{s/u}$-leaves to $\overline{P}^{s/u}$-leaves. Since the foliations $P^{s/u}$ and $\overline{P}^{s/u}$ are smooth, $h_i$ is smooth away from $\gamma \cap P$. Since the slopes of $\overline{E}^{s/u}$ are bounded, $h_i$ is bi-Lipschitz. Moreover, we have $h_1=h_2$ on $Q_1 \cap Q_2$.

Since $\overline{\phi}^t=\phi^t$ preserves $\overline{\Lambda}^{s/u}$, the first return map preserves $\overline{P}^{s/u}$. This implies that $h_2 g = g h_1$. Taking the suspension, we get the desired homeomorphism $H$.
\end{proof}

\begin{rmk} \label{rmk:horsurpAproof}
We make some observations on the proof of \Cref{thm:horsurpA}.

\begin{itemize}
    \item Part (1) of the proof of \Cref{thm:horsurpA} follows many of the same ideas as \cite{Goo83}, but there are some added complications due to our more general setting. The situation considered in \cite{Goo83} is essentially \Cref{eg:closedorbitcurve}. We point out two important features that are special to this setting.
    \begin{enumerate}[label=(\roman*)]
        \item In the setting of \cite{Goo83}, the foliation $\mathcal{H}$ on the surgery annulus $A$ can be chosen to consist of curves of constant slope. Having such a property greatly simplifies the analysis that we performed in \Cref{claim:thinsurmapsudyn}.
        \item In the setting of \Cref{eg:closedorbitcurve}, up to shrinking the annulus towards the closed orbit, one can take the first return time of the annulus to be arbitrarily large. This allows one to get by with less precise control over the differential $d \sigma$. In some sense, the notion of $(\delta, R)$-thin surgery annuli compensates for the lack of control over this first return time.
    \end{enumerate}
    \item We remark that no assumption on the thinness of the surgery map is needed to construct the weak stable/unstable plane fields $\overline{E}^{cs/cu}$. Thinness is only used to establish the contraction/expansion dynamics and to find the line fields $\overline{E}^{s/u}$.
    
    Moreover, we note that the parameters of thinness, $\delta$ and $R$, depend continuously on the gluing map $\alpha$ in the following sense: 
    \begin{cor} \label{cor:thinnesscontinuous}
    For a value of $(\delta, R)$ suitable for a fixed parametrized surgery annulus $(A, \mathcal{H}, \mathcal{K}, \alpha)$ as in \Cref{thm:horsurpA}, the same value of $(\delta, R)$ is suitable for parametrized surgery annuli $(A, \overline{\mathcal{H}}, \overline{\mathcal{K}}, \overline{\alpha})$ where $\overline{\alpha}$ is sufficiently close to $\alpha$ in the $C^1$-topology.
    \end{cor}
    \begin{proof}
    The values of $(\delta, R)$ are chosen so that the conclusion of \Cref{claim:thinsurmapsudyn} holds for some fixed value $\epsilon$, hence we have to analyze the proof of that lemma.
    
    For parametrizations $\overline{\alpha}$ close to $\alpha$ in the $C^1$-topology, the matrix representative of $d\sigma$ in \Cref{claim:thinsurmapsudyn} varies only by a small amount. In particular we will still have 
    $$R \frac{a(h,k)b(h+\rho(k),k)}{a(h,k)d(h,k)+b(h,k)c(h,k)} > 1$$
    for the same value of $R$ and
    $$\frac{a(h,k)d(h+\rho(k),k)+b(h+\rho(k),k)c(h,k)}{a(h,k)d(h,k)+b(h,k)c(h,k)} > 1 - \epsilon$$
    for $|\rho(k)| \leq \delta$, for the same value of $\delta$.
    \end{proof}   
    \item The notion of slope is used both in the construction of the plane fields $\overline{E}^{cs/cu}$ and in the construction of the line fields $\overline{E}^{s/u}$, but there is a subtle difference in the strategy of the two usages.
    \begin{itemize}
        \item When constructing the plane fields $\overline{E}^{cs/cu}$, we first define the cone fields and show that they are nesting. Then we use a slope function to show that the intersection is a plane field. In particular here we can modify the slope function (by multiplying it by some large number $L$) until it behaves nicely on our cone fields.
        \item When constructing the line fields $\overline{E}^{s/u}$, we do not have natural candidates for the cone fields. Hence we first define a slope function and analyze how it behaves, before defining the cone field in terms of the slope. 
    \end{itemize}
    It is conceivable that by endowing further structures on the surgery annulus $A$, one can pursue the former strategy for constructing the line fields $\overline{E}^{s/u}$. However, this will add unnecessary rigidity to the category of surgery annuli, possibly affecting our arguments in the next subsection.
    \item We say that $\phi^t$ has \textbf{orientable stable/unstable foliation} if $E^{s/u}$ is orientable in $M \cut \sing(\phi^t)$, respectively. Note that since we assume $M$ is oriented, $\phi^t$ having orientable stable foliation implies that it has orientable unstable foliation, and vice versa. In this case, we can fix orientations on $E^s$ and $E^u$ and perform the cone field arguments on, say, the positive side of $E^{s/u}$. This would return line fields $\overline{E}^s$ and $\overline{E}^u$ that are orientable. That is, the surgered flow has orientable stable and unstable foliations as well.
\end{itemize}
\end{rmk}

From now on, in the setting of \Cref{thm:horsurpA}, we refer to the pseudo-Anosov flow $\phi^t_\sigma(A, \mathcal{H}, \mathcal{K}, \alpha)$ as a flow obtained by performing \textbf{horizontal Goodman surgery} on $\phi^t$.

\subsection{Invariance} \label{subsec:horsurinvar}

Our first goal in this subsection is to show that the orbit equivalence class of the flow $\phi^t_\sigma(A, \mathcal{H}, \mathcal{K}, \alpha)$ only depends on any leaf of $\mathcal{H}$ and on the coefficient of $\sigma$. Using this, we then show that the orbit equivalence class of the flow is invariant under isotopies of the surgery curve as well. The key tool for showing these results is \Cref{thm:structuralstability}, which we restate below.

\begin{thm:structuralstability}
Let $\phi^t$ be a pseudo-Anosov flow on a closed $3$-manifold $M$. Let $\nu$ be a fixed neighborhood of $\sing(\phi^t)$. There exists $\epsilon>0$ so that for every pseudo-Anosov flow $\overline{\phi^t}$ on $M$, if $\dot{\phi^t}$ and $\dot{\overline{\phi}^t}$ are $\epsilon$ close in the $C^1$-topology and if $\dot{\phi} = \dot{\overline{\phi}}$ in $\nu$, then $\phi^t$ and $\overline{\phi^t}$ are orbit equivalent through a homeomorphism that is isotopic to identity.
\end{thm:structuralstability}

We show \Cref{thm:structuralstability} in \Cref{sec:structuralstability}.

We will establish our invariance results via a sequence of lemmas.

\begin{lemma} \label{lemma:invsurmap}
Let $(A, \mathcal{H}, \mathcal{K}, \alpha)$ be a parametrized surgery annulus. For fixed $n$, $\delta$, and $R$, the space of $(\delta, R)$-thin surgery maps with coefficient $n$ is path-connected.

In particular, the orbit equivalence class of the flow $\phi^t_\sigma(A, \mathcal{H}, \mathcal{K}, \alpha)$ obtained by horizontal Goodman surgery only depends on $(A, \mathcal{H}, \mathcal{K}, \alpha)$ and the coefficient of $\sigma$.
\end{lemma}
\begin{proof}
Let $\sigma_0$ and $\sigma_1$ be two $(\delta, R)$-thin surgery maps with coefficient $n$. We have $\alpha \sigma_i \alpha^{-1}(h,k)=(h+\rho_i(k),k)$ for functions $\rho_i$ as in \Cref{defn:surmap}. Up to scaling the domain intervals on which $\rho_i$ are non-constant, we can assume that $R_0 = \max \rho'_0 = \max \rho'_1 = R_1$. Then up to translating these intervals, we can assume that $\rho_0=\rho_1$ on $J_0 \cap J_1$. We can now interpolate between $\rho_0$ and $\rho_1$ by $\rho_t(k) = t\rho_1(k)+(1-t)\rho_0(k)$.

For the second statement, we fix $\delta$ and $R$ so that $(\delta, R)$-thin surgery maps on $(A, \mathcal{H}, \mathcal{K}, \alpha)$ give pseudo-Anosov flows. A continuous family of such surgery maps gives a family of pseudo-Anosov flows with continuously varying generating vector field in a neighborhood of $A$. Thus we get the desired orbit equivalences from \Cref{thm:structuralstability}.
\end{proof}

In particular, \Cref{lemma:invsurmap} means that we can always shrink the support of the gluing map to a smaller surgery annulus, thus shrinking $A$ (while staying in the same orbit equivalence class).

\begin{lemma} \label{lemma:invpara}
Let $(A, \mathcal{H}, \mathcal{K})$ be a surgery annulus. The set of orientation preserving diffeomorphism $\alpha: A \times S^1 \times I$ that sends the leaves of $\mathcal{H}$ to $S^1 \times \{k\}$ and the leaves of $\mathcal{K}$ to $\{h\} \times I$ is connected is path-connected.

In particular, the orbit equivalence class of the flow $\phi^t_\sigma(A, \mathcal{H}, \mathcal{K}, \alpha)$ obtained by horizontal Goodman surgery only depends on $(A, \mathcal{H}, \mathcal{K})$ and the coefficient of $\sigma$.
\end{lemma}
\begin{proof}
For a fixed surgery annulus $(A, \mathcal{H}, \mathcal{K})$, the set of parametrizations can be identified with $\Diffeo^+(S^1) \times \Diffeo^+(I)$. It is an elementary fact that both $\Diffeo^+(S^1)$ and $\Diffeo^+(I)$ are connected.

For second statement, it suffices to show that if $\alpha_n$ is a sequence of parametrizations converging to $\alpha$, then there exists sufficiently thin surgery maps $\sigma_n$ and $\sigma$ so that $\phi^t_{\sigma_n}(A, \mathcal{H}, \mathcal{K}, \alpha_n)$ is orbit equivalent to $\phi^t_\sigma(A, \mathcal{H}, \mathcal{K}, \alpha)$ for large $n$. 

\Cref{cor:thinnesscontinuous} implies that there exists $(\delta, R)$ so that any $(\delta, R)$-thin surgery map yields a pseudo-Anosov flow for every $(A, \mathcal{H}, \mathcal{K}, \alpha_n)$ and $(A, \mathcal{H}, \mathcal{K}, \alpha)$. We fix such a surgery map $\sigma$ for $(A, \mathcal{H}, \mathcal{K}, \alpha)$. We can write $\alpha \sigma \alpha^{-1}(h,k) = (h+\rho(k),k)$. We define surgery maps $\sigma_n$ by $\alpha_n \sigma_n \alpha_n^{-1}(h,k) = (h+\rho(k),k)$. Then $\sigma_n$ are $(\delta, R)$-thin and the generating vector field for $\phi^t_{\sigma_n}(A, \mathcal{H}, \mathcal{K}, \alpha_n)$ converges to that of $\phi^t_\sigma(A, \mathcal{H}, \mathcal{K}, \alpha)$ in a neighborhood of $A$. Thus we get the desired orbit equivalences from \Cref{thm:structuralstability}.
\end{proof}

\begin{lemma} \label{lemma:invkfol}
Let $A \subset M \backslash \sing(\phi^t)$ be an embedded oriented annulus with a foliation by positive/negative curves $\mathcal{H}$ where $T\mathcal{H}$ is steady. The space of foliations $\mathcal{K}$ such that $(A, \mathcal{H}, \mathcal{K})$ is a surgery annulus is path-connected.

In particular, the orbit equivalence class of the flow $\phi^t_\sigma(A, \mathcal{H}, \mathcal{K}, \alpha)$ obtained by horizontal Goodman surgery only depends on $(A, \mathcal{H})$ and the coefficient of $\sigma$.
\end{lemma}
\begin{proof}
Let $\mathcal{K}_0$ and $\mathcal{K}_1$ be two such foliations. We define a continuous family of foliations $\mathcal{K}_t$ by specifying the slope of $T\mathcal{K}_t$ at $x \in A$ to be the linear interpolation between that of $T\mathcal{K}_0$ and $T\mathcal{K}_1$.

It is clear that $T\mathcal{K}_t$ is negative. Hence the leaves of $\mathcal{K}_t$ are transverse to those of $\mathcal{H}$, in particular implying that they are non-separating arcs on $A$. We further claim that $T\mathcal{K}_t$ is steady. For every $x, y \in A$, $t > 0$ such that $y=\phi^t(x)$, the slope of $T\mathcal{K}_i|_y$ is smaller than than of $d\phi^t(T\mathcal{K}_i|_x)$, for $i=0,1$. Since $d\phi^t|_x$ acts on the set of slopes by multiplication by some factor (depending on $x$), the same inequality holds under linear interpolation to show that $T\mathcal{K}_t$ is steady.

To prove the second statement, observe that we can pick a continuous family of parametrizations $\alpha_t$ for $(A, \mathcal{H}, \mathcal{K}_t)$, by, for example, fixing the parametrization of a single leaf of $\mathcal{H}$ and fixing a parametrization of the space of leaves of $\mathcal{H}$. Thus we can pick appropriate gluing maps to apply \Cref{thm:structuralstability} as in the proof of \Cref{lemma:invpara}.
\end{proof}

\begin{thm} \label{thm:invhleaf}
Let $(A_i, \mathcal{H}_i, \mathcal{K}_i, \alpha_i)$ be a parametrized surgery annulus for $i=1,2$. Suppose there is a curve $c$ that is a leaf of both $\mathcal{H}_1$ and $\mathcal{H}_2$. Then for fixed $n$, the flows $\phi^t_{\sigma_i}(A_i, \mathcal{H}_i, \mathcal{K}_i, \alpha_i)$ obtained by horizontal Goodman surgery are orbit equivalent whenever $\sigma_1$ and $\sigma_2$ are sufficiently thin surgery maps with coefficient $n$. 

That is, the orbit equivalence class of the flow $\phi^t_\sigma(A, \mathcal{H}, \mathcal{K}, \alpha)$ obtained by horizontal Goodman surgery only depends on any leaf of $\mathcal{H}$ and the coefficient of $\sigma$.
\end{thm}
\begin{proof}
Up to shrinking the support of $\sigma_1$ and $\sigma_2$ using \Cref{lemma:invsurmap}, we can assume that $c$ lies on the boundary of both $A_1$ and $A_2$. There are two cases here, depending on whether $A_1$ and $A_2$ lie on the same side of $c$. See \Cref{fig:invhleafarg}. We first tackle the case when $A_1$ and $A_2$ lie on different sides of $c$.

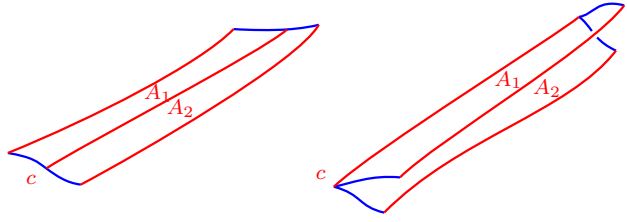
\begin{figure}
    \centering
    \fontsize{8pt}{8pt}\selectfont
\begingroup%
  \makeatletter%
  \providecommand\color[2][]{%
    \errmessage{(Inkscape) Color is used for the text in Inkscape, but the package 'color.sty' is not loaded}%
    \renewcommand\color[2][]{}%
  }%
  \providecommand\transparent[1]{%
    \errmessage{(Inkscape) Transparency is used (non-zero) for the text in Inkscape, but the package 'transparent.sty' is not loaded}%
    \renewcommand\transparent[1]{}%
  }%
  \providecommand\rotatebox[2]{#2}%
  \newcommand*\fsize{\dimexpr\f@size pt\relax}%
  \newcommand*\lineheight[1]{\fontsize{\fsize}{#1\fsize}\selectfont}%
  \ifx\svgwidth\undefined%
    \setlength{\unitlength}{232.77860146bp}%
    \ifx\svgscale\undefined%
      \relax%
    \else%
      \setlength{\unitlength}{\unitlength * \real{\svgscale}}%
    \fi%
  \else%
    \setlength{\unitlength}{\svgwidth}%
  \fi%
  \global\let\svgwidth\undefined%
  \global\let\svgscale\undefined%
  \makeatother%
  \begin{picture}(1,0.34200415)%
    \lineheight{1}%
    \setlength\tabcolsep{0pt}%
    \put(0,0){\includegraphics[width=\unitlength,page=1]{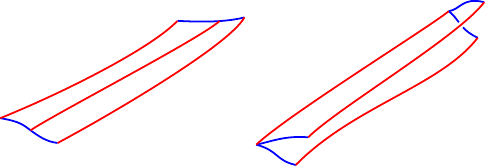}}%
    \put(0.22134868,0.18622239){\color[rgb]{1,0,0}\makebox(0,0)[lt]{\lineheight{1.25}\smash{\begin{tabular}[t]{l}$A_1$\end{tabular}}}}%
    \put(0.25801746,0.16162003){\color[rgb]{1,0,0}\makebox(0,0)[lt]{\lineheight{1.25}\smash{\begin{tabular}[t]{l}$A_2$\end{tabular}}}}%
    \put(0.85129603,0.19435073){\color[rgb]{1,0,0}\makebox(0,0)[lt]{\lineheight{1.25}\smash{\begin{tabular}[t]{l}$A_2$\end{tabular}}}}%
    \put(0.03049508,0.04757231){\color[rgb]{1,0,0}\makebox(0,0)[lt]{\lineheight{1.25}\smash{\begin{tabular}[t]{l}$c$\end{tabular}}}}%
    \put(0.49901861,0.05689231){\color[rgb]{1,0,0}\makebox(0,0)[lt]{\lineheight{1.25}\smash{\begin{tabular}[t]{l}$c$\end{tabular}}}}%
    \put(0.78852816,0.20762952){\color[rgb]{1,0,0}\makebox(0,0)[lt]{\lineheight{1.25}\smash{\begin{tabular}[t]{l}$A_1$\end{tabular}}}}%
  \end{picture}%
\endgroup%

    \caption{Left: $A_1$ and $A_2$ lying on different sides of $c$. Right: $A_1$ and $A_2$ lying on the same side of $c$.}
    \label{fig:invhleafarg}
\end{figure}

In this case we claim that up to shrinking $A_i$ and modifying $\mathcal{K}_i$, we can find a surgery annulus $(A, \mathcal{H}, \mathcal{K})$ such that $A_i \subset A$ with $\mathcal{H}|_{A_i} = \mathcal{H}_i$ and $\mathcal{K}|_{A_i} = \mathcal{K}_i$. This can be done by shrinking each $A_i$ so that it is a thin strip lying close to but not containing $c$, then finding an embedded annulus $A$ containing both $A_i$ with a foliation $\mathcal{H}$ by curves that extends both $\mathcal{H}_i$. If each $A_i$ is sufficiently shrunk, $T\mathcal{H}$ will be close to $Tc$ in the $C^1$-topology, hence is steady. Now we define a foliation $\mathcal{K}$ on $A$ be specifying $T\mathcal{K}$ to be the lines of slope, say, $-1$. As in \Cref{prop:horsurcurvetosurann}, $\mathcal{K}$ will be a foliation by non-separating arcs on $A$ and $T\mathcal{K}$ will be negative and steady. Restricting $\mathcal{K}$ to each $A_i$ gives a foliation $\mathcal{K}_i$ that makes $(A_i, \mathcal{H}_i, \mathcal{K}_i)$ into a surgery annulus.

We then apply \Cref{lemma:invsurmap} to see that the flows obtained by horizontal Goodman surgery on $(A_i, \mathcal{H}_i, \mathcal{K}_i)$ are each orbit equivalent to that obtained by horizontal Goodman surgery on $(A, \mathcal{H}, \mathcal{K})$, hence are orbit equivalent to each other.

In the case when $A_1$ and $A_2$ lie on the same side of $c$, we use \Cref{prop:horsurcurvetosurann} to find a surgery annulus $A_3$ containing $c$ that lies on the other side of $c$ and use the above argument on $(A_1, A_3)$ then $(A_2, A_3)$.
\end{proof}

In view of \Cref{thm:invhleaf}, we will write $\phi^t_\sigma(A, \mathcal{H}, \mathcal{K}, \alpha)$ as $\phi^t_{\frac{1}{n}}(c)$, where $c$ is a leaf of $\mathcal{H}$ and $n$ is the coefficient of $\sigma$, from now on. This is notation is inspired from the fact that $\phi^t_{\frac{1}{n}}(c)$ is defined on the $3$-manifold $M_{\frac{1}{n}}(c)$ obtained by performing Dehn surgery on $M$ along $c$ with coefficient $\frac{1}{n}$. Here this coefficient is taken with respect to the choice of longitude determined by the flow $\phi^t$.

\begin{thm} \label{thm:invcurveperturb}
Let $c$ be a positive/negative horizontal surgery curve. Let $n$ be a positive/negative integer, respectively. Then for $\overline{c}$ close to $c$ in the $C^1$-topology, the flows $\phi^t_{\frac{1}{n}}(c)$ and $\phi^t_{\frac{1}{n}}(\overline{c})$ are orbit equivalent.
\end{thm}
\begin{proof}
Let $(A, \mathcal{H}, \mathcal{K})$ be a surgery annulus such that $c$ is a leaf of $\mathcal{H}$. Let $d$ be another leaf of $\mathcal{H}$. For $\overline{c}$ close to $c$ in the $C^1$-topology, one can find an embedded annulus $\overline{A}$ with a foliation $\overline{\mathcal{H}}$ by closed curves containing $\overline{c}$ and $d$ as leaves. Furthermore, by taking $d$ and $\overline{c}$ close enough to $c$, $T\overline{\mathcal{H}}$ can be taken to be close to $Tc$, hence be steady. The theorem then follows since $\phi^t_\frac{1}{n}(c) \cong \phi^t_\frac{1}{n}(d) \cong \phi^t_\frac{1}{n}(\overline{c})$.
\end{proof}

\Cref{thm:invcurveperturb} implies the following corollary.

\begin{cor} \label{cor:invcurveisotopy}
Let $c_s$, $s \in [0,1]$ be a family of positive/negative surgery curves. Then for every positive/negative integer $n$, respectively, $\phi^t_{\frac{1}{n}}(c_0)$ and $\phi^t_{\frac{1}{n}}(c_1)$ are orbit equivalent.
\end{cor}

\subsection{Examples} \label{subsec:horsureg}

In this subsection, we discuss the effect of performing horizontal Goodman surgery along some of the horizontal surgery curves given in \Cref{subsec:curveeg}.

\begin{eg} \label{eg:pAmapstraightsur}
This example continues from \Cref{eg:pAmapstraightcurve}. Let $S$ be a closed oriented surface and let $f:S \to S$ be an orientation preserving pseudo-Anosov map. Let $c \subset S$ be a straight positive/negative curve. \Cref{prop:pAmapstraightcurve} shows that $c$ is a positive/negative horizontal surgery curve for the suspension flow $\phi^t_f$ on the mapping torus $T_f$, respectively.

We claim that the flow $(\phi^t_f)_{\frac{1}{n}}(c)$ is the suspension flow of the map $f \tau^n_c: S \to S$, where $\tau_c:S \to S$ denotes the positive Dehn twist along $c$. Indeed, we can choose a section $S_0$ for $\phi^t_f$ lying slightly below $c$. We can then choose the surgery annulus to be disjoint from $S_0$. After surgery, $S_0$ is still a section, and one sees that the first return map is $f \tau^n_c$.

For positive/negative integers $n$, the fact that $(\phi^t_f)_{\frac{1}{n}}(c)$ is a pseudo-Anosov flow then implies the following proposition.

\begin{prop} \label{prop:pAmaptwist}
Let $S$ be a closed oriented surface and let $f:S \to S$ be an orientation preserving pseudo-Anosov map. Let $c \subset S$ be a straight positive/negative curve. Then $f \tau^n_c$ is pseudo-Anosov for every positive/negative integer $n$, respectively.
\end{prop}

\end{eg}

\begin{eg} \label{eg:closedorbitsur}
This example continues from \Cref{eg:closedorbitcurve}. Let $\phi^t$ be a pseudo-Anosov flow on closed oriented $3$-manifold $M$ and let $\gamma$ be a closed orbit of $\phi^t$. Let $N$ be a tubular neighborhood of $\gamma$ and let $A$ be an annulus in a positive/negative quadrant of $N$ that has one boundary component $\overline{\gamma}$ lying along $\gamma$ and which is transverse to $\phi^t$ in its interior. \Cref{prop:closedorbitcurve} shows that there is a negative/positive horizontal surgery curve $c$ lying on $A$ that is isotopic to $\overline{\gamma}$, respectively.

Performing horizontal Goodman surgery along such a $c$ is the classical operation of Goodman surgery. In his thesis, Shannon showed the following theorem.

\begin{thm}[Shannon {\cite{Sha20}}] \label{thm:Goodman=Fried}
Let $\phi^t$ be a transitive pseudo-Anosov flow on a closed orientable $3$-manifold $M$. Let $\gamma$ be a closed orbit of $\phi^t$. Fix a positive/negative integer $n$. There exists a neighborhood $N$ of $\gamma$ such that for every positive/negative horizontal surgery curve $c$ lying on a transverse annulus in $N$, the flows $\phi^t_\frac{1}{n}(c)$ obtained by (horizontal) Goodman surgery on $c$ are orbit equivalent to one another.

Moreover, the restriction of $\phi^t$ to $M \backslash \gamma$ is orbit equivalent to the restriction of $\phi^t_\frac{1}{n}(c)$ to $M_\frac{1}{n}(c) \backslash \gamma$, where the image of every local section $D$ in $M$ near $\gamma$ under the orbit equivalence can be extended into an immersed surface with boundary $\overline{D}$ in $M_\frac{1}{n}(c)$. 
\end{thm}
\begin{proof}
The first statement in the case when $\phi^t$ is Anosov is stated as \cite[Corollary 3.2(a)]{Sha20}. We recount the proof here to illustrate how Shannon's proof carries through to pseudo-Anosov flows, as well as how the second statement follows.

By \cite{Bru95} (which follows ideas from \cite{Fri83}), $\phi^t$ admits a \textbf{Birkhoff section} $S$ whose boundary is disjoint from $\gamma$. Recall that this means that $S$ is a embedded surface with boundary where 
\begin{itemize}
    \item the interior of $S$ is transverse to the flow,
    \item the boundary of $S$ lies along closed orbits of $\phi^t$,
    \item and every orbit of $\phi^t$ meets $S$ in finite forward and backward time.
\end{itemize}
Let $f:S \to S$ be the first return map of $S$.

In \cite[Proposition 3.14]{Sha20}, Shannon shows that $\phi^t_\frac{1}{n}(c)$ admits a Birkhoff section $\overline{S}$ so that the first return map on $\overline{S} \backslash \gamma$ is conjugate to the restriction of $f$ to $S \backslash (S \cap \gamma)$. The way Shannon does this is to first isotope the surgery annulus $A$ to lie on the boundary of a normal neighborhood $N$, then think of $\phi^t_\frac{1}{n}(c)$ as $\phi^t$ cut along $\partial N$ and reglued with a Dehn twist. He then shows how to extend the image of $S \cap \partial N$ under this regluing map into a local Birkhoff section inside $N$. The crucial fact here is that the generating vector field $\dot{\phi}$ can be explicitly written down in $N$.

When $\phi^t$ is pseudo-Anosov but $\gamma$ is non-singular, this argument holds word-by-word. If $\gamma$ is singular, one can repeat the argument by using a neighborhood furnished by the last item in \Cref{defn:pAflow}.

From the observation that the twist parameters of this Birkhoff section are determined by the surgery coefficient $n$, one can then apply \cite[Theorem A]{Sha20} (or \cite[Theorem 3.9]{Sha21}) to see that the flows $\phi^t_\frac{1}{n}(c)$ are orbit equivalent to one another. Strictly speaking \cite[Theorem A]{Sha20} is only stated in the case when $\gamma$ is non-singular and orientation-preserving, but the result holds without these additional hypothesis, as remarked below \cite[Theorem 3.9]{Sha21}.

The second statement follows from similar ideas. One can define an orbit equivalence between the restriction of $\phi^t$ to $M \backslash (\gamma \cup \partial S)$ and the restriction of $\phi^t_\frac{1}{n}(c)$ to $\overline{M} \backslash \partial \overline{S}$ by sending the interior of $S \backslash (S \cap \gamma)$ to the interior of $\overline{S}$ via the conjugation above. One then applies \cite[Theorem 2.18]{Sha20} (or \cite[Theorem 3.20]{Sha21}) to each component of $\partial S$ to obtain an orbit equivalence between the restriction of $\phi^t$ to $M \backslash \gamma$ and the restriction of $\phi^t_\frac{1}{n}(c)$ to $M_\frac{1}{n}(c) \backslash \gamma$ agreeing with the previous orbit equivalence near $\gamma$. By definition, the image of neighborhoods of $S$ at $S \cap \gamma$ under the orbit equivalence can be extend to neighborhoods of $\overline{S}$ near $\gamma$.
\end{proof}

When the half-leaves of $\Lambda^{s/u}$ do not rotate along $\gamma$, each surgery curve $c$ as above is isotopic to $\gamma$ and $E^{s/u}|_\gamma$ determines a curve $\lambda$ on $\partial N$. Here $\lambda$ is a longitude, meaning that it intersects the meridian $\mu$ in one point. In particular the $3$-manifold $M_\frac{1}{n}(c)$ is naturally diffeomorphic to the $3$-manifold $M_\frac{1}{n}(\gamma)$ obtained by performing Dehn surgery on $M$ along $\gamma$ with coefficient $\frac{1}{n}$. As such, we will write $\phi^t_\frac{1}{n}(c)$ as $\phi^t_\frac{1}{n}(\gamma)$ and refer to it as a flow obtained by \textbf{Goodman-Fried surgery} on $\gamma$.
\end{eg}

\begin{eg} \label{eg:scalloptorussur}
This example continues from \Cref{eg:scalloptoruscurve}. Let $\phi^t$ be a pseudo-Anosov flow on a closed oriented $3$-manifold $M$. Let $T$ be a scalloped transverse torus. We let $\ell^{s/u}$ be the homology class of one of the closed $T^{s/u}$-leaf, respectively. Note that these are only defined up to a sign. Let us fix some choice so that the intersection number $\langle \ell^u, \ell^s \rangle > 0$. We let $m^{s/u}$ be homology classes so that $(l^s, m^s)$ and $(l^u, m^u)$ each form a positive basis for $H_1(T^2; \mathbb{Z})$.

Now consider the $3$-manifold $M \cut T$ obtained by cutting $M$ along $T$. There are two copies of $T$ in $M \cut T$ --- one on the positive side of $T$ and one on the negative side of $T$. We denote these by $T_+$ and $T_-$ respectively. If we endow $T_+$ with the basis $(\ell^s, m^s)$ and $T_-$ with the basis $(\ell^u, m^u)$, then the gluing map $T_- \to T_+$ that recovers $M$ has matrix representation 
$\begin{bmatrix}
a & b \\
c & d
\end{bmatrix}$
where $c = - \langle \ell^u, \ell^s \rangle < 0$.

\Cref{prop:scalloptoruscurve} shows that for every positive/negative primitive integer homology class $\alpha$, there is a positive/negative horizontal surgery curve $c_\alpha$ on $T$ with homology class $\pm \alpha$. Observe that the $3$-manifold $M_\frac{1}{n}(c_\alpha)$ for which the flow $\phi^t_\frac{1}{n}(c_\alpha)$ is defined on can be described topologically as gluing up $M \cut T$ by some map $T_- \to T_+$. We compute the matrix representative of this map:

Let the coordinates of $\alpha$ in the basis $(\ell^s, m^s)$ be $(x_0, y_0)$.
Suppose that $\alpha$ is positive. Then up to reversing the sign of $\alpha$, we have
$$\begin{cases}
\langle \ell^u, \alpha \rangle = ay_0 - cx_0 &\geq 0 \\
\langle \ell^s, \alpha \rangle = y_0 &\leq 0.
\end{cases}$$
In the basis $(\ell^s, m^s)$, a single positive Dehn twist along $c_\alpha$ has matrix representation
$\begin{bmatrix}
1+x_0y_0 & -x^2_0 \\
y^2_0 & 1-x_0y_0
\end{bmatrix}$.
Hence the matrix representation of the gluing map for $M_1(c_\alpha)$ is
$$\begin{bmatrix}
1+x_0y_0 & -x^2_0 \\
y^2_0 & 1-x_0y_0
\end{bmatrix}
\begin{bmatrix}
a & b \\
c & d
\end{bmatrix}
=
\begin{bmatrix}
a+(ay_0-cx_0)x_0 & b+(by_0-dx_0)x_0 \\
c+(ay_0-cx_0)y_0 & d+(by_0-dx_0)y_0
\end{bmatrix}$$
In particular, notice that $c+(ay_0-cx_0)y_0 \leq c$.

Symmetrically, if $\alpha$ is negative, then up to reversing the sign of $\alpha$, we have
$$\begin{cases}
\langle \ell^u, \alpha \rangle = ay_0 - cx_0 &\leq 0 \\
\langle \ell^s, \alpha \rangle = y_0 &\leq 0.
\end{cases}$$
In the basis $(\ell^s, m^s)$, a single negative Dehn twist along $c_\alpha$ has matrix representation
$\begin{bmatrix}
1-x_0y_0 & x^2_0 \\
-y^2_0 & 1+x_0y_0
\end{bmatrix}$.
Hence the matrix representation of the gluing map for $M_{-1}(c_\alpha)$ is
$$\begin{bmatrix}
1-x_0y_0 & x^2_0 \\
-y^2_0 & 1+x_0y_0
\end{bmatrix}
\begin{bmatrix}
a & b \\
c & d
\end{bmatrix}
=
\begin{bmatrix}
a-(ay_0-cx_0)x_0 & b-(by_0-dx_0)x_0 \\
c-(ay_0-cx_0)y_0 & d-(by_0-dx_0)y_0
\end{bmatrix}$$
In particular, notice that once again $c-(ay_0-cx_0)y_0 \leq c$.

Notice that the torus $T$ is transverse to the surgered flow $\phi_{\frac{1}{n}}(c_\alpha)$ after regluing. We further claim that it is a scalloped transverse torus. 

The key is to observe that there is a closed orbit of $\phi^t$ some multiple of which is homotopic to a closed $T^u$-leaf on $T_-$ in $M \cut T$. Indeed, consider the half-leaves of the unstable foliation $\Lambda^u$ that meet $T$ in closed $T^u$-leaves. Such half-leaves are annuli with one boundary component lying along a multiple of a closed orbit in $M \cut T$. These half-leaves can intersect $T$ multiple times but each of these intersections is a closed $T^u$-leaf, hence we can extract a sub-annulus giving a homotopy between a multiple of a closed orbit to a closed $T^u$-leaf on $T_-$ in $M \cut T$. Horizontal Goodman surgery does not modify the flow on this annulus, hence we still have a leaf of the unstable foliation of the surgered flow intersecting $T_-$ in a closed $T^u$-leaf.

Symmetrically, we have a leaf of the stable foliation of the surgered flow intersecting $T_+$ in a closed $T^s$-leaf. The homology class of the closed leaves of $T^u$ on $T_-$ differs from that of $T^s$ on $T_+$ under the new gluing map since $c \pm (ay_0-cx_0)y_0 \leq c < 0$. Using \cite[Definition 2.12]{BFM23}, this shows that $T$ remains a scalloped torus after regluing. 

This in particular implies that we can repeat the above procedure and perform another horizontal Goodman surgery along a suitable curve on $T$. We now wish to investigate which $3$-manifolds we can obtain using such repeated operations.

To that end, we make the following algebraic definition. Define a directed graph $\mathfrak{G}$ whose vertices are matrices $\begin{bmatrix}
a & b \\
c & d
\end{bmatrix} \in \mathrm{SL}_2 \mathbb{Z}$ where $c \leq -1$.
The directed edges of $\mathfrak{G}$ are defined to be:
\begin{itemize}
    \item $\begin{bmatrix}
    a & b \\
    c & d
    \end{bmatrix} \to 
    \begin{bmatrix}
    a+(ay_0-cx_0)x_0 & b+(by_0-dx_0)x_0 \\
    c+(ay_0-cx_0)y_0 & d+(by_0-dx_0)y_0
    \end{bmatrix}$ for 
    $\begin{cases}
    ay_0 - cx_0 &\geq 0 \\
    y_0 &\leq 0.
    \end{cases}$, and
    \item $\begin{bmatrix}
    a & b \\
    c & d
    \end{bmatrix} \to 
    \begin{bmatrix}
    a-(ay_0-cx_0)x_0 & b-(by_0-dx_0)x_0 \\
    c-(ay_0-cx_0)y_0 & d-(by_0-dx_0)y_0
    \end{bmatrix}$ for 
    $\begin{cases}
    ay_0 - cx_0 &\leq 0 \\
    y_0 &\leq 0.
    \end{cases}$.
\end{itemize}
Then the above computation implies the following proposition
\begin{prop} \label{prop:scalloptorusimplications}
Let $\phi^t$, $T$, $(l^s, m^s)$, $(l^u, m^u)$, and 
$\begin{bmatrix}
a & b \\
c & d
\end{bmatrix}$
be as above. Then for every $\begin{bmatrix}
m & n \\
p & q
\end{bmatrix} \in \mathfrak{G}$ for which there exists a directed path from $\begin{bmatrix}
a & b \\
c & d
\end{bmatrix}$
to $\begin{bmatrix}
m & n \\
p & q
\end{bmatrix}$, the $3$-manifold obtained by gluing up $M \cut T$ by $\begin{bmatrix}
m & n \\
p & q
\end{bmatrix}: T_- \to T_+$ admits a pseudo-Anosov flow.
\end{prop}

We then prove the following lemma that concerns the structure of $\mathfrak{G}$.

\begin{lemma} \label{lemma:scalloptorusgraphstructure}
Let $\mathfrak{G}_n = \left\{ \begin{bmatrix}
a & b \\
c & d
\end{bmatrix} \in \mathrm{SL}_2 \mathbb{Z} \mid c = -n \right\}$. Notice that $\mathfrak{G}_n$ form a partition of the vertices of $\mathfrak{G}$. Then:
\begin{enumerate}
    \item Every edge of $\mathfrak{G}$ goes from a vertex in $\mathfrak{G}_m$ to a vertex in $\mathfrak{G}_n$, where $m \leq n$.
    \item The full subgraph of $\mathfrak{G}$ spanned by $\mathfrak{G}_1$ is strongly connected.
    \item For every vertex $v$ in $\mathfrak{G}_n$ where $n \geq 2$, there exists a directed edge path from a vertex $w$ in $\mathfrak{G}_m$ to $v$, where $m < n$.
\end{enumerate}
\end{lemma}
\begin{proof}
(1) was already shown  in our computation above. To prove (2), notice that 
$$\mathfrak{G}_1 = \left\{ \begin{bmatrix}
a & -ad+1 \\
-1 & d
\end{bmatrix} \mid a,d \in \mathbb{Z} \right\}.$$
By taking $(x_0, y_0)=\pm (1,0)$, we get edges 
$\begin{bmatrix}
a & -ad+1 \\
-1 & d
\end{bmatrix} \to 
\begin{bmatrix}
a \pm 1 & -(a \pm 1)d+1 \\
-1 & d
\end{bmatrix}$, and by taking $(x_0, y_0)=(a,-1)$, we get edges
$\begin{bmatrix}
a & -ad+1 \\
-1 & d
\end{bmatrix} \to 
\begin{bmatrix}
a & -a(d \pm 1)+1 \\
-1 & d \pm 1
\end{bmatrix}$.
These edges allow one to go between any two vertices in $\mathfrak{G}_1$.

For (3), notice that by taking $(x_0, y_0)=(a,c)$, we get edges 
$\begin{bmatrix}
a & b \\
c & d
\end{bmatrix} \to 
\begin{bmatrix}
a & b \pm a \\
c & d \pm c
\end{bmatrix}$. Using these edges, we can assume that $v = 
\begin{bmatrix}
a & b \\
c & d
\end{bmatrix} \in \mathfrak{G}_n$, where $|d|<|c|$.
We then observe that for $d \geq 0$, we can take $(x_0, y_0)=(-b,-d)$ to get the edge
$\begin{bmatrix}
a+b & b \\
c+d & d
\end{bmatrix} \to 
\begin{bmatrix}
a & b \\
c & d
\end{bmatrix}$, and for $d \leq 0$, we can take $(x_0, y_0)=(b,d)$ to get the edge
$\begin{bmatrix}
a-b & b \\
c-d & d
\end{bmatrix} \to 
\begin{bmatrix}
a & b \\
c & d
\end{bmatrix}$.
\end{proof}

In particular, \Cref{prop:scalloptorusimplications} and \Cref{lemma:scalloptorusgraphstructure}(2) and (3) implies the following corollary.

\begin{cor} \label{cor:scalloptorusintersection1}
Let $\phi^t$, $T$, $(l^s, m^s)$ and $(l^u, m^u)$ be as above. If $\langle l^s, l^u \rangle = -1$, then for every 
$\begin{bmatrix}
m & n \\
p & q
\end{bmatrix} \in \mathrm{SL}_2 \mathbb{Z}$ with $p<0$, the $3$-manifold obtained by gluing up $M \cut T$ by $\begin{bmatrix}
m & n \\
p & q
\end{bmatrix}: T_- \to T_+$ admits a pseudo-Anosov flow.
\end{cor}
\end{eg}

\begin{eg} \label{eg:interiortorussur}
This example continues from \Cref{eg:interiortoruscurve}. Let $\phi^t$ be a pseudo-Anosov flow on closed oriented $3$-manifold $M$. Let $T$ be a positive Reebless non-scalloped transverse torus. We let $\ell$ be the homology class of one of the closed leaves of $T^{s/u}$ and let $m$ be a homology class so that $(l,m)$ form a positive basis for $H_1(T; \mathbb{Z})$.

Under this basis, $M$ can be recovered from $M \cut T$ using the gluing map $T_- \to T_+$ with matrix representative
$\begin{bmatrix}
1 & 0 \\
0 & 1
\end{bmatrix}$.
By \Cref{prop:interiortoruscurve}, there is a positive horizontal surgery curve $c$ with homology class $m$ on $T$. The $3$-manifold $M_\frac{1}{n}(c)$ for which the flow $\phi^t_\frac{1}{n}(c)$ is defined on can be described topologically as gluing up $M \cut T$ by some map $T_- \to T_+$ with matrix representative
$\begin{bmatrix}
1 & 0 \\
n & 1
\end{bmatrix}$.

As in \Cref{eg:scalloptorussur}, $T$ is transverse to the surgered flow after regluing. We further claim that it is a scalloped transverse torus. 

The proof is as in \Cref{eg:scalloptorussur}: We have a leaf of the unstable foliation of the surgered flow intersecting $T_-$ in a closed $T^u$-leaf, and a leaf of the stable foliation intersecting $T_+$ in a closed $T^s$-leaf. Since $n > 0$, these have different homology classes hence using \cite[Definition 2.12]{BFM23}, this shows that $T$ becomes a scalloped torus after regluing.

We record this as the following proposition.

\begin{prop} \label{prop:interiortorussur}
Let $\phi^t$ be a pseudo-Anosov flow on a closed oriented $3$-manifold $M$. Let $T$ be a positive/negative Reebless non-scalloped transverse torus. Let $\ell$ be the homology class of one of the closed leaves of $T^{s/u}$ and let $m$ be a homology class so that $(l,m)$ form a positive basis for $H_1(T; \mathbb{Z})$. 

Then for every negative/positive integer $n$, respectively, the $3$-manifold obtained by gluing up $M \cut T$ by $\begin{bmatrix}
1 & 0 \\
n & 1
\end{bmatrix}: T_- \to T_+$ admits a pseudo-Anosov flow, for which $T$ is a scalloped transverse torus.
\end{prop}

\end{eg}

\begin{eg} \label{eg:bicontactsur}
This example continues from \Cref{eg:bicontactcurve}. Let $\phi^t$ be an Anosov flow on a closed oriented $3$-manifold $M$. Recall that there is a positive contact structure $\xi_+$ and a negative contact structure $\xi_-$ such that $T\phi = \xi_+ \cap \xi_-$. 

In \cite{Sal22}, Salmoiraghi introduces a surgery operation performed along \textbf{Legendrian-transverse knots}. These are curves $c$ whose tangent field $Tc$ is tangent to $\xi_-$ and transverse to $\xi_+$. As reasoned in \Cref{eg:bicontactcurve}, these curves are positive horizontal surgery curves. We briefly review the operation in the language of this paper, with the goal of showing \Cref{prop:Sal22} below.

Given a Legendrian-transverse knot $c$ and a positive integer $n$, Salmoiraghi's operation goes as follows:
\begin{itemize}
    \item Find an annulus $A$ containing $c$ and transverse to the flow, such that $A$ admits two foliations:
    \begin{itemize}
        \item a foliation $\mathcal{H}$ by closed curves that are tangent to $\xi_-$, and
        \item a foliation $\mathcal{K}$ by non-separating arcs that are tangent to $\xi_+$.
    \end{itemize}
    Using the same reasoning as in \Cref{eg:bicontactcurve}, we see that $T\mathcal{H}$ is positive and steady, while $T\mathcal{K}$ is negative and steady. That is, $(A, \mathcal{H}, \mathcal{K})$ is a positive surgery annulus.
    \item Choose some parametrization $\alpha$ of $(A, \mathcal{H}, \mathcal{K})$.
    \item Cut the $3$-manifold $M$ along $A$ and reglue it using a map $\sigma:A \to A$ such that $\sigma = \mathrm{id}$ near $\partial A$, $\alpha \sigma \alpha^{-1}(h,k) = (h+\rho(k),k)$, and $\rho(0)-\rho(1)=n$.
\end{itemize}
We denote the resulting flow by $\phi^t_{\sigma}(A, \mathcal{H}, \mathcal{K}, \alpha)$.

A difference between Salmoiraghi's operation and \Cref{constr:sur} is that Salmoiraghi is less restrictive with the function $\rho$. We also point out that Salmoiraghi does not address whether the flows he constructs using different choices of $A$, $\mathcal{H}$, $\mathcal{K}$, $\alpha$, $\sigma$ are orbit equivalent. Hence the best statement we can make from this discussion is the following.

\begin{prop} \label{prop:Sal22}
For some choice of construction data, Salmoiraghi's flow $\phi^t_{\sigma}(A, \mathcal{H}, \mathcal{K}, \alpha)$ is orbit equivalent to the flow $\phi^t_{\frac{1}{n}}(c)$ obtained by horizontal Goodman surgery.

In particular, for such choice of construction data, Salmoiraghi's flow $\phi^t_{\sigma}(A, \mathcal{H}, \mathcal{K}, \alpha)$ is Anosov. 
\end{prop}

We speculate that Salmoiraghi's flow is orbit equivalent to $\phi^t_{\frac{1}{n}}(c)$ whenever the former is Anosov. See \Cref{subsec:Sal} for more discussion on the possible relations between Salmoiraghi's work and our paper.
\end{eg}

\begin{noneg} \label{noneg:braid}
We end this section with a non-example showing that steadiness is necessary in the definition of a horizontal surgery curve. Let $S$ be a closed oriented surface and let $f:S \to S$ be an orientation preserving pseudo-Anosov map. Let $\phi^t$ be the suspension pseudo-Anosov flow on the mapping torus $M=T_f$. Let $c \subset S$ be a straight positive curve. By \Cref{prop:pAmapstraightcurve}, $c$ can be considered as a positive horizontal surgery curve in $M$.

Now fix a neighborhood $N \cong c \times I \times I$ of $c$. Let $b$ be the curve obtained by inserting the closed braid $\beta$ in \Cref{fig:braids} left into $N$, as described in \Cref{eg:braidcurve}. By placing $b$ close to $c$, we can assume that $b$ is also a positive curve. Note, however, that $Tb$ will not be steady. This is because the crossing of $\beta$ is positive, so at the corresponding crossing $(x,y,t)$ of $b$ inside $N$, the slope of $Tb|_y$ is less than $d\phi^t(Tb|_x)$, instead of the other way around.

Consider the $3$-manifold $M_1(b)$ obtained by performing Dehn surgery along $b$ with coefficient $1$. We claim that this $3$-manifold is reducible, i.e. it contains an embedded $2$-sphere that does not bound a $3$-ball. To see this, notice that there is a Mobius band $B$ bounded by $b$ with core $c$. For a thin tubular neighborhood $\nu$ of $b$, $B$ intersects $\partial \nu$ in the slope $m+l$. See \Cref{fig:nonegbraid}.

\begin{figure}
    \centering
\begingroup%
  \makeatletter%
  \providecommand\color[2][]{%
    \errmessage{(Inkscape) Color is used for the text in Inkscape, but the package 'color.sty' is not loaded}%
    \renewcommand\color[2][]{}%
  }%
  \providecommand\transparent[1]{%
    \errmessage{(Inkscape) Transparency is used (non-zero) for the text in Inkscape, but the package 'transparent.sty' is not loaded}%
    \renewcommand\transparent[1]{}%
  }%
  \providecommand\rotatebox[2]{#2}%
  \newcommand*\fsize{\dimexpr\f@size pt\relax}%
  \newcommand*\lineheight[1]{\fontsize{\fsize}{#1\fsize}\selectfont}%
  \ifx\svgwidth\undefined%
    \setlength{\unitlength}{321.26188263bp}%
    \ifx\svgscale\undefined%
      \relax%
    \else%
      \setlength{\unitlength}{\unitlength * \real{\svgscale}}%
    \fi%
  \else%
    \setlength{\unitlength}{\svgwidth}%
  \fi%
  \global\let\svgwidth\undefined%
  \global\let\svgscale\undefined%
  \makeatother%
  \begin{picture}(1,0.27139215)%
    \lineheight{1}%
    \setlength\tabcolsep{0pt}%
    \put(0,0){\includegraphics[width=\unitlength,page=1]{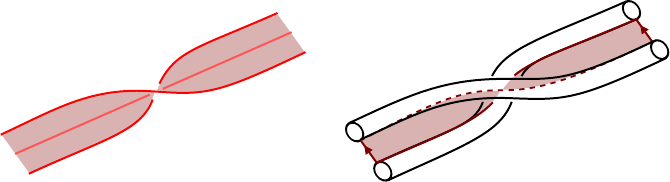}}%
    \put(0.28731321,0.21559104){\color[rgb]{1,0,0}\makebox(0,0)[lt]{\lineheight{1.25}\smash{\begin{tabular}[t]{l}$b$\end{tabular}}}}%
    \put(0.36465828,0.16792785){\color[rgb]{1,0.33333333,0.33333333}\makebox(0,0)[lt]{\lineheight{1.25}\smash{\begin{tabular}[t]{l}$c$\end{tabular}}}}%
    \put(0,0){\includegraphics[width=\unitlength,page=2]{nonegbraid.pdf}}%
  \end{picture}%
\endgroup%

    \caption{A non-example showing that steadiness is necessary in the definition of a horizontal surgery curve.}
    \label{fig:nonegbraid}
\end{figure}

Hence when we perform Dehn surgery by filling in that slope, we obtain an embedded projective plane $P$ in $M_1(b)$, essentially obtained by capping off $B$ with a disc. The boundary of a collar neighborhood of $P$ is then a $2$-sphere that bounds a twisted $I$-bundle over $\mathbb{R}P^2$ on one side and bounds a $3$-manifold containing a non-separating closed surface $S$ on the other side, neither of which can be a $3$-ball. 

It is a classical fact that $3$-manifolds admitting pseudo-Anosov flows must be irreducible. Hence we conclude that $M_1(b)$ cannot admit a pseudo-Anosov flow. In particular the statement of \Cref{thm:horsurpA} fails for $b$.
\end{noneg}

\section{Piecewise smooth curves} \label{sec:piecewisesmooth}

In this section we develop a partial generalization of horizontal Goodman surgery to make sense of performing surgery along certain piecewise smooth curves. Piecewise smooth curves naturally come up in the context of Goodman-Fried surgeries, and the material in this section will help simplify the presentation in the proof of \Cref{thm:horsuralmostequiv}. This discussion also serves as beginning steps towards understanding the relation between horizontal Goodman surgery on pseudo-Anosov flows and horizontal surgery on veering triangulations; see \Cref{subsec:vt}.

\subsection{Piecewise smooth horizontal surgery curves} \label{subsec:piecewisedefn}

Let $M$ be a closed oriented $3$-manifold. A \textbf{piecewise smooth curve} in $M$ is an embedded curve $c$ that is a concatenation of finitely many smooth embedded paths. The points at which the paths are joint together are called the \textbf{turns}. At each turn, the two tangent lines determined by the two paths may or may not agree. 

Let $\phi^t$ be a pseudo-Anosov flow on $M$.
A piecewise smooth curve is \textbf{positive/negative} if it is a concatenation of positive/negative paths respectively.

\begin{defn} \label{defn:piecewisehorsurcurve}
Let $\phi^t$ be a pseudo-Anosov flow on $M$.
Let $c$ be a positive piecewise smooth curve. Fix an instantaneous metric.
We say that $c$ is a \textbf{positive/negative piecewise smooth horizontal surgery curve} if at every crossing $(x,y,t)$ of $c$, the slope of $Tc|_y$ is greater than that of $d\phi^t(Tc|_x)$ on either side. More precisely, fix an orientation on $E^s|_x$ and transfer it to an orientation on $E^s|_y$ using $d\phi^t$. If $p^L_x$ and $p^R_x$, and $p^L_y$ and $p^R_y$ are the paths that constitute $c$ on the left and right of $x$, and $y$ respectively, then the condition is that $\slope(T{p^L_y}|_y) > \slope(d\phi^t(Tp^L_x|_x))$ and $\slope(T{p^R_y}|_y) > \slope(d\phi^t(Tp^R_x|_x))$.
\end{defn}

Note that this definition is independent on the choice of instantaneous metric.

\Cref{prop:horsurcurveflowisotopy} generalizes to this setting; the same proof carries over.

\begin{prop} \label{prop:piecewiseflowisotopy}
Being a piecewise smooth horizontal surgery curve is preserved under isotopy along orbits of $\phi^t$.
\end{prop}

\subsection{Approximation by smooth horizontal surgery curves} \label{subsec:piecewiseapprox}

Our next task is to make sense of `performing horizontal Goodman surgery on a piecewise smooth horizontal surgery curve'. The idea is to approximate these curves using smooth horizontal surgery curves.

Fix an instantaneous metric. 
We extend the notation $\slope(Tc|_v)$ to the case when $v$ is a turn to mean the minimal interval containing the slopes of the two paths that constitute $c$ at $v$.
We say that a sequence of piecewise smooth curves $(c_n)$ \textbf{converges towards a piecewise smooth curve $c$ in the $C^1$-topology} if for every $\epsilon>0$, there exists $N$ so that for every $n \geq N$,
\begin{itemize}
    \item the turns of $c_n$ are contained in the $\epsilon$-neighborhoods of the turns of $c$,
    \item $c_n$ is $\epsilon$-close to $c$ outside of the $\epsilon$-neighborhoods of the turns of $c$ in the $C^1$-topology, and
    \item $\slope(Tc_n|_x)$ stays bounded within the $\epsilon$-neighborhood of $\slope(Tc|_v)$ for $x$ in the $\epsilon$-neighborhood of each turn $v$ of $c$.
\end{itemize}

We say that a piecewise smooth horizontal surgery curve $c$ in $M \backslash \sing(\phi^t)$ is \textbf{generic} if:
\begin{itemize}
    \item no two turns of $c$ lie on the same closed orbit of $\phi^t$, and
    \item if two turns $x,y$ of $c$ lie on the same infinite orbit of $\phi^t$, say $y=\phi^t(x)$ for $t>0$, then $\slope(Tc|_y) > \slope(d\phi^t(Tc|_x))$ in the sense that every element in the former is larger than every element in the latter. 
\end{itemize}
In this paper, we will only consider piecewise smooth horizontal surgery curves that are generic. 
This is for the sake of simplifying the proofs of \Cref{prop:piecewiseapproxexist} and \Cref{prop:piecewiseapproxunique} below, and is sufficient for our applications in \Cref{sec:almostequiv}. 

The following proposition states that approximations by smooth horizontal surgery curves exist, with some additional control on where they lie relative to $c$.

\begin{prop} \label{prop:piecewiseapproxexist}
Let $c$ be a positive/negative generic piecewise smooth horizontal surgery curve. Let $A$ be an annulus in $M$ containing $c$ and transverse to the flow. Then there exists a sequence of positive/negative horizontal surgery curves on $A$ converging to $c$ on either side of $c$, respectively.
\end{prop}
\begin{proof}
We prove this when $c$ is positive. The case when $c$ is negative is similar. 
Fix an instantaneous metric.
Let $H$ and $h$ be the maximum and minimum slopes of the paths that make up $c$, respectively. Pick $T_1$ so that $\lambda^{-2T_1} < \frac{h}{H}$. Also pick $T_0$ to be a positive value less than the first return time of $A$.

Away from fixed neighborhoods $S_v$ of the turns of $c$, we can take a sequence of unions of paths $p_n$ lying on $A$ and converging to $c$ in the $C^1$-topology. As in \Cref{prop:horsurcurvestable}, we see that $Tp_n$ is positive and steady for large $n$. The point of the proof is to show that we can connect these up within $S_v$.

To that end, we divide the turns of $c$ into two types: 
\begin{enumerate}
    \item those that lie along infinite orbits, and
    \item those that lie along closed orbits. 
\end{enumerate}

For the turns of type (1), we can simply connect up the paths $p_n$ using paths $d_{v,n}$ that converge to $c \cap S_v$. Recall that this in particular means that the slope of $d_{v,n}$ is contained in an arbitrarily small neighborhood of $\slope(Tc|_v)$. 

By \Cref{claim:curveperturbcrossings}, every time $[T_0,T_1]$ crossing $(\overline{x},\overline{y},\overline{t})$ of $q_n := p_n \cup \bigcup_v d_{v,n}$ lies close to a crossing $(x,y,t)$ of $c$. If at least one of $x$ and $y$ is not a turn of $c$, then the steadiness condition is $C^1$-open hence holds for $(\overline{x},\overline{y},\overline{t})$. If both $x$ and $y$ are turns of $c$, then they are necessarily of type (1). By genericity, we have $\slope(Tc|_y) > \slope(d\phi^t(Tc|_x))$, which is a $C^1$-open condition hence holds for $(\overline{x},\overline{y},\overline{t})$, in particular implying steadiness. We conclude that $Tq_n$ is steady.

For the turns of type (2), we have to work harder. By the assumption of genericity, each such turn $v$ is the unique turn lying on a closed orbit $\gamma$. 
Up to shrinking $S_v$ we can assume that $N = \phi^{[-T_1,T_1]}(S_v)$ is an embedded cylinder or solid torus.
Let $f$ be the first return map on $S_v$. Let $P$ be the period of $\gamma$ and let $N = \lfloor \frac{T_1}{P} \rfloor$. 
We set $s_v = \bigcap_{k=-N}^N f^k(S_v)$ and suppose that the paths $p_n$ are chosen in the complement of $s_v$.

Our task is to extend the paths $p_n \cap S_v$ into paths $d_{v,n}$ on $S_v$ converging to $c \cap S_v$ such that
at each intersection between $d_{v,n}$ and $f(d_{v,n}),...,f^N(d_{v,n})$, the slope of $d_{v,n}$ is larger than that of $f^k(d_{v,n})$.

We think of $S_v$ as a subset of $\mathbb{R}^2$ with $v$ being the origin, so that we can talk about the quadrants of $S_v$. Notice that for a path $d$ of positive slope in the first or third quadrant of $S_v$, $d$ would never intersect $f^k(d)$. This is because, say if $d$ intersects $f^k(d)$ in $z_1=f^k(z_2)$ in the first quadrant, where $z_i \in d$, then $f^k$ maps $z_2$ to $z_1$, but $z_2$ lies to the top-right or bottom-left of $z_1$, contradicting the dynamics of $f$ in the first quadrant.

Meanwhile, note that a path $d$ of constant slope in the second or fourth quadrant would satisfy the second item above. 
Hence one way to choose the paths $(d_{v,n})$ is to take paths with constant slope lying arbitrarily close to the origin in the second or fourth quadrant (depending on which side of $c$ we want the constructed curves to converge to $c$ on), then quickly interpolate those into the given paths in the first and third quadrants. See \Cref{fig:piecewiseapproxexist} for a graphical summary.

Once we have such paths $d_{v,n}$, we apply \Cref{claim:curveperturbcrossings} again: Every time $[T_0,T_1]$ crossing $(\overline{x},\overline{y},\overline{t})$ of $c_n := q_n \cup \bigcup_v d_{v,n}$ lies close to a crossing $(x,y,t)$ of $c$. If at least one of $x$ and $y$ is not a turn, then the steadiness condition is $C^1$-open hence holds for $(\overline{x},\overline{y},\overline{t})$. If both $x$ and $y$ are turns, then it suffices to assume that they are both of type (2), in which case the steadiness condition holds by the arranged property on the intersections between $d_{v,n}$ and $f^k(d_{v,n})$.
\end{proof}

\begin{figure}
    \centering
    \fontsize{10pt}{10pt}\selectfont
\begingroup%
  \makeatletter%
  \providecommand\color[2][]{%
    \errmessage{(Inkscape) Color is used for the text in Inkscape, but the package 'color.sty' is not loaded}%
    \renewcommand\color[2][]{}%
  }%
  \providecommand\transparent[1]{%
    \errmessage{(Inkscape) Transparency is used (non-zero) for the text in Inkscape, but the package 'transparent.sty' is not loaded}%
    \renewcommand\transparent[1]{}%
  }%
  \providecommand\rotatebox[2]{#2}%
  \newcommand*\fsize{\dimexpr\f@size pt\relax}%
  \newcommand*\lineheight[1]{\fontsize{\fsize}{#1\fsize}\selectfont}%
  \ifx\svgwidth\undefined%
    \setlength{\unitlength}{206.00870448bp}%
    \ifx\svgscale\undefined%
      \relax%
    \else%
      \setlength{\unitlength}{\unitlength * \real{\svgscale}}%
    \fi%
  \else%
    \setlength{\unitlength}{\svgwidth}%
  \fi%
  \global\let\svgwidth\undefined%
  \global\let\svgscale\undefined%
  \makeatother%
  \begin{picture}(1,0.76269335)%
    \lineheight{1}%
    \setlength\tabcolsep{0pt}%
    \put(0,0){\includegraphics[width=\unitlength,page=1]{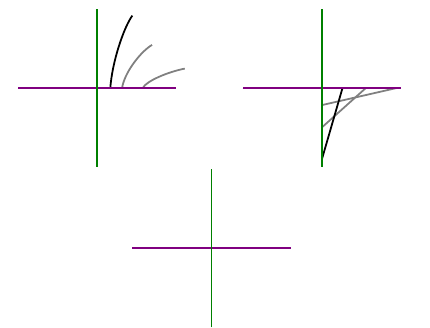}}%
    \put(0.32097448,0.72104326){\color[rgb]{0,0,0}\makebox(0,0)[lt]{\lineheight{1.25}\smash{\begin{tabular}[t]{l}$d$\end{tabular}}}}%
    \put(0.78845224,0.43163934){\color[rgb]{0,0,0}\makebox(0,0)[lt]{\lineheight{1.25}\smash{\begin{tabular}[t]{l}$d$\end{tabular}}}}%
    \put(0.39713595,0.62931388){\color[rgb]{0.50196078,0.50196078,0.50196078}\makebox(0,0)[lt]{\lineheight{1.25}\smash{\begin{tabular}[t]{l}$f^k(d)$\end{tabular}}}}%
    \put(0.84291608,0.49810031){\color[rgb]{0.50196078,0.50196078,0.50196078}\makebox(0,0)[lt]{\lineheight{1.25}\smash{\begin{tabular}[t]{l}$f^k(d)$\end{tabular}}}}%
    \put(0,0){\includegraphics[width=\unitlength,page=2]{piecewiseapproxexist.pdf}}%
  \end{picture}%
\endgroup%

    \caption{Top left: For any path $d$ of positive slope in the first quadrant, $d$ will not intersect $f^k(d)$. Top right: For any path $d$ of constant slope in the fourth quadrant, at each intersection between $d$ and $f^k(d)$ for $k>0$, the slope of $d$ is larger than that of $f^k(d)$. Bottom: Putting these facts together, we can construct a sequence of horizontal surgery curves converging to a given piecewise smooth horizontal surgery curve.}
    \label{fig:piecewiseapproxexist}
\end{figure}

The next proposition addresses the uniqueness of approximations by smooth horizontal surgery curves. 

\begin{prop} \label{prop:piecewiseapproxunique}
Let $c$ be a positive/negative piecewise smooth horizontal surgery curve. Suppose $(c^1_n)$ and $(c^2_n)$ are sequences of positive/negative horizontal surgery curves that converge to $c$. Then for large $n$, $c^1_n$ is isotopic to $c^2_n$ through positive/negative horizontal surgery curves, respectively.
\end{prop}
\begin{proof}
As in \Cref{prop:piecewiseapproxexist}, we choose neighborhoods $S_v$, $N_v$, and $s_v$ of the turns $v$.
For large $n$, we can interpolate between $c^1_n$ and $c^2_n$ outside of the neighborhoods $s_v$ for the type (2) turns $v$. The interpolating paths are $C^1$-close to $c$ hence have steady tangent fields, as in the proof of \Cref{prop:piecewiseapproxexist}.

Once again, the hard work is within the neighborhoods $s_v$.
As in \Cref{prop:piecewiseapproxexist}, we can talk about the quadrants of $S_v$ and consider the first return map $f$. We package most of the technical details of the argument in the following claim.

\begin{claim} \label{claim:piecewiseapproxconstantslope}
For large $n$, up to isotopy through surgery curves, we can assume that each $c^i_n$ is a path of constant slope in the second or fourth quadrant of $s_v$.
\end{claim}
\begin{proof}
We drop the superscript $i$ for convenience. If $c_n$ passes through the origin then we are done. Otherwise, without loss of generality suppose that $c_n$ passes through the interior of the fourth quadrant. 
We denote by $H_n$ and $h_n$ the maximum and minimum slopes of $c_n$ in this quadrant respectively.

Let $d_n$ be the path connecting the intersection points between $c_n$ and the coordinate axes with constant slope. An intermediate value theorem argument as in \Cref{prop:closedorbitcurve} shows that such a path of constant slope exists.
Note that the value of this constant slope is bounded between $H_n$ and $h_n$.

We isotope $c_n$ near its intersection points with the coordinate axes so that it agrees with $d_n$ in a small neighborhood of these points. We can do so while preserving the property that $c_n$ is a horizontal surgery curve throughout the isotopy: 
We can arrange for the intermediate curve in the isotopy to stay close to $c$ in the $C^1$-topology,
so that we only have to ensure that at each intersection between $c_n$ and $f(c_n),...,f^N(c_n)$, the slope of $c_n$ is larger than that of $f^k(c_n)$. Such intersection points lie in the fourth quadrant away from the coordinate axes, so we can modify $c_n$ in a small enough neighborhood without affecting these points. We illustrate this as the first step in \Cref{fig:piecewiseapproxbigonsteps}.

\begin{figure}
    \centering
    \fontsize{4pt}{4pt}\selectfont
    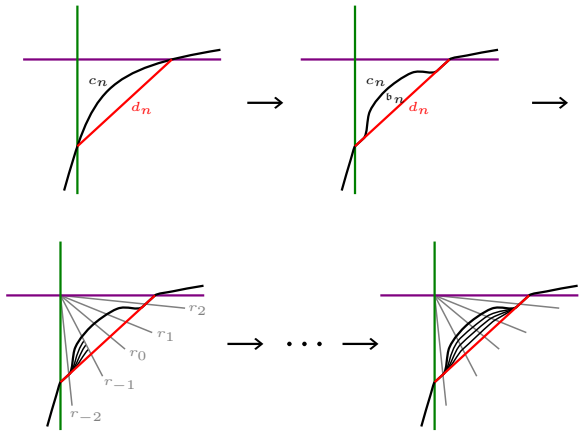
    \caption{A pictorial summary of the proof of \Cref{claim:piecewiseapproxconstantslope}.}
    \label{fig:piecewiseapproxbigonsteps}
\end{figure}

Up to a perturbation of $c_n$, $c_n$ and $d_n$ now bound a union of bigons $\mathfrak{b}_n$ contained within the interior of the fourth quadrant. We wish to isotope $c_n$ across $\mathfrak{b}_n$ to $d_n$, while ensuring that throughout the isotopy, 
\begin{itemize}
    \item the slopes of $c_n$ stay bounded within $[h_n,H_n]$, and
    \item at each intersection between $c_n$ and $f^k(c_n)$, the slope of the former is larger than that of latter,
\end{itemize}
so that $c_n$ remains a horizontal surgery curve throughout the isotopy.

We define an \textbf{overlapping bigon} to be a subset of $f^k(\mathfrak{b}_n) \cap \mathfrak{b}_n$ that is a bigon with one side on $c_n$ and the other side on $f^k(d_n)$. Here, a priori, we allow $k$ to take the value of any integer, but notice that there can only be finitely many values of $k$ that arise. We define an \textbf{innermost bigon} to be an overlapping bigon that does not properly contain the image of any overlapping bigon under some $f^j$. Since each $f^j$ does not have fixed points other than the origin, no overlapping bigon can properly contain the image of itself. This implies that innermost bigons exist.
Notice that each bigon in $\mathfrak{b}_n$ itself is an overlapping bigon and might be innermost. 
Our strategy is to isotope $c_n$ across innermost bigons one at a time.

Let $\mathfrak{b}'_n$ be an innermost bigon. We analyze how the paths in $f^j(c_n \cup d_n)$ can interact with and within $\mathfrak{b}'_n$. 

First, observe that none of the paths in $f^j(c_n \cup d_n)$ can enter then leave $\mathfrak{b}'_n$ through $c_n$, or enter then leave $\mathfrak{b}'_n$ through $d_n$. That is, \Cref{fig:piecewiseapproxbigonarcs} top left cannot occur. Indeed, a path in $f^j(c_n)$ cannot enter then leave $\mathfrak{b}'_n$ through $c_n$ by assumption of steadiness, and such a path cannot enter then leave $\mathfrak{b}'_n$ through $d_n$ for otherwise $\mathfrak{b}'_n$ would not be innermost, and similarly for a path in $f^j(d_n)$.

This implies that each path $p$ obtained by restricting $f^j(c_n \cup d_n)$ to $\mathfrak{b}'_n$ is a path between a point $p_c \in c_n$ and a point $p_d \in d_n$. Moreover, since the slope of $p$ is positive, either
\begin{itemize}
    \item the slope of $p$ at $p_c$ is greater than the slope of $c_n$ at $p_c$ and the slope of $p$ at $p_d$ is greater than the slope of $d_n$ at $p_d$, or
    \item the slope of $p$ at $p_c$ is less than the slope of $c_n$ at $p_c$ and the slope of $p$ at $p_d$ is less than the slope of $d_n$ at $p_d$.
\end{itemize}
That is, \Cref{fig:piecewiseapproxbigonarcs} top middle cannot occur. Since $c_n$ and $d_n$ are steady, the former is exactly the case when $j$ is negative while the latter is the case exactly when $j$ is positive.

We next claim that at each intersection between a path in $f^j(c_n \cup d_n)$ for $j>0$ and a path in $f^j(c_n \cup d_n)$ for $j<0$, the slope of the former is larger than that of the latter. Otherwise by the discussion in the previous paragraph, such paths would form a bigon and contradict the assumption that $\mathfrak{b}'_n$ was innermost. See \Cref{fig:piecewiseapproxbigonarcs} top right.

\begin{figure}
    \centering
    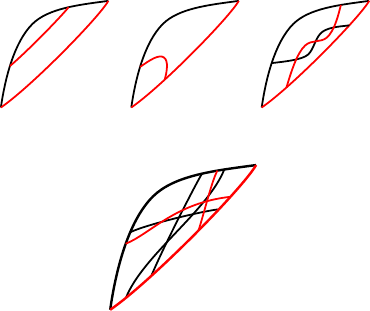
    \caption{Top: Configurations of $f^j(c_n \cup d_n)$ in $\mathfrak{b}'_n$ that are impossible. Bottom: A possible configuration of $f^j(c_n \cup d_n)$ in $\mathfrak{b}'_n$.}
    \label{fig:piecewiseapproxbigonarcs}
\end{figure}

In \Cref{fig:piecewiseapproxbigonarcs} bottom we show a possible configuration of $f^j(c_n \cup d_n)$ in $\mathfrak{b}'_n$. 

With this analysis, we are ready to define the isotopy. Let $r_0$ be a negative path in the fourth quadrant starting at the origin, and set $r_k=f^k(r_0)$. Let $R_k$ be the closed region in the fourth quadrant bounded by $r_k$ and $r_{k+1}$. Notice that only finitely many $R_k$ can intersect $\mathfrak{b}'_n$. We will define the isotopy across $\mathfrak{b}'_n$ in one region at a time. Here we think of an isotopy as a collection of paths, i.e. the trace of the isotopy.

We start with the region $R_k$ intersecting $\mathfrak{b}'_n$ with the minimum value of $k$. By our analysis above, we can define the isotopy within $\mathfrak{b}'_n \cap R_k$ so that at each intersection between an intermediate path in the isotopy and $f^j((c_n \cup d_n) \cap R_{k-j})$, the slope of the former is larger/smaller than the slope of the latter, for $j$ positive/negative respectively. In fact, since the paths in $f^j((c_n \cup d_n) \cap R_{k-j})$ for $j>0$ have slope $<H_n$ while the paths in $f^j((c_n \cup d_n) \cap R_{k-j})$ for $j<0$ have slope $>h_n$, we can ensure that the slope of the paths in the isotopy stay within $[h_n,H_n]$. We illustrate this as the second step in \Cref{fig:piecewiseapproxbigonsteps}.

When this is done, at each intersection between a path in $f^j(\mathfrak{b}'_n)$ for $j>0$ and a path in $f^j(\mathfrak{b}'_n)$ for $j<0$, the slope of the former is larger than that of the latter. Here by `a path in $f^j(\mathfrak{b}'_n)$' we mean the image of a path in the partially defined isotopy.

Inductively, suppose that we have defined the isotopy within $\mathfrak{b}'_n \cap R_k$, ..., $\mathfrak{b}'_n \cap R_{k+l}$ so that at each intersection between a path in $f^j(\mathfrak{b}'_n)$ for $j>0$ and a path in $f^j(\mathfrak{b}'_n)$ for $j<0$, the slope of the former is larger than that of the latter. 
We extend the isotopy within $\mathfrak{b}'_n \cap R_{k+l+1}$ so that at each intersection between the intermediate paths and $f^j(\mathfrak{b}'_n \cap R_{k+l+1-j})$, the slope of the former is larger/smaller than the slope of the latter, for $j$ positive/negative respectively.
As above, we can ensure that the slope of the paths in the isotopy stay within $[h_n,H_n]$.

At the end of the induction, we get an isotopy of $c_n$ across $\mathfrak{b}'_n$. This reduces the number of overlapping bigons by one. Hence repeating this enough times, we get the desired isotopy of $c_n$ to $d_n$.
\end{proof}

With $c^i_n$ arranged to be of the form in \Cref{claim:piecewiseapproxconstantslope}, we can isotope the portion of $c^i_n$ in the second or fourth quadrant through paths of constant slope to a path passing through the origin. 
For $n$ large enough so that $c_n$ is close enough to $c$, we can extend this into an isotopy of $c^i_n$ supported within $s_v$ so that the intermediate curves stay close to $c$ in the $C^1$-topology.
See \Cref{fig:piecewiseapproxorigin}. Moreover, throughout the isotopy, $c^i_n$ restricted to the second or fourth quadrant of $s_v$ is of constant slope, thus remains a horizontal surgery curve.

\begin{figure}
    \centering
    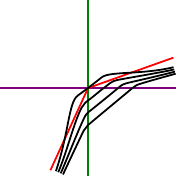
    \caption{With $c^i_n$ arranged to be of the form in \Cref{claim:piecewiseapproxconstantslope}, we can isotope $c^i_n$ so that it passes through the origin in $S_v$.}
    \label{fig:piecewiseapproxorigin}
\end{figure}

This reduces it to the case when each $c^i_n$ passes through the origin in $s_v$. In this case we can simply linearly interpolate between the curves.
\end{proof}

\Cref{prop:piecewiseapproxexist} and \Cref{prop:piecewiseapproxunique} allows for the following definition.

\begin{defn} \label{defn:piecewisehorsur}
Let $c$ be a positive/negative piecewise smooth horizontal surgery curve. Let $n$ be a positive/negative integer, respectively. The flow obtained by \textbf{horizontal Goodman surgery along $c$ with coefficient $\frac{1}{n}$}, which we denote by $\phi^t_{\frac{1}{n}}(c)$, is defined to be the flow $\phi^t_{\frac{1}{n}}(c_k)$ for large $k$, where $(c_k)$ is a sequence of horizontal surgery curves converging to $c$.
\end{defn}

\begin{cor} \label{cor:piecewiseinvisotopy}
Let $c_s$, $s \in [0,1]$ be a family of positive/negative piecewise smooth horizontal surgery curves. Then for every positive/negative integer $n$, respectively, $\phi^t_{\frac{1}{n}}(c_0)$ and $\phi^t_{\frac{1}{n}}(c_1)$ are orbit equivalent.
\end{cor}
\begin{proof}
It suffices to show that for a fixed $s_0$, the set of $s$ such that $\phi^t_{\frac{1}{n}}(c_s)$ is orbit equivalent to $\phi^t_{\frac{1}{n}}(c_{s_0})$ contains a neighborhood of $s_0$. Suppose otherwise, then there exists a sequence of $s_k$ converging to $s_0$ such that each $\phi^t_{\frac{1}{n}}(c_{s_k})$ is not orbit equivalent to $\phi^t_{\frac{1}{n}}(c_{s_0})$.

For each $k$, we can take a sequence $d_{s_k,j}$ of smooth horizontal surgery curves converging to $c_{s_k}$ such that $\phi^t_{\frac{1}{n}}(d_{s_k,j}) \cong \phi^t_{\frac{1}{n}}(c_{s_k})$. Then we can extract a diagonal sequence $d_{s_k,j_k}$ converging to $c_{s_0}$. But then $\phi^t_{\frac{1}{n}}(c_{s_k}) \cong \phi^t_{\frac{1}{n}}(d_{s_k,j_k}) \cong \phi^t_{\frac{1}{n}}(c_{s_0})$ contradicting our hypothesis.
\end{proof}

\subsection{Interaction with Goodman-Fried surgeries} \label{subsec:horsurcurveGF}

We discuss one application of the generalization of horizontal Goodman surgery to piecewise smooth curves, which concerns how horizontal Goodman surgeries interact with Goodman-Fried surgeries.

Let $\phi^t$ be a pseudo-Anosov flow on a closed oriented $3$-manifold $M$.
Let $c$ be a generic piecewise smooth horizontal surgery curve.
Let $\beta$ be an orientation-preserving closed orbit that intersects $c$ at exactly one point $z$. Here recall that a closed orbit $\beta$ is said to be \textbf{orientation-preserving} if $E^s|_\beta$ and $E^u|_\beta$ are orientable line bundles.

Let $k \in \mathbb{Z}$. Take some orbit equivalence between the restriction of $\phi^t$ to $M \backslash \beta$ and the restriction of $\phi^t_\frac{1}{k}(\beta)$ to $M_\frac{1}{k}(\beta) \backslash \beta$. Transfer $c \backslash z$ from the former to the latter. 

The closure of this image of $c \backslash z$ is an interval with endpoints lying on, in general, distinct points on $\beta$. Via isotopy along orbits, pushing one endpoint along $\beta$, we get a closed curve $c^k$ on $M_\frac{1}{k}(\beta)$. Note that in general $c^k$ will have a turn on $\beta$. This is why we required a generalization of horizontal Goodman surgery to piecewise smooth curves. Note also that $c^k$ is not uniquely defined, since there is indeterminacy in which endpoint we push along $\beta$ and how many times we do so. Nevertheless, the following lemma is always true.

\begin{lemma} \label{lemma:horsurcurveGF}
In the setup above, $c^k$ is a generic piecewise smooth horizontal surgery curve.
\end{lemma}
\begin{proof}
We can divide the crossings $(x,y,t)$ of $c^k$ into two types:
\begin{itemize}
    \item Crossings where $x$ and $y$ do not lie on $\beta$
    \item Crossings where $x$ and $y$ lie on $\beta$
\end{itemize}

Since $c^k$ can be obtained by isotoping $c \backslash z$ along orbits, the fact that $c$ is a horizontal surgery curve implies that the crossings of the first type satisfy \Cref{defn:piecewisehorsurcurve}, similar to \Cref{prop:piecewiseflowisotopy}. For the crossings of the second type, we must have $x=y$, so \Cref{defn:piecewisehorsurcurve} holds from the dynamics of the first return map on $\beta$.

That $c^k$ is generic follows from the fact that $c$ is generic and that none of the turns of $c$, other than possibly $z$, can lie on $\beta$.
\end{proof}

We discuss one setting where the isotopy class of the curve $c^k$ above is uniquely determined. 
Let $c_s$, $s \in [0,1]$ be a family of generic piecewise smooth horizontal surgery curves.
Let $\beta$ be an orientation preserving closed orbit that intersects the union of $c_s$ at exactly one point $z \in c_1$.

Transfer $\bigcup_{s \in [0,1]} c_s \backslash \beta$ from $M \backslash \beta$ to $M_\frac{1}{k}(\beta) \backslash \beta$ using an orbit equivalence as above.
Then up to isotoping $c_s$ along orbits, we can arrange for $c_s$ to limit onto a curve $c^k_1$. See \Cref{fig:horsurcurveGF} second and third frames. 
By \Cref{lemma:horsurcurveGF}, $c^k_1$ is a generic piecewise smooth horizontal surgery curve.

\begin{figure}
    \centering
    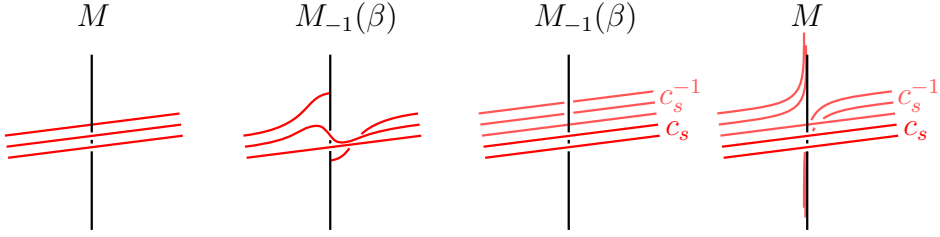
    \caption{Pushing $c_s$ across an orientation preserving closed orbit $\beta$. In this example we take $k=-1$.}
    \label{fig:horsurcurveGF}
\end{figure}

In particular, by \Cref{prop:piecewiseapproxexist}, we can extend $c^k_1$ into a family of horizontal surgery curves $c^k_s$ for $s \in [1,2]$. We transfer $\bigcup_{s \in [1,2]} c^k_s \backslash \beta$ back into $M$. Note that $c_s$ for $s \in [0,1)$ differs from $c^k_s$ for $s \in (1,2]$ homotopically by $|k|$ multiples of $\beta$. 

\begin{prop} \label{prop:horsurcurveGF}
In the setup above, $\phi^t_{\frac{1}{n}}(c_0)$ is almost equivalent to $\phi^t_{\frac{1}{n}}(c^k_2)$.
\end{prop}
\begin{proof}
The orbit equivalence $M \backslash \beta \cong M_\frac{1}{k}(\beta) \backslash \beta$ induces almost equivalences between $\phi^t_{\frac{1}{n}}(c_0)$ and $(\phi^t_{\frac{1}{k}}(\beta))_{\frac{1}{n}}(c_0)$, and between $\phi^t_{\frac{1}{n}}(c^k_2)$ and $(\phi^t_{\frac{1}{k}}(\beta))_{\frac{1}{n}}(c^k_2)$.

Meanwhile we have $(\phi^t_{\frac{1}{k}}(\beta))_{\frac{1}{n}}(c_0) \cong (\phi^t_{\frac{1}{k}}(\beta))_{\frac{1}{n}}(c^k_2)$ by \Cref{cor:piecewiseinvisotopy}.
\end{proof}

We refer to the operation of obtaining $c^k_s$, $s \in (1,2]$, from $c_s$, $s \in [0,1)$ as \textbf{pushing $c_s$ across $\beta$ (with coefficient $\frac{1}{k}$)}. See \Cref{fig:horsurcurveGF} for a schematic summary of the whole procedure. This operation will play a large role in the proof of \Cref{thm:horsuralmostequiv}.

\section{Almost equivalence} \label{sec:almostequiv}

The main goal in this section is to prove \Cref{thm:horsuralmostequiv}, which we restate below for the reader's convenience.

\begin{thm:horsuralmostequiv}
Let $\phi^t$ be a transitive pseudo-Anosov flow on a closed oriented $3$-manifold $M$. Let $c$ be a positive/negative horizontal surgery curve for $\phi^t$. Then for every positive/negative integer $n$, respectively, the flow $\phi^t_{\frac{1}{n}}(c)$ is almost equivalent to $\phi^t$.
\end{thm:horsuralmostequiv}

The proof of \Cref{thm:horsuralmostequiv} will span \Cref{subsec:sawtooth} to \Cref{subsec:slidehomotopy}. We refer to \Cref{subsec:introalmostequivproof} for a sketch of the proof. The remaining short \Cref{subsec:almostequivcor} is devoted to explaining how \Cref{cor:intropAmap} to \Cref{cor:introinteriortorus} can be deduced from \Cref{thm:horsuralmostequiv}.

As a notational shorthand, in this section we will denote two flows $\phi^t_1$ and $\phi^t_2$ being almost equivalent as $\phi^t_1 \sim \phi^t_2$.

\subsection{Building a sawtooth} \label{subsec:sawtooth}

Let $\phi^t$ be a transitive pseudo-Anosov flow on a closed oriented $3$-manifold $M$. Let $c$ be a positive horizontal surgery curve for $\phi^t$. 
Fix an instantaneous metric.

Choose a thin annulus $\mathbb{A}$ containing $c$ that is transverse to the flow and disjoint from $\sing(\phi^t)$. Note that the induced stable/unstable foliation $\mathbb{A}^{s/u}$ is a foliation of $\mathbb{A}$ by non-separating arcs. Fix some identification of $\mathbb{A}$ with a subset of $S^1 \times S^1$ so that the leaves of $\mathbb{A}^{s/u}$ are sent to intervals lying on the vertical/horizontal circles $\{u\} \times S^1$/$S^1 \times \{s\}$ respectively. In particular we can consider $c$ to be a graph of some increasing function $\{u=c(s)\}$. 

\begin{defn} \label{defn:sawtooth}
A \textbf{sawtooth} is a subset $S$ of $\mathbb{A}$ of the form $S = \bigcup_i \{(u,s) \in \mathbb{A} \mid s_i \leq s \leq s_{i+1}, c(s) \leq u \leq u_i\}$, where $\{s_i\}$ is a cyclic sequence of points on $S^1$ and $u_i \geq c(s_{i+1})$ for every $i$. See \Cref{fig:sawtoothdefn}.

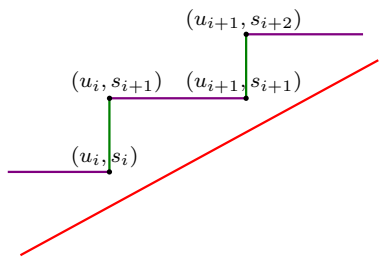
\begin{figure}
    \centering
    \fontsize{8pt}{8pt}\selectfont
\begingroup%
  \makeatletter%
  \providecommand\color[2][]{%
    \errmessage{(Inkscape) Color is used for the text in Inkscape, but the package 'color.sty' is not loaded}%
    \renewcommand\color[2][]{}%
  }%
  \providecommand\transparent[1]{%
    \errmessage{(Inkscape) Transparency is used (non-zero) for the text in Inkscape, but the package 'transparent.sty' is not loaded}%
    \renewcommand\transparent[1]{}%
  }%
  \providecommand\rotatebox[2]{#2}%
  \newcommand*\fsize{\dimexpr\f@size pt\relax}%
  \newcommand*\lineheight[1]{\fontsize{\fsize}{#1\fsize}\selectfont}%
  \ifx\svgwidth\undefined%
    \setlength{\unitlength}{139.83617948bp}%
    \ifx\svgscale\undefined%
      \relax%
    \else%
      \setlength{\unitlength}{\unitlength * \real{\svgscale}}%
    \fi%
  \else%
    \setlength{\unitlength}{\svgwidth}%
  \fi%
  \global\let\svgwidth\undefined%
  \global\let\svgscale\undefined%
  \makeatother%
  \begin{picture}(1,0.6461887)%
    \lineheight{1}%
    \setlength\tabcolsep{0pt}%
    \put(0,0){\includegraphics[width=\unitlength,page=1]{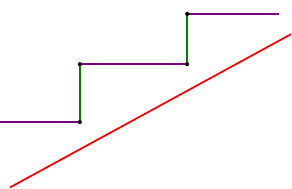}}%
    \put(0.16982608,0.25490351){\color[rgb]{0,0,0}\makebox(0,0)[lt]{\lineheight{1.25}\smash{\begin{tabular}[t]{l}$(u_i,s_i)$\end{tabular}}}}%
    \put(0.16982608,0.45522077){\color[rgb]{0,0,0}\makebox(0,0)[lt]{\lineheight{1.25}\smash{\begin{tabular}[t]{l}$(u_i,s_{i+1})$\end{tabular}}}}%
    \put(0.47881634,0.45521598){\color[rgb]{0,0,0}\makebox(0,0)[lt]{\lineheight{1.25}\smash{\begin{tabular}[t]{l}$(u_{i+1},s_{i+1})$\end{tabular}}}}%
    \put(0.47880969,0.62417632){\color[rgb]{0,0,0}\makebox(0,0)[lt]{\lineheight{1.25}\smash{\begin{tabular}[t]{l}$(u_{i+1},s_{i+2})$\end{tabular}}}}%
  \end{picture}%
\endgroup%

    \caption{Definition of a sawtooth.}
    \label{fig:sawtoothdefn}
\end{figure}

We call the points $(u_i,s_{i+1})$ the \textbf{peaks} of $S$ and the points $(u_i,s_i)$ the \textbf{troughs} of $S$.
\end{defn}

Our first goal is to build a specific sawtooth. Let $\delta$ be a small positive number, the exact value of which is to be chosen later. We will want our sawtooth to lie within the $\delta$-neighborhood of $c$ in $\mathbb{A}$, which we denote as $N_\delta(c)$.

\begin{claim} \label{claim:sawtooth0}
There exists a sawtooth $S_0 \subset N_\delta(c)$ such that:
\begin{itemize}
    \item $S_0$ contains an annulus $A_0$ carrying a foliation $\widehat{\mathcal{H}}_0$ by positive horizontal surgery curves, where $c$ is a leaf of $\widehat{\mathcal{H}}_0$.
    \item The troughs of $S_0$ lie on the other boundary component of $A_0$, which we denote as $c_1$.
    \item The troughs of $S_0$ lie on distinct orientation-preserving closed orbits $\gamma_i$.
    \item Each $\gamma_i$ intersects $c_1$ at exactly one point.
\end{itemize}
See \Cref{fig:sawtooth0} for an illustration of $S_0$.

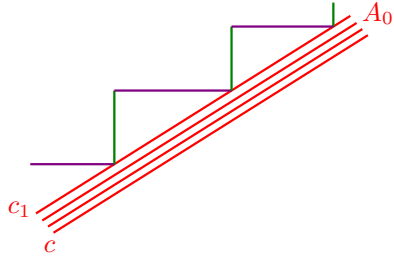
\begin{figure}
    \centering
    \fontsize{10pt}{10pt}\selectfont
\begingroup%
  \makeatletter%
  \providecommand\color[2][]{%
    \errmessage{(Inkscape) Color is used for the text in Inkscape, but the package 'color.sty' is not loaded}%
    \renewcommand\color[2][]{}%
  }%
  \providecommand\transparent[1]{%
    \errmessage{(Inkscape) Transparency is used (non-zero) for the text in Inkscape, but the package 'transparent.sty' is not loaded}%
    \renewcommand\transparent[1]{}%
  }%
  \providecommand\rotatebox[2]{#2}%
  \newcommand*\fsize{\dimexpr\f@size pt\relax}%
  \newcommand*\lineheight[1]{\fontsize{\fsize}{#1\fsize}\selectfont}%
  \ifx\svgwidth\undefined%
    \setlength{\unitlength}{158.28318678bp}%
    \ifx\svgscale\undefined%
      \relax%
    \else%
      \setlength{\unitlength}{\unitlength * \real{\svgscale}}%
    \fi%
  \else%
    \setlength{\unitlength}{\svgwidth}%
  \fi%
  \global\let\svgwidth\undefined%
  \global\let\svgscale\undefined%
  \makeatother%
  \begin{picture}(1,0.60390131)%
    \lineheight{1}%
    \setlength\tabcolsep{0pt}%
    \put(0,0){\includegraphics[width=\unitlength,page=1]{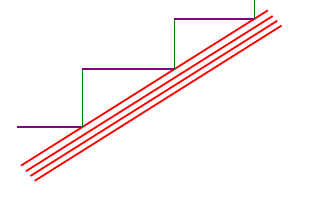}}%
    \put(-0.0038252,0.10216412){\color[rgb]{1,0,0}\makebox(0,0)[lt]{\lineheight{1.25}\smash{\begin{tabular}[t]{l}$c_1$\end{tabular}}}}%
    \put(0.08235617,0.00890914){\color[rgb]{1,0,0}\makebox(0,0)[lt]{\lineheight{1.25}\smash{\begin{tabular}[t]{l}$c$\end{tabular}}}}%
    \put(0.84042657,0.56125211){\color[rgb]{1,0,0}\makebox(0,0)[lt]{\lineheight{1.25}\smash{\begin{tabular}[t]{l}$A_0$\end{tabular}}}}%
  \end{picture}%
\endgroup%

    \caption{Illustration of the sawtooth $S_0$.}
    \label{fig:sawtooth0}
\end{figure}
\end{claim}
\begin{proof}
Let $c_1$ be a curve on $\mathbb{A}$ lying above and close to $c$ in the $C^1$-topology. Recall that since $\phi^t$ is transitive, the set of closed orbits is dense. In fact, the set of orientation-preserving closed orbits is also dense. Hence up to an arbitrarily small perturbation of $c_1$, we can pick points $(u_i,s_i)$ lying on $c_1$ such that
\begin{itemize}
    \item The sawtooth $S_0$ with troughs $(u_i,s_i)$ lies within $N_\delta(c)$.
    \item $(u_i,s_i)$ lie on distinct orientation-preserving closed orbits $\gamma_i$.
\end{itemize}
Up to a further arbitrarily small perturbation of $c_1$, we can assume that each $\gamma_i$ only intersects $c_1$ in $(u_i,s_i)$.

Now we can pick a foliation $\widehat{\mathcal{H}}_0$ of the annulus between $c$ and $c_1$ by closed curves. If $c_1$ is close enough to $c$ in the $C^1$-topology in the first place, then $T\widehat{\mathcal{H}}_0$ is close to $Tc$, hence the leaves of $\widehat{\mathcal{H}}_0$ are horizontal surgery curves (see \Cref{prop:horsurcurvetosurann}).
\end{proof}

We now enlarge $S_0$ by `raising its steps' slightly.

\begin{claim} \label{claim:sawtooth1}
There exists $\eta>0$ such that the set $\bigcup_i \phi^{[0,\infty)}(\{u_i\} \times [s_{i+1}, s_{i+1} + \eta])$ can only intersect the set $\bigcup_i ([u_i,u_{i+1}] \times [s_{i+1}, s_{i+1} + \eta])$ within $\bigcup_i (\{u_i\} \times [s_i, s_{i+1} + \eta])$.
\end{claim}
\begin{proof}
Consider the intersection points between the set $\phi^{[0,\infty)}(u_i,s_{i+1})$ and $\mathbb{A}$. This is some set that limits onto the set of intersection points between $\gamma_i$ and $\mathbb{A}$ from the stable direction. Notice that $\phi^{[0,\infty)}(u_i,s_{i+1})$ can only intersect $\bigcup_i [u_i,u_{i+1}] \times \{s_{i+1}\}$ in $(u_i,s_{i+1})$, since the orbits $\gamma_i$ are assumed to be distinct. Similarly, $\gamma_i$ can only intersect $\bigcup_i [u_i,u_{i+1}] \times \{s_{i+1}\}$ in $(u_i,s_i)$. This implies that there is some definite vertical distance between $\phi^{[0,\infty)}(u_i,s_{i+1})$ and $\bigcup_i [u_i,u_{i+1}] \times \{s_{i+1}\}$ away from $\{u_i\} \times [s_i, s_{i+1}]$, or more precisely, there exists $\tau >0$ such that for every $(u,s) \in \phi^{[0,\infty)}(u_i,s_{i+1}) \cap \mathbb{A}$ and every $(u',s') \in \bigcup_i [u_i,u_{i+1}] \times \{s_{i+1}\}$ with $(u,s),(u',s') \notin \{u_i\} \times [s_i, s_{i+1}]$, we have $|s-s'| > \tau$. See \Cref{fig:sawtooth1}.

\begin{figure}
    \centering
    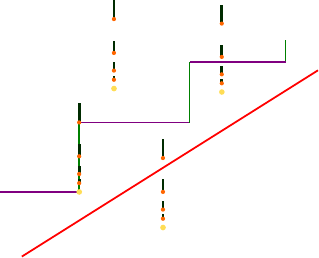
    \caption{An illustration of the proof of \Cref{claim:sawtooth1}. The orange dots are the intersection of $\phi^{[0,\infty)}(u_i,s_{i+1})$ with $\mathbb{A}$. They limit to the yellow dots, which are the intersection of $\gamma_i$ with $\mathbb{A}$. The dark green lines are the intervals of intersection between $\phi^{[0,\infty)}(\{u_i\} \times [s_{i+1}, s_{i+1} + \eta])$ and $\mathbb{A}$.}
    \label{fig:sawtooth1}
\end{figure}

There exists $\eta$ so that the intervals of intersection between $\phi^{[0,\infty)}(\{u_i\} \times [s_{i+1}, s_{i+1} + \eta])$ and $\mathbb{A}$ are of length $< \tau$. This uses the fact that the flow $\phi^t$ is exponentially contracting in the stable direction. This value of $\eta$ satisfies the claim.
\end{proof}

For each $i$, there is a first return map $r_i:\{u_i\} \times [s_i, s_{i+1} + \eta] \to \{u_i\} \times [s_i, s_{i+1} + \eta]$ induced by the flow $\phi^t$. Pick a large enough $N_i$ so that $r_i^{N_i}(\{u_i\} \times [s_i, s_{i+1} + \eta]) \subset \{u_i\} \times [s_i, s_i + \eta]$. Then we can define a map $g_i : \{u_i\} \times [s_i, s_{i+1} + \eta] \to \{u_i\} \times [s_i, s_{i+1} + \eta]$ as follows:
\begin{itemize}
    \item Start by applying $r_i^{N_i}$ to get into $\{u_i\} \times [s_i, s_i + \eta]$.
    \item Slide along the horizontal lines to get into $\{u_{i-1}\} \times [s_i, s_i + \eta] \subset \{u_{i-1}\} \times [s_{i-1}, s_i + \eta]$.
    \item Apply $r_{i-1}^{N_{i-1}}$, and repeat until we cycle back to $\{u_i\} \times [s_i, s_{i+1} + \eta]$.
\end{itemize}
By the Brouwer fixed point theorem, $g_i$ has a fixed point, which we denote by $(u_i, s_i + \eta_i)$.
We let $S$ be the sawtooth given by the union of $S_0$ and $\bigcup_i [u_i,u_{i+1}] \times [s_{i+1}, s_{i+1} + \eta_{i+1}]$.

\subsection{Building the homotopy} \label{subsec:buildhomotopy}

For each $t \in [0,1]$, we define a line field $L_t$ on the closure of $S \backslash A_0$ by $\slope(L_t|_{(u,s)})=t \cdot \slope(Tc_1|_{(u,c_1(u))})$. 
We then define a map $g^t_i : \{u_i\} \times [s_i, s_{i+1} + \eta_{i+1}] \to \{u_i\} \times [s_i, s_{i+1} + \eta_{i+1}]$ as follows:
\begin{itemize}
    \item Start by applying $r_i^{N_i}$ to get into $\{u_i\} \times [s_i, s_i + \eta_i]$.
    \item Slide along the trajectories of $L_t$ to get into $\{u_{i-1}\} \times [s_{i-1}, s_i + \eta_i]$.
    \item Apply $r_{i-1}^{N_{i-1}}$, and repeat until we cycle back to $\{u_i\} \times [s_i, s_{i+1} + \eta_{i+1}]$.
\end{itemize}
Note the following properties of $g^t_i$:
\begin{itemize}
    \item $g^0_i=g_i$, hence in particular $g^0_i(u_i,s_{i+1}+\eta_{i+1})=(u_i,s_{i+1}+\eta_{i+1})$.
    \item $g^1_i(u_i,s_i)=(u_i,s_i)$ since in this case we are simply following along $c_1$.
    \item For each $s \in [s_i, s_{i+1} + \eta_{i+1}]$, the $s$-coordinate of $g^t_i(u_i,s)$ is a strictly decreasing function of $t$.
\end{itemize}
From these properties, it follows that for every $(u_i,s)$, there is a unique value $t_i(s)$ such that $g^{t_i(s)}_i(u_i,s)=(u_i,s)$.

For each $s \in [s_1,s_2+\eta_2]$, we let $\widehat{H_s}$ be the union of $L_{t_1(s)}$ trajectories one goes through for defining $g^{t_1(s)}_1$. 
See \Cref{fig:sawtoothfol} for an illustration. 

\begin{figure}
    \centering
    \fontsize{10pt}{10pt}\selectfont
\begingroup%
  \makeatletter%
  \providecommand\color[2][]{%
    \errmessage{(Inkscape) Color is used for the text in Inkscape, but the package 'color.sty' is not loaded}%
    \renewcommand\color[2][]{}%
  }%
  \providecommand\transparent[1]{%
    \errmessage{(Inkscape) Transparency is used (non-zero) for the text in Inkscape, but the package 'transparent.sty' is not loaded}%
    \renewcommand\transparent[1]{}%
  }%
  \providecommand\rotatebox[2]{#2}%
  \newcommand*\fsize{\dimexpr\f@size pt\relax}%
  \newcommand*\lineheight[1]{\fontsize{\fsize}{#1\fsize}\selectfont}%
  \ifx\svgwidth\undefined%
    \setlength{\unitlength}{144.94271346bp}%
    \ifx\svgscale\undefined%
      \relax%
    \else%
      \setlength{\unitlength}{\unitlength * \real{\svgscale}}%
    \fi%
  \else%
    \setlength{\unitlength}{\svgwidth}%
  \fi%
  \global\let\svgwidth\undefined%
  \global\let\svgscale\undefined%
  \makeatother%
  \begin{picture}(1,0.66068036)%
    \lineheight{1}%
    \setlength\tabcolsep{0pt}%
    \put(0,0){\includegraphics[width=\unitlength,page=1]{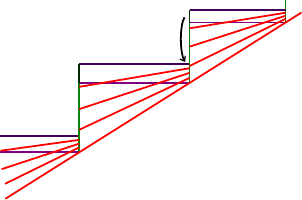}}%
    \put(0.45776108,0.52639107){\color[rgb]{0,0,0}\makebox(0,0)[lt]{\lineheight{1.25}\smash{\begin{tabular}[t]{l}$g_{i+1}$\end{tabular}}}}%
    \put(0,0){\includegraphics[width=\unitlength,page=2]{sawtoothfol.pdf}}%
    \put(0.140161,0.31542153){\color[rgb]{0,0,0}\makebox(0,0)[lt]{\lineheight{1.25}\smash{\begin{tabular}[t]{l}$g_i$\end{tabular}}}}%
  \end{picture}%
\endgroup%

    \caption{The foliation $\widehat{\mathcal{H}}$ on $S$.}
    \label{fig:sawtoothfol}
\end{figure}

An important property of each $\widehat{H}_s$ is the following.

\begin{claim}
For small enough values of $\delta$, $T\widehat{H}_s$ is steady for every $s \in [s_1,s_2+\eta_2)$. 
\end{claim}
\begin{proof}
Let $H$ and $h$ be the maximum and minimum slopes of $Tc$ respectively. Then the maximum and minimum slopes of $T\widehat{H_s}$ are $t_1(s)H$ and $t_1(s)h$ respectively. Pick $T_1$ so that $\lambda^{-2T_1} < \frac{h}{H}$. Then the steadiness condition for $T\widehat{H_s}$ automatically holds for triples $(\overline{x},\overline{y},\overline{t})$ where $\overline{t} \geq T_1$.
We also let $T_0$ be the first return time of $\mathbb{A}$.

By \Cref{claim:curveperturbcrossings}, up to shrinking $\delta$, we can assume that any time $[T_0,T_1]$ crossing $(\overline{x},\overline{y},\overline{t})$ of $\widehat{H_s}$ lies arbitrarily close to a crossing $(x,y,t)$ of $c$.
In particular $\frac{1}{t_1(s)} \slope(d\phi^{\overline{t}}(T\widehat{H}_s|_{\overline{x}}))$ is arbitrarily close to $\slope(d\phi^t(Tc|_x))$ and $\frac{1}{t_1(s)} \slope(T\widehat{H}_s|_{\overline{y}})$ is arbitrarily close to $\slope(Tc|_y)$, so the steadiness condition for $c$ at $(x,y,t)$ implies that for $T\widehat{H}_s$ at $(\overline{x},\overline{y},\overline{t})$.
\end{proof}

For each $i$, consider the set $T_i = \phi^{[0,N_i p_i]}(\{u_i\} \times [s_i, s_{i+1} + \eta_{i+1}])$ where $p_i$ is the period of $\gamma_i$. We think of $T_i$ as the image of a rectangle foliated by orbit segments of $\phi^t$. 

\begin{figure}
    \centering
\begingroup%
  \makeatletter%
  \providecommand\color[2][]{%
    \errmessage{(Inkscape) Color is used for the text in Inkscape, but the package 'color.sty' is not loaded}%
    \renewcommand\color[2][]{}%
  }%
  \providecommand\transparent[1]{%
    \errmessage{(Inkscape) Transparency is used (non-zero) for the text in Inkscape, but the package 'transparent.sty' is not loaded}%
    \renewcommand\transparent[1]{}%
  }%
  \providecommand\rotatebox[2]{#2}%
  \newcommand*\fsize{\dimexpr\f@size pt\relax}%
  \newcommand*\lineheight[1]{\fontsize{\fsize}{#1\fsize}\selectfont}%
  \ifx\svgwidth\undefined%
    \setlength{\unitlength}{326.47169459bp}%
    \ifx\svgscale\undefined%
      \relax%
    \else%
      \setlength{\unitlength}{\unitlength * \real{\svgscale}}%
    \fi%
  \else%
    \setlength{\unitlength}{\svgwidth}%
  \fi%
  \global\let\svgwidth\undefined%
  \global\let\svgscale\undefined%
  \makeatother%
  \begin{picture}(1,0.3736677)%
    \lineheight{1}%
    \setlength\tabcolsep{0pt}%
    \put(0,0){\includegraphics[width=\unitlength,page=1]{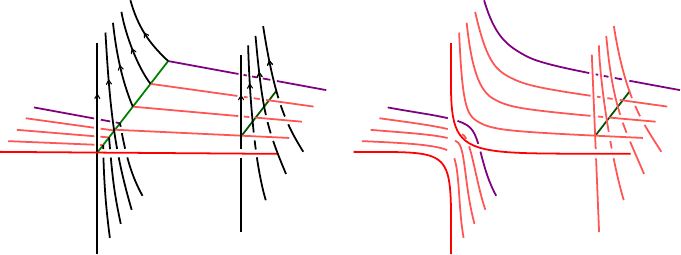}}%
    \put(0.1012354,0.23152658){\color[rgb]{0,0,0}\makebox(0,0)[lt]{\lineheight{1.25}\smash{\begin{tabular}[t]{l}$\gamma_i$\end{tabular}}}}%
    \put(0,0){\includegraphics[width=\unitlength,page=2]{perpsmooth.pdf}}%
    \put(0.62127069,0.23152791){\color[rgb]{0,0,0}\makebox(0,0)[lt]{\lineheight{1.25}\smash{\begin{tabular}[t]{l}$\gamma_i$\end{tabular}}}}%
    \put(0.22116063,0.10668225){\color[rgb]{1,0.33333333,0.33333333}\makebox(0,0)[lt]{\lineheight{1.25}\smash{\begin{tabular}[t]{l}$H^\perp_s$\end{tabular}}}}%
    \put(0.74119581,0.10668354){\color[rgb]{1,0.33333333,0.33333333}\makebox(0,0)[lt]{\lineheight{1.25}\smash{\begin{tabular}[t]{l}$H_s$\end{tabular}}}}%
  \end{picture}%
\endgroup%

    \caption{Left: Constructing the curves $H^\perp_s$. $H^\perp_s$ possibly self-intersects along intervals of intersection between $T_i$ and $S_1$ that lie away from the lines $u=u_i$. Right: Isotoping the curves $H^\perp_s$ along flow lines to obtain the curves $H_s$ that are horizontal with respect to the flow.}
    \label{fig:perpsmooth}
\end{figure}

For each $s \in [s_1,s_2+\eta_2]$, we let $H^\perp_s$ be the closed curve obtained by concatenating the intervals of $\widehat{H}_s$ with the orbit segments on $T_i$. Here we orient the intervals of $\widehat{H}_s$ in the direction of decreasing $u$. That $H^\perp_s$ is a closed curve follows from the construction of $\widehat{H}_s$. 
The $\perp$ in our notation of $H^\perp_s$ is to remind us that $H^\perp_s$ consists of both the \emph{horizontal} intervals $\widehat{H_s}$ and the \emph{vertical} orbits in $T_i$. See \Cref{fig:perpsmooth} left.

Note that each $H^\perp_s$ may not be embedded. The potential self-intersections come from the intervals of intersection between $T_i$ and $S_1$ that lie away from the lines $u=u_i$. See \Cref{fig:perpsmooth} left again.

By \Cref{claim:sawtooth1}, the vertical orbits in $H^\perp_{s_2+\eta_2}$ can only meet $\widehat{H}_{s_2+\eta_2}$ along the lines $u=u_i$, hence $H^\perp_{s_2+\eta_2}$ is embedded. In fact, by construction $H^\perp_{s_2+\eta_2}$ lies on a leaf of the unstable foliation $\Lambda^u$, so that leaf must contain a closed orbit $\gamma$, with $H^\perp_{s_2+\eta_2}$ isotopic to a multiple of $\gamma$.

Similarly, by the last item in \Cref{claim:sawtooth0}, $H^\perp_{s_1}$ is embedded. Thus we conclude that there exists $\epsilon>0$ so that if $H^\perp_s$ self-intersects, then $s \in [s_1+\epsilon,s_2+\eta_2-\epsilon]$.

Furthermore, since there can only be finitely many intervals where the self-intersections of $H^\perp_s$ occurs, we can perturb $\widehat{H}_s$ for $s \in [s_1+\epsilon,s_2+\eta_2-\epsilon]$ near these intervals so that there are only finitely many values of $s$ for which $H^\perp_s$ self-intersects, and such that each $H^\perp_s$ has at most one self-intersection point, while preserving the steadiness of $T\widehat{H}_s$.

With this arranged, we can isotope $H^\perp_s$, for each $s \in [s_1,s_2+\eta_2]$, along orbits to get a curve $H_s$ so that $TH_s \neq T\phi^t$ at every point, i.e. $H_s$ is horizontal everywhere. See \Cref{fig:perpsmooth} right. If $H^\perp_s$ is embedded then $H_s$ will be embedded. Note however, that $H_s$ is only piecewise smooth, and not smooth. This is because the slope of $TH^\perp_s$ at the two endpoints of each vertical orbit segment are equal. After isotopy, one of these endpoints is raised while the other is lowered, so the slopes no longer agree. Nevertheless, we can make sense of the following claim by \Cref{defn:piecewisehorsurcurve}.

\begin{claim} \label{claim:smoothtohorsurcurve}
If $H_s$ is embedded then it is a generic piecewise smooth horizontal surgery curve.
\end{claim}
\begin{proof}
Recall that $T\widehat{H}_s$ is steady. As in \Cref{prop:horsurcurveflowisotopy}, there is a map from the crossings of $H_s$ to those of $\widehat{H}_s$. This map is no longer injective (nor surjective) since the isotopy along orbits identifies endpoints of $\widehat{H}_s$. Nevertheless, the steadiness condition for $H_s$ implies that for $\widehat{H}_s$ as in that proof.

To see that $H_s$ is generic, observe that the turns of $H_s$ lie on the leaves of $\Lambda^s$ containing $\gamma_i$, which are distinct, hence the turns in fact lie on distinct orbits.
\end{proof}

\subsection{Chain of almost equivalences from the homotopy} \label{subsec:slidehomotopy}

The following claim initiates our chain of almost equivalences in proving \Cref{thm:horsuralmostequiv}.

\begin{claim} \label{claim:almostequivstart}
Up to shrinking $\epsilon$, $\phi^t_{\frac{1}{n}}(c)$ is almost equivalent to $\phi^t_{\frac{1}{n}}(H_{s_1+\epsilon})$.
\end{claim}
\begin{proof}
Recall the annulus $A_0 \subset S_1$. This is foliated by horizontal surgery curves. 
Since $H^\perp_{s_1}$ is embedded, up to shrinking $\epsilon$, there exists a foliated sub-annulus $A_1 \subset A_0$ containing $c_1$ so that $H^\perp_s$ does not intersect $A_1$ for $s \in (s_1,s_1+\epsilon]$. Let $c_{1-\epsilon}$ be the boundary component of $A_1$ other than $c_1$. 

The foliation on $A_0$ determines an isotopy of horizontal surgery curves from $c$ to $c_{1-\epsilon}$. Hence by \Cref{cor:invcurveisotopy}, we have $\phi^t_{\frac{1}{n}}(c) \cong \phi^t_{\frac{1}{n}}(c_{1-\epsilon})$.

Now consider the union of $H_s$ for $s \in [s_1,s_1+\epsilon]$ and the annulus $A_1$ restricted to $M \backslash \bigcup \gamma_i$. This is an embedded surface that intersects tubular neighborhoods of each $\gamma_i$ in a curve of homology class $\mu + N_i \lambda$, in the notation of \Cref{eg:closedorbitsur}.

Consider $\overline{M} = M_{\frac{1}{N_i}}(\gamma_i)$ and $\overline{\phi}^t = \phi^t_{\frac{1}{N_i}}(\gamma_i)$. By identifying $M \backslash \bigcup \gamma_i$ with $\overline{M} \backslash \bigcup \gamma_i$, we get a surface in the latter that intersects tubular neighborhoods of each $\gamma_i$ in the meridian. Thus up to isotoping the surface along orbits, we can take a closure to get an annulus $\overline{A_1}$ in $\overline{M}$. 

The annulus $\overline{A_1}$ admits a foliation $\overline{\mathcal{H}}$ by generic piecewise smooth horizontal surgery curves, induced from $H_s$ for $s \in [s_1,s_1+\epsilon]$ and the foliation on $A_1$. Here, the fact that the leaf of $\overline{\mathcal{H}}$ passing through the $\gamma_i$ is a generic piecewise smooth horizontal surgery curve follows from \Cref{lemma:horsurcurveGF}.

Hence we get an almost equivalence $\phi^t_{\frac{1}{n}}(c_{1-\epsilon}) \sim \phi^t_{\frac{1}{n}}(H_{s_1+\epsilon})$ as in \Cref{prop:horsurcurveGF}.
\end{proof}

For any sub-interval $[s',s''] \subset [s_1+\epsilon,s_2-\epsilon]$ that does not contain any value of $s$ for which $H_s$ self-intersects, we have $\phi^t_{\frac{1}{n}}(H_{s'}) \cong \phi^t_{\frac{1}{n}}(H_{s''})$ by \Cref{cor:piecewiseinvisotopy}.

The following claim shows that we can `jump over' the values of $s$ for which $H_s$ self-intersects.

\begin{claim} \label{claim:almostequivjump}
Suppose $H_{s_0}$ self-intersects, then for small $\epsilon>0$, $\phi^t_{\frac{1}{n}}(H_{s_0-\epsilon})$ is almost equivalent to $\phi^t_{\frac{1}{n}}(H_{s_0+\epsilon})$.
\end{claim}
\begin{proof}
Recall how self-intersection points occur: Some orbit segment in $H^\perp_{s_0}$ intersects a horizontal interval in $H^\perp_{s_0}$ at an interior point $x$. We may assume that in isotoping $H^\perp_{s_0}$ along orbits to get $H_{s_0}$, we only isotope the horizontal intervals near their endpoints, hence that the point $x$ remains fixed. In particular we can use the coordinates on $\mathbb{A}$ to describe the situation around $x$.

Now choose some small neighborhood $S \subset \mathbb{A}$ of $x$ so that $N=\phi^{[-\tau,\tau]}(S)$ is a cylinder that intersects $H_{s_0}$ in a union of two arcs with endpoints on $\phi^{[-\tau,\tau]}(\partial S)$, for some $\tau$.
For example, one can choose $N$ to be a neighborhood of the orbit segment in $H^\perp_s$ giving rise to the self-intersection.

Then up to shrinking $\epsilon$, $N$ intersects $H_s$ in a union of two arcs with endpoints on $\phi^{[-\tau,\tau]}(\partial S)$ for all $s \in [s_0-\epsilon,s_0+\epsilon]$. In particular, we can consider the projection of $H_s \cap N$ onto $S$. Since the self-intersection point $x$ is isolated, there are two possibilities for the form of $H_s \cap N$ depending on whether the horizontal interval in $H^\perp_{s_0}$ moves in front or behind of the vertical orbit segment as $s$ varies. The two cases can be distinguished by whether
\begin{enumerate}
    \item the projection of $H_s \cap N$ onto $S$ bounds a bigon for $s \in [s_0-\epsilon,s_0)$, illustrated in \Cref{fig:selfintersect1}, or
    \item the projection of $H_s \cap N$ onto $S$ bounds a bigon for $s \in (s_0,s_0+\epsilon]$, illustrated in \Cref{fig:selfintersect2}.
\end{enumerate}
We focus on case (1) for now. We will indicate how to modify our arguments to tackle case (2) at the end of the proof.

\begin{figure}
    \centering
    \fontsize{8pt}{8pt}\selectfont
    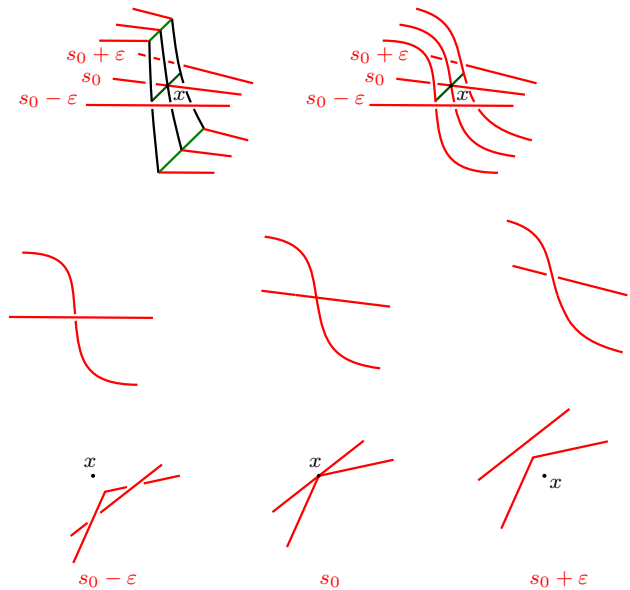
    \caption{Case (1) of the form of $H_s \cap N$. Top left: The homotopy $H^\perp_s$ in $N$, $s \in [s_0-\epsilon, s_0+\epsilon]$. Top right: The homotopy $H_s$ in $N$, $s \in [s_0-\epsilon, s_0+\epsilon]$. Middle and bottom: $H_s \cap N$ and its projection onto $S$, at times $s_0-\epsilon, s_0, s_0+\epsilon$.}
    \label{fig:selfintersect1}
\end{figure}

\begin{figure}
    \centering
    \fontsize{8pt}{8pt}\selectfont
    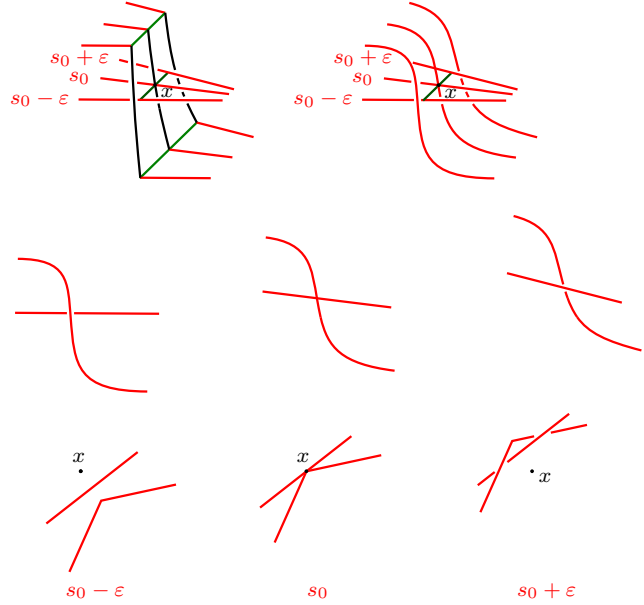
    \caption{Case (2) of the form of $H_s \cap N$. Top left: The homotopy $H^\perp_s$ in $N$, $s \in [s_0-\epsilon, s_0+\epsilon]$. Top right: The homotopy $H_s$ in $N$, $s \in [s_0-\epsilon, s_0+\epsilon]$. Middle and bottom: $H_s \cap N$ and its projection onto $S$, at times $s_0-\epsilon, s_0, s_0+\epsilon$.}
    \label{fig:selfintersect2}
\end{figure}

For $\epsilon$ small enough, we may isotope $H_s$ for $s \in [s_0-\epsilon,s_0+\epsilon]$ so that it agrees with $H_{s_0}$ outside of $N$. This can be achieved by isotoping $\widehat{H}_s$ so that it agrees with $\widehat{H}_{s_0}$ outside of $N$. For $\epsilon$ small enough, $T\widehat{H}_{s_0}$ is close to $T\widehat{H}_s$ so the intervals throughout the isotopy have steady tangent field. By the same argument as in \Cref{claim:smoothtohorsurcurve}, this shows that $H_s$ is a generic piecewise smooth horizontal surgery curve throughout the isotopy.

With this arranged, the homotopy $H_s$ between $H_{s_0 \pm \epsilon}$ can now be thought of as two families of arcs within $N$. We refer to the family of arcs coming from the horizontal intervals in $H^\perp_s$ as $a_s$, and the other family of arcs as $b_s$. 

In fact, up to shrinking $\epsilon$, we can further assume that after isotopy, $\bigcup_s T\widehat{H}_s$ is steady, which implies that the curves obtained by joining the portion of $H_{s_0}$ outside $N$ with $a_s$ and $b_{s'}$ in $N$ are generic horizontal surgery curves, for any $s,s'$ such that the curve is embedded.

Right now the self-intersection point $x$ lies on an infinite orbit. Our next few steps are aimed at modifying $H_s$ so that the self-intersection point lies on a closed orbit.

There exists a neighborhood $\nu \subset \mathbb{A}$ of $x$ such that given any $z \in \nu$, we can homotope $b_s$ so that the turn in $b_s$ traces out a curve in $\nu$ passing through $z$, while maintaining the property that the curves obtained by joining the portion of $H_{s_0}$ outside $N$ with $a_s$ and $b_{s'}$ in $N$ are horizontal surgery curves, for any $s,s'$ such that the curve is embedded. One way to see this is to again consider the intervals $\widehat{H}_s$. For $\nu$ small, a perturbation of $b_s$ as above can be achieved by a small perturbation of $\widehat{H}_s$, which preserves the property that $\bigcup_s T\widehat{H}_s$ is steady.

Pick a number $T_1$ so that $\lambda^{-2T_1} < \frac{\min\slope(H_s)}{\max\slope(H_s)}$, where the maximum and minimum are taken over all $s \in [s_0-\epsilon, s_0+\epsilon]$. 
Since the set of turns of $H_s$ lie on a finite number of intervals lying along the stable leaves of the closed orbits $\gamma_i$, we may assume that $\nu$ is disjoint from $\phi^{[-T_1,T_1]}(v)$ for every turn $v$ of $H_s$ lying outside of $N$.
This implies that for every perturbation of $b_s$ as above, the curves obtained by joining the portion of $H_{s_0}$ outside $N$ with $a_s$ and $b_{s'}$ in $N$ are \emph{generic} piecewise smooth horizontal surgery curves, for any $s,s'$ such that the curve is embedded.

We now use the fact that $\phi^t$ is transitive to pick an orientation preserving closed orbit $\beta$ passing through a point $z \in \nu$. We perturb $H_s$ as above so that $z$ is the turn of some $b_{s_b}$. Meanwhile, we let $s_a$ be the unique value so that $z$ lies on $a_{s_a}$. Now we can reparametrize the two families of arcs $a_s$ and $b_s$, so that the self-intersection point of the curves $H_s$ is at $z$, while preserving the following properties:
\begin{itemize}
    \item Each $H_s$ is a generic piecewise smooth horizontal surgery curve.
    \item The projection of $H_s \cap N$ onto $S$ bounds a bigon for $s \in [s_0-\epsilon,s_0)$
    \item Each $H_s$ agrees with $H_{s_0}$ outside of $N$.
    \item The curves obtained by joining the portion of $H_{s_0}$ outside $N$ with $a_s$ and $b_{s'}$ in $N$ are generic horizontal surgery curves, for any $s,s'$ such that the curve is embedded.
\end{itemize}
See \Cref{fig:selfintersectperturb} for an illustration of our modifications.

\begin{figure}
    \centering
    \fontsize{8pt}{8pt}\selectfont
    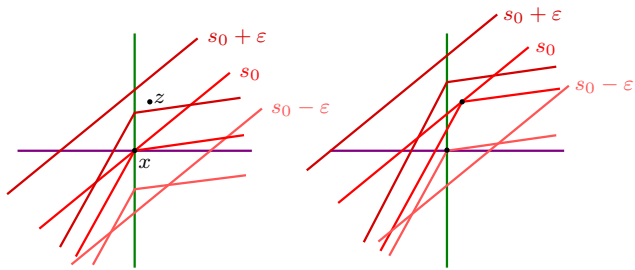
    \caption{Modifying $H_s$ so that the self-intersection point changes from $x$ to $z$, which lies on a closed orbit $\beta$.}
    \label{fig:selfintersectperturb}
\end{figure}

Furthermore, up to an arbitrarily small perturbation of $H_{s_0}$, we can assume that $\beta$ only meets $H_{s_0}$ at $z$. Then up to shrinking $\epsilon$, we can assume that the union of $H_s$, $s \in [s_0-\epsilon,s_0+\epsilon]$, only intersects $\beta$ in $z$.
 
The strategy of the proof at this point is to push the arcs $a_s$ and $b_s$ through $\beta$ with various coefficients, so that we can isotope $a_{s_0-\epsilon}$ to $a_{s_0+\epsilon}$ and $b_{s_0-\epsilon}$ to $b_{s_0+\epsilon}$ without them passing through each other. 
The sequence of operations we will perform is illustrated in \Cref{fig:localmoves1}.
Here in labelling $\beta$ by a number $k$, we mean performing the operation in $M_{\frac{1}{k}}(\beta)$. The moves where we switch this label to a different number is applying an identification $M_{\frac{1}{k}}(\beta) \backslash \beta \cong M_{\frac{1}{k'}}(\beta) \backslash \beta$. We describe and justify our sequence of moves more precisely in the following paragraphs.

\begin{figure}
    \centering
    \fontsize{6pt}{6pt}\selectfont
    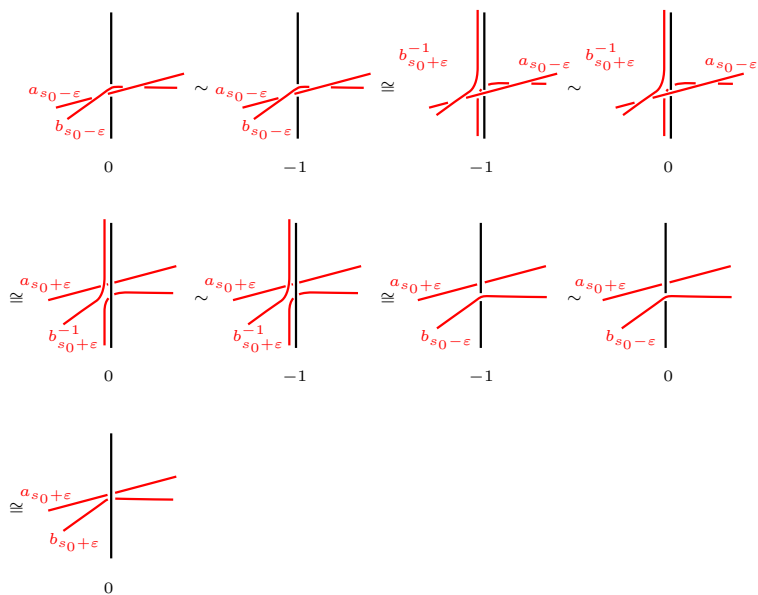
    \caption{Performing a sequence of local moves to show that $\phi^t_{\frac{1}{n}}(H_{s_0-\epsilon}) \sim \phi^t_{\frac{1}{n}}(H_{s_0+\epsilon})$ in case (1).}
    \label{fig:localmoves1}
\end{figure}

We start from $\phi^t_{\frac{1}{n}}(H_{s_0-\epsilon})$. We take an orbit equivalence between $\phi^t$ restricted to $M \backslash \beta$ and $\phi^t_{-1}(\beta)$ restricted to $M_{-1}(\beta) \backslash \beta$. 
This orbit equivalence induces an almost equivalence 
$$\phi^t_{\frac{1}{n}}(H_{s_0-\epsilon}) \sim (\phi^t_{-1}(\beta))_{\frac{1}{n}}(H_{s_0-\epsilon})$$
which is the first step in \Cref{fig:localmoves1}.

We now push $b_s$ across $\beta$ with coefficient $-1$. This is essentially the operation discussed in \Cref{subsec:horsurcurveGF}:
Up to isotopy along orbits, the image of the arcs $b_s$ in $M_{-1}(\beta) \backslash \beta$, for $s \in [s_0-\epsilon,s_0)$, converge towards an arc $b^{-1}_{s_0}$. Let $H^{-1}_{s_0}$ be the curve given by the composition of $H_{s_0}$ outside $N$ with $a_{s_0-\epsilon}$ and $b^{-1}_{s_0}$. This is a generic piecewise smooth horizontal surgery curve by \Cref{lemma:horsurcurveGF}. We apply \Cref{prop:piecewiseapproxexist} to obtain a family of horizontal surgery curves $H^{-1}_s$, $s \in [s_0,s_0+\epsilon]$ on the other side of $\beta$ in $M_{-1}(\beta)$. Up to shrinking $\epsilon$, we can arrange for $H^{-1}_s$ to contain the portion of $H_{s_0}$ outside $N$ and with $a_{s_0-\epsilon}$, similar to how we arrange for $H_s$ to agree with $H_{s_0}$ above. We let $b^{-1}_s$ be the portion of $H^{-1}_s$ that differs from $b_{s_0-\epsilon}$. Each $b^{-1}_s$ lies in the union of $N$ and a neighborhood of $\beta$.

Note that the composition of $H_{s_0}$ outside $N$ with $a_{s_0+\epsilon}$ and $b^{-1}_{s_0}$ is also a generic piecewise smooth horizontal surgery curve, hence up to shrinking $\epsilon$, we can assume that the curves $H^1_s$ given by the composition of $H_{s_0}$ outside $N$ with $a_{s_0+\epsilon}$ and $b^{-1}_s$, for $s \in [s_0, s_0+\epsilon]$, are generic piecewise smooth horizontal surgery curves as well. This will come into play later.

For now, notice that \Cref{cor:piecewiseinvisotopy} implies that 
$$(\phi^t_{-1}(\beta))_{\frac{1}{n}}(H_{s_0-\epsilon}) \cong (\phi^t_{-1}(\beta))_{\frac{1}{n}}(H^{-1}_{s_0+\epsilon}).$$
This is the second step in \Cref{fig:localmoves1}. We caution the reader that this orbit equivalence does \emph{not} send $\beta$ to $\beta$.

We transfer the curves $H^{-1}_s$ and arcs $b^{-1}_s$ back to $M$ using the orbit equivalence $M \backslash \beta \cong M_{-1}(\beta) \backslash \beta$. Once again the orbit equivalence induces an almost equivalence 
$$(\phi^t_{-1}(\beta))_{\frac{1}{n}}(H^{-1}_{s_0+\epsilon}) \sim \phi^t_{\frac{1}{n}}(H^{-1}_{s_0+\epsilon}).$$
This is the third step in \Cref{fig:localmoves1}.

Notice that up to isotoping $b^{-1}_{s_0+\epsilon}$ along orbits, the curves $H^0_s$ obtained by joining the portion of $H_{s_0}$ outside $N$ with $a_s$ and $b^{-1}_{s_0+\epsilon}$, for $s \in [s_0-\epsilon,s_0+\epsilon]$, determines an isotopy of curves. We claim that up to shrinking $\epsilon$, we can assume that these curves are generic piecewise smooth horizontal surgery curves.

To see this, notice that up to isotoping along orbits, as $s \to s_0$, $b^{-1}_s$ limits onto an arc $b^0_{s_0}$ which contains the image of $b_{s_0} \backslash \beta$ isotoped along orbits. Note that $b^0_{s_0}$ is different from $b_{s_0}$. The latter is obtained by connecting up two intervals in $\widehat{H}_{s_0}$ by some vertical orbit segment $\beta^0 \subset \beta$ then isotoping along orbits, while the former is obtained by connecting up the same two intervals using the vertical orbit segment $\beta \backslash \beta^0$, then isotoping along orbits. See \Cref{fig:arcsurgery}. 

\begin{figure}
    \centering
    \fontsize{10pt}{10pt}\selectfont
    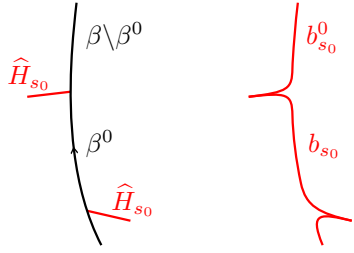
    \caption{Comparing the arcs $b_{s_0}$ and $b^0_{s_0}$.}
    \label{fig:arcsurgery}
\end{figure}

It suffices to show that the curves obtained by joining the portion of $H_{s_0}$ outside $N$ with $a_s$ and $b^0_{s_0}$, for $s \in [s_0-\epsilon,s_0+\epsilon]$, are piecewise smooth horizontal surgery curves. But this follows from the fact that $T\widehat{H}_{s_0}$ is steady, as in the proof of \Cref{claim:smoothtohorsurcurve}.

Thus \Cref{cor:piecewiseinvisotopy} implies that 
$$\phi^t_{\frac{1}{n}}(H^{-1}_{s_0+\epsilon}) = \phi^t_{\frac{1}{n}}(H^0_{s_0-\epsilon}) \cong \phi^t_{\frac{1}{n}}(H^0_{s_0+\epsilon}).$$
This is the fourth step in \Cref{fig:localmoves1}. As in the second step, we caution the reader that this orbit equivalence does \emph{not} send $\beta$ to $\beta$.

As before, the orbit equivalence $M \backslash \beta \cong M_{-1}(\beta) \backslash \beta$ induces the almost equivalence 
$$\phi^t_{\frac{1}{n}}(H^0_{s_0+\epsilon}) \sim (\phi^t_{-1}(\beta))_{\frac{1}{n}}(H^0_{s_0+\epsilon}).$$
This is the fifth step in \Cref{fig:localmoves1}.

Recall the family of piecewise smooth horizontal surgery curves $H^1_s$ given by the composition of $H_{s_0}$ outside $N$ with $a_{s_0+\epsilon}$ and $b^{-1}_s$, for $s \in [s_0, s_0+\epsilon]$. We can extend this family to $s \in [s_0-\epsilon,s_0)$ since by assumption the curves given by the composition of $H_{s_0}$ outside $N$ with $a_{s_0+\epsilon}$ and $b_s$ are piecewise smooth horizontal surgery curves. Hence by \Cref{cor:piecewiseinvisotopy}, we have 
$$(\phi^t_{-1}(\beta))_{\frac{1}{n}}(H^0_{s_0+\epsilon}) = (\phi^t_{-1}(\beta))_{\frac{1}{n}}(H^1_{s_0+\epsilon}) \cong (\phi^t_{-1}(\beta))_{\frac{1}{n}}(H^1_{s_0-\epsilon}).$$
This is the sixth step in \Cref{fig:localmoves1}.

The orbit equivalence $M \backslash \beta \cong M_{-1}(\beta) \backslash \beta$ induces the almost equivalence 
$$(\phi^t_{-1}(\beta))_{\frac{1}{n}}(H^1_{s_0-\epsilon}) \sim \phi^t_{\frac{1}{n}}(H^1_{s_0-\epsilon}).$$
This is the seventh step in \Cref{fig:localmoves1}.

Finally, the curves $H^2_s$ given by the composition of $H_{s_0}$ outside $N$ with $a_{s_0+\epsilon}$ and $b_s$, for $s \in [s_0-\epsilon, s_0+\epsilon]$ are generic horizontal surgery curves, so by \Cref{cor:piecewiseinvisotopy}, we have 
$$\phi^t_{\frac{1}{n}}(H^1_{s_0-\epsilon}) =  \phi^t_{\frac{1}{n}}(H^2_{s_0-\epsilon}) \cong \phi^t_{\frac{1}{n}}(H^2_{s_0+\epsilon}) = \phi^t_{\frac{1}{n}}(H_{s_0+\epsilon}).$$
This is the eighth step in \Cref{fig:localmoves1}.

The argument is very similar when the local form of $H_s$ in $N$ is of case (2), i.e. that indicated in \Cref{fig:selfintersect2}. We first modify $H_s$ so that the self-intersection point lies on a closed orbit $\beta$. This part of the proof carries over word-for-word. With this arranged, we use the sequence of operations indicated in \Cref{fig:localmoves2} to produce the desired almost equivalence. We let the reader fill in the precise description of these operations.

\begin{figure}
    \centering
    \fontsize{6pt}{6pt}\selectfont
    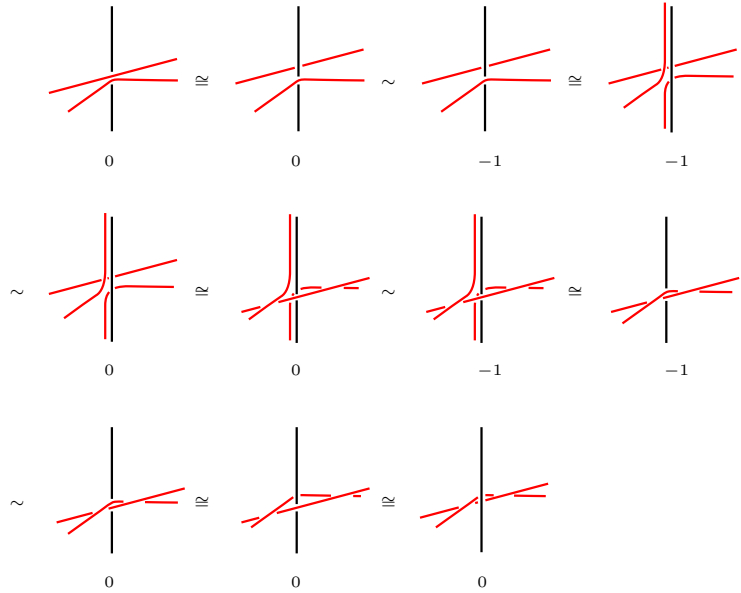
    \caption{Performing a sequence of local moves to show that $\phi^t_{\frac{1}{n}}(H_{s_0-\epsilon}) \sim \phi^t_{\frac{1}{n}}(H_{s_0+\epsilon})$ in case (2).}
    \label{fig:localmoves2}
\end{figure}

\end{proof}

Applying \Cref{claim:almostequivstart} and \Cref{claim:almostequivjump} repeatedly, we conclude that $\phi^t_{\frac{1}{n}}(c) \sim \phi^t_{\frac{1}{n}}(H_{s_2+\eta_2-\epsilon})$. To conclude the proof of \Cref{thm:horsuralmostequiv}, observe that for small $\epsilon$, $H_{s_2+\eta_2-\epsilon}$ lies close to an embedded curve on the unstable leaf of the closed orbit $\gamma$, hence can be isotoped through orbits to lie on an annulus in an arbitrarily small neighborhood of $\gamma$. By \Cref{thm:Goodman=Fried}, we then have $\phi^t_{\frac{1}{n}}(H_{s_2+\eta_2-\epsilon}) \sim \phi^t$.

\subsection{Deriving the corollaries} \label{subsec:almostequivcor}

\begin{proof}[Proof of \Cref{cor:intropAmap}]
This follows from combining \Cref{prop:pAmapstraightcurve}, \Cref{prop:pAmaptwist} and \Cref{thm:horsuralmostequiv}.
\end{proof}

\begin{proof}[Proof of \Cref{cor:introscalloptorus}]
This follows from combining \Cref{prop:scalloptoruscurve}, \Cref{cor:scalloptorusintersection1} and \Cref{thm:horsuralmostequiv}.
\end{proof}

\begin{proof}[Proof of \Cref{cor:introinteriortorus}]
This follows from combining \Cref{prop:interiortoruscurve}, \Cref{prop:interiortorussur} and \Cref{thm:horsuralmostequiv}.
\end{proof}

\section{Discussion and further questions} \label{sec:questions}

We discuss some future directions coming out of this paper.

\subsection{Generalizations of the surgery operation}

The reader will notice that we have been restricting to oriented $3$-manifolds throughout this paper. This assumption is needed in order to make sense of positive/negative curves, as well as other signs in our discussions. The definition of pseudo-Anosov flows, however, does not require orientability. Hence a natural future direction is

\begin{prob}
Generalize horizontal Goodman surgery to (pseudo-)Anosov flows on non-orientable $3$-manifolds.
\end{prob}

The difficulty here is that one could have orientation-reversing crossings. More precisely, given a curve $c$ that whose tangent field never lies on $T\phi \cup E^{s/u}$, it is always possible to choose some orientation of $E^{s/u}$ along $c$. The problematic crossings $(x,y,t)$ are ones where $d\phi^t$ preserves the orientation on $E^s$ yet reverses the orientation on $E^u$, or vice versa. In that case, one would need more care in defining cone fields for them to nest correctly, in order to show that the surgered flow is (pseudo-)Anosov.

Closely related to this problem is the task of generalizing the surgery operation to multiple curves. 

\Cref{constr:introhorsur} generalizes immediately to performing positive/negative surgeries along a collection of positive/negative curves $c_1,...,c_k$ where $T(\bigcup c_i)$ is steady, respectively. However, a collection of positive/negative curves that are individually horizontal surgery curves may not satisfy this condition. In fact, inspecting the proof of \Cref{thm:horsurpA}, it would seem that without additional hypotheses, the cone field argument would fall apart easily in such a setting, that is, without the steadiness condition on the union of curves.

\subsection{Transitivity in \Cref{thm:horsuralmostequiv}}

We note that transitivity plays an essential role in the proof of \Cref{thm:horsuralmostequiv}. This hypothesis is used for (i) constructing the sawtooth in \Cref{claim:sawtooth0}, and (ii) allowing one to perform the local operations in \Cref{fig:localmoves1} and \Cref{fig:localmoves2}.

\begin{conj}
The assumption of transitivity is necessary in \Cref{thm:horsuralmostequiv}. 

That is, there exists a non-transitive pseudo-Anosov flow $\phi^t$ on a closed oriented $3$-manifold, a positive/negative horizontal surgery curve $c$, and a positive/negative integer $n$, respectively, so that $\phi^t_{\frac{1}{n}}(c)$ is not almost equivalent to $\phi^t$.
\end{conj}

One possible approach to construct such an example is as follows: Start with a non-transitive pseudo-Anosov flow $\phi^t$ on a closed oriented $3$-manifold $M$, for which there is a transverse torus $T$ separating $M$ into two atoroidal $3$-manifolds $M_1$ and $M_2$. Such examples were constructed in \cite{FW80}, and the techniques in \cite{BBY17} allow for construction of more complicated examples.

One then attempts to find a horizontal surgery curve $c$ that passes through $T$ essentially, i.e. cannot be isotoped to lie in $M_1$ or $M_2$. For large enough $n$, one expects $M_{\frac{1}{n}}(c)$ to be atoroidal, which would imply that $\phi^t_{\frac{1}{n}}(c)$ is transitive by \cite[Proposition 2.7]{Mos92a}. In particular, $\phi^t_{\frac{1}{n}}(c)$ cannot be almost equivalent to $\phi^t$.

\subsection{Relation with veering triangulations} \label{subsec:vt}

A veering triangulation is an ideal triangulation on a $3$-manifold satisfying certain combinatorial conditions. We will not go into the full definition here, and will instead refer the reader to \cite{SS19}. For the purposes of this discussion, one only needs to know that one of the combinatorial conditions is that the faces of the tetrahedra are cooriented.

It has been discovered recently that there is a correspondence between veering triangulations and pseudo-Anosov flows. 
Without going into too much details, the key facts are that 
\begin{enumerate}
    \item Given a veering triangulation $\Delta$ on an oriented $3$-manifold $N$, there exists a transitive pseudo-Anosov flow $\phi^t(\Delta, M)$ on a closed oriented $3$-manifold $M$.
    \item Given a transitive pseudo-Anosov flow $\phi^t$ on a closed oriented $3$-manifold $M$, there exists a collection of closed orbits $\mathcal{C}$ such that $M \backslash \mathcal{C}$ admits a veering triangulation $\Delta = \Delta(\phi^t, \mathcal{C})$ whose faces are positively transverse to the orbits of $\phi^t$.
    \item In the notation above, we have
    \begin{enumerate}
        \item $\phi^t(\Delta, M)$ admits a collection of closed orbits $\mathcal{C}$ such that $\Delta = \Delta(\phi^t(\Delta, M), \mathcal{C})$, and
        \item $\phi^t=\phi^t(\Delta(\phi^t, \mathcal{C}), M)$.
    \end{enumerate}
\end{enumerate}
We refer to \cite[Chapter 2]{Tsa23a} for a much more thorough exposition, as well as refer to \cite{SS20}, \cite{SS19}, \cite{SS21}, \cite{SS23}, \cite{SSpart5}, and \cite{LMT21} for proofs.

In \cite{Tsa22a}, the author introduced the \textbf{horizontal surgery operation} on veering triangulations. A \textbf{horizontal surgery curve} for a veering triangulation $\Delta$ is a curve on the stable branched surface of $\Delta$ lying away from the vertices and satisfying various combinatorial conditions. Here the \textbf{stable branched surface} of a veering triangulations is a $2$-complex dual to the triangulation. In particular a curve lying away from the vertices on the stable branched surface can be thought of as a curve lying on the faces of $\Delta$. Performing horizontal surgery along a horizontal surgery curve $c$ for $\Delta$ (with coefficient $\frac{1}{n}$ where $n$ is of the appropriate sign) involves slitting $\Delta$ along the annulus given by the union of faces that $c$ passes through, then inserting a suitably triangulated solid torus within. We denote the resulting veering triangulation as $\Delta_{\frac{1}{n}}(c)$.

We conjecture that horizontal Goodman surgery on pseudo-Anosov flows corresponds exactly to horizontal surgery on veering triangulations. We formulate this more precisely as follows.

\begin{conj} \label{conj:vthorsur}
\leavevmode
\begin{enumerate}
    \item Let $\Delta$ be a veering triangulation on an oriented $3$-manifold $N$.
    \begin{enumerate}
        \item Every horizontal surgery curve $c$ for $\Delta$ is isotopic to a horizontal surgery curve for $\phi^t(\Delta, M)$.
        \item $\phi^t(\Delta_{\frac{1}{n}}(c), M_{\frac{1}{n}}(c))$ is orbit equivalent to $\phi^t(\Delta, M)_{\frac{1}{n}}(c)$.
    \end{enumerate}
    \item Let $\phi^t$ be a transitive pseudo-Anosov flow on a closed oriented $3$-manifold $M$.  
    \begin{enumerate}
        \item Every horizontal surgery curve $c$ for $\phi^t$ is isotopic to a horizontal surgery curve for $\Delta(\phi^t, \mathcal{C})$, for some collection of closed orbits $\mathcal{C}$,
        \item $\Delta(\phi^t_{\frac{1}{n}}(c), \mathcal{C})$ is isomorphic to $\Delta(\phi^t, \mathcal{C})_{\frac{1}{n}}(c)$.
    \end{enumerate}
\end{enumerate}
\end{conj}

A strategy for showing (1a) is to consider a boundary component of the annulus given by the union of faces that $c$ passes through. This is a union of edges isotopic to $c$, but is only a piecewise smooth curve in general. Nevertheless, one might be able to argue that this is a horizontal surgery curve \`a la \Cref{defn:piecewisehorsurcurve} if the edges of the veering triangulation are assumed to be in a nice position beforehand.

We remark that in the planned sequel \cite{Tsa24}, we will show a special case of (1), namely when $\Delta$ is layered.

One motivation for \Cref{conj:vthorsur} is that if it were true, then one would be able to combine it with \Cref{thm:horsuralmostequiv} to derive almost equivalences between many flows using the tool of veering triangulations. In particular, it would be possible to do a systematic probing of \Cref{conj:Ghys} by working through the veering triangulation census \cite{VeeringCensus}.

\subsection{Instantaneous metrics}

We remind the reader that the definition of horizontal surgery curves (\Cref{defn:horsurcurve} and \Cref{defn:piecewisehorsur}) and the operation of horizontal Goodman surgery (\Cref{constr:sur}) do not depend on a choice of an instantaneous metric. In most of our proofs, we only pick an instantaneous metric in order to gain quantitative control.

With that being said, \Cref{lemma:constantslope} is an exception where the choice of an instantaneous metric defines a criterion for identifying horizontal surgery curves. In order to utilize \Cref{lemma:constantslope} better, one must gain a deeper understanding as to how instantaneous metrics can be constructed.

In particular, we are interested in whether one can extend a partially defined metric into an instantaneous metric. We formulate this more precisely as follows.

\begin{quest} \label{quest:relinstantmetric}
Let $\phi^t$ be a pseudo-Anosov flow on a closed oriented $3$-manifold $M$. Let $\mathring{M}$ be an open submanifold of $M$ containing $\sing(\phi^t)$. Suppose we are given a path metric $\mathring{d}$ on $\mathring{M}$ induced from a Riemannian metric $\mathring{g}$ away from $\sing(\phi^t)$, such that:
$$||d\phi^t(v)||_{\mathring{g}} < \lambda^{-t} ||v||_{\mathring{g}}$$
for every $v \in E^s|_x, t>0$ such that $x, \phi^t(x) \in \mathring{M}$, and 
$$||d\phi^t(v)||_{\mathring{g}} < \lambda^t ||v||_{\mathring{g}}$$
for every $v \in E^u|_x, t<0$ such that $x, \phi^t(x) \in \mathring{M}$, for some $\lambda>1$.

Then is it always possible to extend $\mathring{d}$ to an instantaneous metric $d$ on $M$?
\end{quest}

If the answer to \Cref{quest:relinstantmetric} is `yes', then a positive answer to the following question seems likely.

\begin{quest} \label{quest:horsur=constantslope}
Is every horizontal surgery curve of constant slope with respect to some instantaneous metric?
\end{quest}

We are also interested in the following question, in which a positive answer to \Cref{quest:relinstantmetric} might prove useful in tackling.

\begin{quest} \label{quest:horsurcurveisotopic}
Are two positive horizontal surgery curves that are isotopic through positive curves always isotopic through positive horizontal surgery curves?
\end{quest}

\subsection{Relation with Salmoiraghi's surgeries} \label{subsec:Sal}

In \Cref{prop:Sal22}, we showed that at least for some choice of construction data, Salmoiraghi's surgery in \cite{Sal22} is equivalent to the horizontal Goodman surgery operation discussed in this paper. 
As noted there, we expect this equivalence to hold more generally whenever Salmoiraghi's surgery gives an Anosov flow.

We also point out that in \cite{Sal21}, Salmoiraghi introduces another surgery operation along Legendrian-transverse knots. One key feature of this surgery operation is that it is performed along annuli that are \emph{tangent} to the flow.
What Salmoiraghi shows in \cite[Theorem 7.1]{Sal21}, however, is that the surgered flow can also be obtained as cutting along a transverse annulus and regluing by some Dehn twist, even though here the exact conditions for the annulus and the gluing map are less clear.

In \cite{Sal21}, Salmoiraghi gives conditions as to when this surgery operation produces an Anosov flow. We speculate that, as in the surgery operation in \cite{Sal22}, this surgery operation is equivalent to horizontal Goodman surgery in those cases, at least for some choice of construction data.

It is important to note that, in both papers, Salmoiraghi deals with the category of flows determined by bicontact structures, which can be showed to be exactly those flows that are \emph{projectively} Anosov --- a class that is more general than Anosov flows. 

Finally, we are interested in the following question.

\begin{quest} \label{quest:legendriantransverse}
Let $\phi^t$ be an Anosov flow on a closed oriented $3$-manifold. Is every horizontal surgery curve for $\phi^t$ a Legendrian-transverse curve with respect to some bicontact structure $(\xi_+,\xi_-)$ that determines $\phi^t$? 
\end{quest}

\Cref{quest:legendriantransverse} seems intimately related with \Cref{quest:horsur=constantslope}. If the answer to the latter is `yes' then we expect a positive answer to the former as well.

\appendix

\section{Structural stability of pseudo-Anosov flows} \label{sec:structuralstability}

In this appendix, we prove the structural stability result \Cref{thm:structuralstability} that we had to use in \Cref{subsec:horsurinvar}. 

We first recall a few constructions. Let $\phi^t$ be a pseudo-Anosov flow on a closed $3$-manifold $M$. Let $\nu$ be a neighborhood of $\sing(\phi^t)$. We can perform the \textbf{double DA operation} on $\sing(\phi^t)$ by modifying $\dot{\phi}$ in $\nu$. This modifies $\phi^t$ into an Axiom A flow $\psi^t$. The stable/unstable lamination of (the chain recurrent set of) $\psi^t$ is obtained by blowing air into the singular leaves of the stable/unstable foliation of $\phi^t$. In particular $\psi^t$ satisfies the strong transversality condition. We refer to \cite[Section 1]{Mos92a} for more details.

In \Cref{fig:doubleda} top we indicate how the dynamics is changed while in \Cref{fig:doubleda} bottom we indicate how the stable/unstable laminations are changed near $\sing(\phi^t)$, under the double DA operation.

\begin{figure}
    \centering
    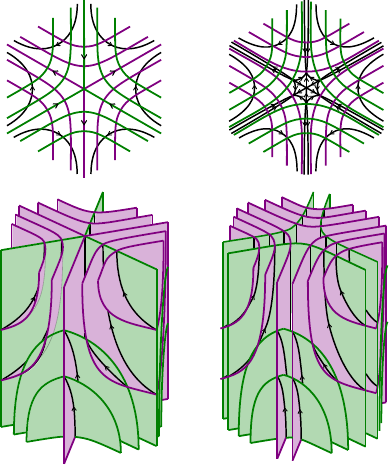
    \caption{A local picture of the double DA operation. Top: The change in dynamics before and after the double DA operation. Bottom: The change in stable/unstable laminations before and after the double DA operation.}
    \label{fig:doubleda}
\end{figure}

Meanwhile, recall the orbit space $\mathcal{O}$ of $\phi^t$ from \Cref{defn:orbitspace}. We can similarly define the orbit space $\mathcal{Q}$ of $\psi^t$ to be the space of orbits of the lifted flow $\widetilde{\psi}^t$ in the universal cover. As for $\mathcal{O}$, $\mathcal{Q}$ inherits two laminations $\mathcal{Q}^s$ and $\mathcal{Q}^u$ from the stable and unstable laminations of $\psi^t$ respectively. Notice that $(\mathcal{Q}, \mathcal{Q}^s, \mathcal{Q}^u)$ can be obtained from $(\mathcal{O}, \mathcal{O}^s, \mathcal{O}^u)$ by blowing air into the singular leaves of $\mathcal{O}^{s/u}$. Conversely, $(\mathcal{O}, \mathcal{O}^s, \mathcal{O}^u)$ can be obtained from $(\mathcal{Q}, \mathcal{Q}^s, \mathcal{Q}^u)$ by collapsing the complementary regions of $\mathcal{Q}^{s/u}$ along leaves of $\mathcal{Q}^{u/s}$ respectively. 

The idea of the proof of \Cref{thm:structuralstability} is that by performing the double DA operation, we switch categories from pseudo-Anosov flows to Axiom A flows satisfying the strong transversality condition, the latter of which is known to satisfy structural stability. We then transfer the orbit equivalence in the category of Axiom A flows back to pseudo-Anosov flows using orbit spaces.

\begin{proof}[Proof of \Cref{thm:structuralstability}]
Suppose $\phi^t$ is a pseudo-Anosov flow on a closed $3$-manifold. We fix a modification of $\phi^t$ in $\nu$ that modifies it into an Axiom A flow $\psi^t$ satisfying the strong transversality condition. For pseudo-Anosov flows $\overline{\phi}^t$ as in the statement of the theorem, we can perform the same modification of the generating vector field on $\nu$ to get a Axiom A flow $\overline{\psi}^t$ satisfying the strong transversality condition. Moreover $\dot{\psi}$ and $\dot{\overline{\psi}}$ are $\epsilon$ close in the $C^1$-topology. Hence by \cite[Theorem 1.1]{Rob75}, $\psi^t$ and $\overline{\psi}^t$ are orbit equivalent. Furthermore, by making $\epsilon$ small, we can ensure that the orbit equivalence is isotopic to identity. See the remarks under \cite[Theorem 1.1]{Rob75}.

Such an orbit equivalence induces a $\pi_1(M)$-equivariant homeomorphism between the orbit spaces $\mathcal{Q}$ and $\overline{\mathcal{Q}}$ of $\psi^t$ and $\overline{\psi}^t$ respectively, preserving the stable/unstable laminations. By collapsing complementary regions, we obtain a $\pi_1(M)$-equivariant homeomorphism between the orbit spaces $\mathcal{O}$ and $\overline{\mathcal{O}}$ of $\phi^t$ and $\overline{\phi}^t$ respectively, preserving the stable/unstable foliations. By \cite[Théorème 3.4]{Bar95}, this implies that $\phi^t$ and $\overline{\phi}^t$ are orbit equivalent through a homeomorphism $f$ whose action on $\pi_1(M)$ is identity. 

It remains to conclude that such a $f$ is isotopic to identity. This follows from work of Waldhausen \cite{Wal68} (when $M$ is Haken), Scott \cite{Sco85} and Boileau-Otal \cite{BO91} (when $M$ is Seifert-fibered and non-Haken), and Gabai-Meyerhoff-Thurston \cite{GMT03} (when $M$ is hyperbolic), as well as the geometrization theorem proved by Perelman.
\end{proof}

We speculate that it should be possible to supply a proof of \Cref{thm:structuralstability} by modifying the classical analytic proofs for Anosov flows, found in, for example \cite{KM73}. The main difficulty is in imposing the correct analytic conditions near the singular orbits. Such a proof should strengthen \Cref{thm:structuralstability} by allowing $\dot{\phi}$ and $\dot{\overline{\phi}}$ to differ on an open set that approaches the singular orbits, provided that the difference goes to $0$ rapidly enough. However, since we do not need such a strong version of structural stability in this paper, we have not attempted to carry out this approach.

\bibliographystyle{alpha}

\bibliography{bib.bib}

\end{document}